\documentclass[reqno]{amsart}
\usepackage{mathtools}
\usepackage{tkz-euclide}
\usepackage{xfrac}
\usepackage{hyperref}
\usepackage{amsmath}  
\usepackage{amsfonts}
\usepackage{amssymb} 
\usepackage{mathrsfs} 
\usepackage{graphicx}  
\usepackage{bm}
\usepackage{tikz} 
\usetikzlibrary{svg.path} 
\usepackage{amssymb}
\usepackage{pgfplots}
\usepackage{amsmath}
\usepackage[toc]{appendix}
\usetikzlibrary{calc} 
\usepackage{xcolor}  
\usetikzlibrary{arrows.meta} 
\usepackage{amsthm} 
\usepackage{float}
\usepackage{enumitem}
\usepackage{comment}

\usepackage[english]{babel}
\usepackage[font=footnotesize, labelfont=bf]{ caption}
\usepackage{ esint }
\usepackage{stackengine,scalerel,graphicx}
\stackMath

\definecolor{trueblue}{rgb}{0.0, 0.45, 0.81}
\definecolor{truegreen}{rgb}{0.13, 0.55, 0.13}

\newcommand{\EEE}{\color{black}}

\newcommand{\eps}{\varepsilon}

\theoremstyle{plain}
\begingroup
\newtheorem{theorem}{Theorem}[section]
\newtheorem{lemma}[theorem]{Lemma}
\newtheorem{example}[theorem]{Example}
\newtheorem{remark}[theorem]{Remark}
\newtheorem{proposition}[theorem]{Proposition}
\newtheorem{corollary}[theorem]{Corollary}
\endgroup

\usepackage{todonotes}

\theoremstyle{definition}
\begingroup
\newtheorem{definition}[theorem]{Definition}
\endgroup

\renewcommand{\tilde}{\widetilde}

\DeclareMathOperator{\argmin}{argmin}

\numberwithin{equation}{section}
\newcommand{\N}{\mathbb{N}}
\newcommand{\Z}{\mathbb{Z}}
\newcommand{\R}{\mathbb{R}}

\newcommand{\defas}{:=}

\usepackage[normalem]{ulem}

\begin{document}
\title[Stochastic homogenisation for   asymptotically piecewise rigid functions]{Stochastic homogenisation for functionals defined on asymptotically piecewise rigid functions}

\author[A. Donnarumma]{Antonio Flavio Donnarumma}
\address[Antonio Flavio Donnarumma]{Department of Mathematics, Friedrich-Alexander Universit\"at Erlangen-N\"urnberg. Cauerstr.~11,
D-91058 Erlangen, Germany}
\email{antonio.flavio.donnarumma@fau.de}

\author[M. Friedrich]{Manuel Friedrich} 
\address[Manuel Friedrich]{Department of Mathematics, Friedrich-Alexander Universit\"at Erlangen-N\"urnberg. Cauerstr.~11,
D-91058 Erlangen, Germany, \& Mathematics M\"{u}nster,  
University of M\"{u}nster, Einsteinstr.~62, D-48149 M\"{u}nster, Germany}
\email{manuel.friedrich@fau.de}

\EEE
\begin{abstract}  
We study  stochastic homogenisation of free-discontinuity surface functionals defined on piecewise rigid functions which arise in the study of fracture in brittle materials. In particular, under standard assumptions on the density, we show that there exists a $\Gamma$-limit almost surely and that it can be represented by a surface integral. In addition, the effective density can be characterised via a suitable cell formula and is deterministic under an  ergodicity  assumption. We also show via $\Gamma$-convergence that the  homogenised functional defined on piecewise rigid functions can  be recovered from a Griffith-type model by passing to the limit of vanishing elastic deformations. 
\end{abstract}

\maketitle

\section{Introduction}
 The last two decades have witnessed a tremendous interest and progress in the analysis of free-discontinuity functionals  of the form  
\begin{align}\label{mainfunctional}
E(u) = \mathcal{E}^{\rm bulk}(u) +   \mathcal{E}^{\rm surface}(u),  
\end{align}
featuring bulk and surface energies given by 
\begin{align}\label{mainfunctional2}
\mathcal{E}^{\rm bulk}(u) =    \int_{U} W(x,\nabla u(x)) \, {\rm d}x, \quad \quad \quad    \mathcal{E}^{\rm surface}(u) = \int_{J_u \cap U} f(x,[u](x),\nu_u(x)) \, {\rm d}\mathcal{H}^{d-1}(x). 
\end{align}
 Such functionals are  prototypes for  many variational models of fracture  \cite{FrancfortMarigo}.  Among the vast body of literature on crack growth, we mention here only  some of the contributions for small strains \cite{Chambolle:2003, Francfort-Larsen:2003,  FriSol16, giapon04a} and finite strains \cite{dalfratoa04,  DalMaso-Lazzaroni:2010}. In \eqref{mainfunctional}, the first part $ \mathcal{E}^{\rm bulk}$  accounts for elastic bulk terms for the unfractured region of the body with reference configuration $U \subset \R^d$, where   $\nabla u$ denotes the deformation gradient. The second contribution $\mathcal{E}^{\rm surface}$, instead, assigns energy contributions related to the crack surface  $J_u$ comparable to the $(d-1)$-dimensional Hausdorff measure $\mathcal{H}^{d-1}(J_u)$ of the crack. In the simplest formulation,  $f$ is constant, representing the \emph{fracture toughness}, given by Griffith's criterion of fracture initiation \cite{griffith1920phenomena}. Densities $f$ depending explicitly on the crack opening $[u]$ allow for fracture  problems of cohesive type, whereas the  presence of the normal $\nu_u$ to the jump set $J_u$ is relevant in the modeling of anisotropic surface energies in brittle crystals \cite{bragel02, frie14}. Eventually, dependence on the material point $x$ accounts for inhomogeneities.

In the finite-strain setting, problems of the form \eqref{mainfunctional} are usually  formulated in the space of the   \emph{special functions of bounded variation} (${SBV}$) (see \cite{DeGiorgi}) or its  generalisation  $GSBV$ \cite{ambrosio2000fbv, dalfratoa04}. The elastic energy density $W$  complies with the principle of frame indifference, growing quadratically around the set of rotations $SO(d)$, i.e.\ $W(x,F) \sim {\rm dist}^2(F,SO(d))$ for $F \in \R^{d\times d}$ close to $SO(d)$. Linearisation of $W$ leads to an invariance under skew-symmetric matrices $\R^{d \times d}_{\rm skew}$. Therefore, for  energies of the form \eqref{mainfunctional} in linearised elasticity, the density $W$ only depends on the symmetric part of the gradient $\frac{1}{2}(\nabla u^T + \nabla u)$. Correspondingly, a more intricate formulation in the space of   \emph{generalised  functions of bounded deformation} ($GSBD$) \cite{Dal11} is needed.

In this paper, we are interested in  the asymptotic analysis of sequences of random free-discontinuity problems $(E^\omega_\eps)_\eps$ of the form \eqref{mainfunctional}, where the parameter $\eps$ may  represent  the size of a microstructure and the corresponding densities $W^\omega_\eps$ and $f^\omega_\eps$ are supposed to be stationary (or even ergodic) with respect to the random variable $\omega$. Here, stationarity can be interpreted as a ``periodicity in law''  generalising the periodicity hypothesis in classical, deterministic homogenisation whereas the ergodicity assumption allows us to relate the average in expectation with the large scale space average.  

 The understanding of such problems in the finite-strain setting is well-developed by means of $\Gamma$-convergence  \cite{Braides:02, DalMaso:93}, where limiting effective  energies of the homogenised material can be derived in terms of \emph{homogenisation formulas} independent of the material point $x$.  After first results for the  deterministic case of periodic   homogenisation  \cite{BraDefVit96},  later    generalised  to the  case  without any periodicity assumptions  \cite{giapon04a}, the problem was recently  addressed   by {\sc Cagnetti, Dal Maso, Scardia, and Zeppieri}  \cite{cagnetti2018gammaconvergence, cagnetti2017stochastic} in a very general framework of   stochastic homogenisation. In contrast, the study of the   linearised  counterpart is still in its infancy. Results appear to be limited to \cite{FriPerSol20a} which addresses the case of deterministic homogenisation in dimension two. Besides the mentioned results in the context of fracture, there is an abundant literature dealing with  homogenisation   problems, both in a discrete and continuous setting, and with different hypotheses on the densities. We refer the reader to some examples in deterministic 
 \cite{hom2, BarchiesiFocardi, hom1, Bouchitte, BraPia19a, hom3} and stochastic  \cite{alicicglo09, BacRuf21, atwoscale,  brasig09,  global, Pellet, RufZeppieri, Zemas}  homogenisation.   
 
%
%
%
%
%
%
%

In this work, we focus on a different regime, namely the one of \emph{asympotically rigid solids}, corresponding to a scaling of the elastic energy density in \eqref{mainfunctional2} of the form
\begin{align}\label{eq: scalingW}
W^\omega_\eps(x,\nabla u(x)) \ge \frac{1}{\delta_\eps^2} {\rm dist}^2(\nabla u(x), SO(d)) \quad \quad \text{for a sequence $(\delta_\eps)_\eps$ with $\delta_\eps \to 0$}. 
\end{align}
In fact, in the limit $\delta_\eps \to 0$, configurations with finite energy are expected to satisfy the constraint $\nabla u \in SO(d)$ a.e.\ in $U$. A   nontrivial \emph{piecewise rigidity result} by {\sc Chambolle, Giacomini, and Ponsiglione} \cite{Chambolle2006PIECEWISERA},  generalising  the classical Liouville theorem for smooth functions, then states that the functions are \emph{piecewise rigid} in the sense that there exists a collection of an at most countable family of  different components each of which subject to a different rigid motion. More precisely, $u$ can be written
\begin{equation}\label{pwr}
u(x)=\sum_{j\in \mathbb{N}}(M_jx+b_j)\chi_{P_j}(x) 
\end{equation}
where $( M_j)_j \subset SO(d)$, $(b_j)_j \subset \mathbb{R}^d$, and $(P_j)_{j}$ is a \emph{Caccioppoli partition} of $U$, see \cite[Section~4]{ambrosio2000fbv}. A passage to asymptotically piecewise rigid functions has been addressed for constant surface densities in \cite[Corollary 2.9]{friedrich2017derivation} and for phase-field approximations in \cite{CicFocZep21}.  The analogy of the scaling \eqref{eq: scalingW} for a linear elastic density would lead to a model with the constraint $\frac{1}{2}(\nabla u^T + \nabla u) =0$ a.e.\ in $U$.  As shown in \cite{Chambolle2006PIECEWISERA, friedrich2018piecewise}, this condition characterizes the space of \emph{piecewise infinitesimal rigid functions}, i.e.\ functions of the form \eqref{pwr}  with the matrices $M_j$ in $\R^{d \times d}_{\rm skew}$ instead of $SO(d)$. In the following, the spaces will be denoted by $PR_L(U)$, for  $L=SO(d)$ and $L=\mathbb{R}_{\rm skew}^{d \times d}$, respectively.

The goal of this article is to extend the asymptotic analysis for $ \delta_\eps \to 0$  to the case of stochastic   homogenisation. More precisely, our scope is twofold. We will first study the effective random $\Gamma$-limit of pure surface energies  $\mathcal{E}^{\rm surface}$, see \eqref{mainfunctional2},  under the constraint that  the configurations are piecewise rigid, both in the nonlinear  case $L=SO(d)$  and the linear case $L=\mathbb{R}_{\rm skew}^{d \times d}$. On the one hand, this extends the deterministic results on $\Gamma$-convergence and integral representation for $PR_L$ developed in \cite{Friedrich_2020} to a stochastic setting, and may in turn contribute to the understanding of random interfacial energies for  brittle materials showing locally rigid behaviour. On the other hand, for the case $L=\mathbb{R}_{\rm skew}^{d \times d}$, the analysis will be a key ingredient for the forthcoming study of random free-discontinuity problems on $GSBD$ featuring both elastic and surface contributions \cite{Donnarumma}. Our second aim consists in the investigation of energies $(E^\omega_\eps)_\eps$ of the form \eqref{mainfunctional} in the regime of asymptotically rigid motions. We will also combine this perspective with a simultaneous passage to a  linearised  formulation in terms of rescaled displacement fields. Whereas the interplay of  homogenisation and  linearisation  is well understood in a purely elastic setting \cite{Neukamm2, Jesenko, Neukamm1}, to the best of our knowledge this issue has not been addressed yet in the realm of free-discontinuity problems.  We now describe our setting and the results in more detail.

 \textbf{Piecwise rigid funtions:} In the first part, we consider random surface functionals of the form 
\begin{equation}
\label{energypr}
\mathcal{E}^L_{\varepsilon}[\omega](u)= 
\int_{J_u\cap U}f\Big(\omega,\frac{x}{\varepsilon},[u](x),\nu_u(x)\Big)\,\mathrm{d}\mathcal{H}^{d-1}(x) 
\end{equation}
for $u \in PR_L(U)$, with $L=SO(d)$ or  $L=\mathbb{R}_{\rm skew}^{d \times d}$, and $U \subset \R^d$ for $d \in \lbrace 2,3 \rbrace$. Here,  the random environment is modeled by a probability space $(\Omega,\mathcal{I},\mathbb{P})$ and a random realization is denoted by $\omega \in \Omega$. We suppose that $f$ is a stationary random surface density with respect to a  group of $\mathbb{P}$-preserving transformations, see Definitions \ref{def: pre}--\ref{stationarity} below for details. In Theorem \ref{gammaconv}, we show that the sequence in \eqref{energypr} almost surely $\Gamma$-converges   to the random   homogenised functional
\begin{align}\label{qeq: limit prob}
\mathcal{E}^L_{\mathrm{hom}}[\omega](u)= \int_{J_u \cap A}f^L_{\mathrm{hom}}(\omega,[u](x),\nu_u(x))\,\mathrm{d}\mathcal{H}^{d-1}(x), 
\end{align}
and we provide a characterization of  $f_{\rm hom}$ in terms of a \emph{homogenisation formula}, see Theorem \ref{HomogenizationFormula}. More precisely, we prove that 
\begin{equation}
\label{eqhom-new}
f^L_{\mathrm{hom}}(\omega,\zeta,\nu)=\lim\limits_{t \to \infty}\frac{1 }{{t}^{d-1}}  \inf_{v}\int_{J_v\cap Q^\nu_t(tx)}f\Big(\omega,x,[v](x),\nu_v(x)\Big)\,\mathrm{d}\mathcal{H}^{d-1}(x).
\end{equation}
Here, for $L=\mathbb{R}_{\mathrm{skew}}^{d \times d}$,   the infimum is taken among all piecewise rigid functions $PR_L$  defined on the oriented cube $Q^\nu_t(tx)$ (see \eqref{eq: Qnot}) which attain piecewise constant boundary conditions related to $\zeta$, see \eqref{Lw} below for details. For $L=SO(d)$ instead, the infimum is taken on the smaller set of piecewise constant functions.   We emphasise  that $f_{\rm hom}$ is  $x$-independent as a consequence of stationarity. Under an additional ergodicity assumption on the transformation group, see Definition \ref{def: pre}, we see that  $\mathcal{E}^L_{\mathrm{hom}}$  is deterministic, i.e.\ does not depend on $\omega$.  Subsequently, we complement the $\Gamma$-convergence result with a compactness property which ensures the convergence of infima,  see Corollary \ref{convminima}. Convergence of corresponding  minimisers  is a delicate issue in free-discontinuity problems and we obtain a result in this direction for almost  minimisers   up to an arbitrarily small error term, see Corollary \ref{convergencealmostminimisers} and Remark~\ref{remarkcompactness} for details.

Our proof follows the strategy devised in  \cite{cagnetti2017stochastic} where, as part of the problem, surface energies defined on piecewise constant functions are studied. As in the seminal work \cite{maso1984nonlinear}, the procedure consists in two steps, namely a purely deterministic step and a stochastic one.

For fixed   random outcome $\omega \in \Omega$, under the assumption that the  homogenised  density $f^{L}_{\rm hom}$ given in \eqref{eqhom-new}   exists, the deterministic step consists in showing that the $\Gamma$-limit exists and takes the form  $\mathcal{E}^{L}_{\mathrm{hom}}[\omega]$, as given in \eqref{qeq: limit prob}. Here, we resort to the results in \cite{Friedrich_2020}, where $\Gamma$-convergence and integral representation for functionals defined on $PR_L$ have been investigated  by means of the localisation  technique for $\Gamma$-convergence, see \cite{DalMaso:93},   and the global method of relaxation \cite{BFLM}. To ensure convergence of infima of certain Dirichlet problems, delicate truncation methods are employed, valid in dimensions $d=2,3$, see \cite[Section 7]{Friedrich_2020}. This is the reason why in the present work we focus on the physically relevant dimensions $d=2,3$.

Then, the stochastic step consists in showing the assumption that the   homogenised  density $f^L_{\rm hom}$   exists almost surely and is independent
of $x$. In this part, we follow the proof in  \cite{cagnetti2017stochastic} which fundamentally relies on applying the Subadditive Ergodic Theorem  by {\sc Akcoglou and Krengel}  \cite{krengel2011ergodic}. Here, the authors construct a suitable $(d-1)$-dimensional subbadditive stochastic process which takes care of the mismatch of dimension between the ambient space and the dimension of the discontinuity set. In Section \ref{sec 4}, we sketch the main steps for convenience of the reader, yet we refer to  \cite[Introduction, Sections 5--6]{cagnetti2017stochastic} for more details. Although large parts of the proof work the same for piecewise constant functions and the space $PR_L$, up to natural adaptations, this is not true for the measurability in $\omega$ of the  minimisation  problem in \eqref{eqhom-new}, crucial in the definition of the subadditive stochastic process. Here, we need to employ suitable truncation and compactness results for piecewise rigid functions. For further details regarding this  issue we refer to Remark~\ref{problemmeasurability}.

\textbf{Asymptotically piecewise rigid functions:} 
The second part of our work consists in studying  functionals including also elastic energies. To this end,  we consider a sequence of energies of the form
\begin{align*}
{\mathcal{F}}_{\varepsilon,\delta}[\omega](y,A)= 
 \int_{A} \frac{1}{\delta^{2}} \Big(W\big(\omega,\frac{x}{\varepsilon}, \nabla y(x)\big)+    c_\delta  |\nabla^2 y(x)|^2\, \mathrm{d}x \Big) + \int_{J_y \cap A}f\Big(\omega,\frac{x}{\varepsilon},[y](x),\nu_y(x)\Big)\, \mathrm{d}\mathcal{H}^{d-1}(x),
\end{align*}
where $\eps$ again stands for the size of the microstructure, the surface part is as in \eqref{energypr}, and the additional bulk part depends on a    random, inhomogeneous,  and frame indifferent stored energy density $W$. We suppose quadratic growth of $W$ around $SO(d)$, reflecting the scaling \eqref{eq: scalingW} in terms of a small parameter $\delta$, which could be interpreted as the typical size of the elastic strain.  The model is an extension of the one in  \cite{friedrich2019griffith}, where 
an asymptotically small second-gradient term has been added   to a classical Griffith-type functional, i.e.  $c_\delta \to 0$ as $\delta \to 0$. This corresponds to a model for \emph{nonsimple materials}, see   \cite{Toupin1962ElasticMW} for a seminal  work  in elasticity theory. Such a  term enhances   the rigidity properties of the nonlinear model and is currently unavoidable to pass to small-strain settings $\delta \to 0$ in dimension $d \ge 3$. We refer to    \cite{friedrich2019griffith} for more details, also regarding the underlying functions space  $GSBV^2_2$, see \eqref{eq: GSBV22}, consisting of the mappings for which both the function itself and its derivative are in the class of generalised  special functions of bounded variation.  In the present contribution, we consider an extension of the nonlinear model in  \cite{friedrich2019griffith} to the case of random bulk and surfaces energies. 

In our main result, Theorem \ref{brittlemateriallimit}, we show that the sequence ${\mathcal{F}}_{\varepsilon,\delta}$ almost surely $\Gamma$-converges to the functional  in \eqref{qeq: limit prob} on $PR_{SO(d)}$ as $\eps,\delta \to 0$. This means that the  Griffith energies ${\mathcal{F}}_{\varepsilon,\delta}$ and the surface energies $\mathcal{E}_{\varepsilon}$  in \eqref{energypr} (defined only on $PR_{SO(d)}$) are \emph{equivalent by $\Gamma$-convergence} in the language of \cite{truski}. The main ingredient to show this equivalence is an approximation result   of functions with small elastic energy by piecewise rigid functions in $PR_{SO(d)}$, see Proposition  \ref{piecewiserigidapproximation}. The construction combines the rigidity result of \cite{friedrich2019griffith} based on the second-order   regularisation, with a  piecewise Poincar\'e inequality, see \cite[Theorem 2.3]{friedrich2018piecewise},  and further geometric arguments for  partitioning of sets. Finally, let us mention that the exact form of $W$, in particular its inhomogeneous and random nature, does not affect the  homogenised  surface energy, and the dependence on $x$ and $\omega$ is  assumed only for the  sake of generality.

Our last result combines the above limit  $\eps,\delta  \to 0$  with a simultaneous passage to a  linearised model defined on piecewise infinitesimal rigid motions.  To this end, as in the  linearisation results for Griffith energies \cite{friedrich2017derivation, friedrich2019griffith}, the deformation $y$ is written in terms of a rescaled displacement field  $u = \frac{1}{\delta^\alpha}(y - {\rm id})$ for some $\alpha >0$, where ${\rm id}$ denotes the identity mapping.  Whereas the choice $\alpha=1$ asymptotically leads to a model comprising bulk and surface terms \cite{friedrich2017derivation, friedrich2019griffith}, we choose  $\alpha \in (0,1)$ in the present work to obtain a pure surface integral in the effective limit. In fact, in Theorem~\ref{brittlemateriallimit2} we  show that the energies ${\mathcal{F}}_{\varepsilon,\delta}$, expressed in terms of $u$, almost surely $\Gamma$-converge  to $\mathcal{E}_{\mathrm{hom}}$ in \eqref{qeq: limit prob} on $PR_{\R^{d \times d}_{\rm skew}}$ as $\eps,\delta \to 0$. We mention that, strictly speaking, this result is proven under a suitable rescaling of the surface energy, see \eqref{completenergyXXX} for details. Moreover, for technical reasons related to frame indifference,  we work under the constraint $\Vert\nabla u\Vert_\infty \le \delta^{-\alpha/4} $ which ensures that  deformation gradients are close to the identity and not to other rotations in $SO(d)$. We refer to \eqref{for sure needed} and the discussion before \eqref{constraints}. 

The proof of the $\Gamma$-liminf inequality is again based on approximation by piecewise (infinitesimal) rigid motions, see Proposition  \ref{piecewiserigidapproximation}. The construction of recovery sequences is more subtle as abstract recovery sequences   provided by Theorem \ref{gammaconv} may have large gradients in $\R^{d \times d}_{\rm skew}$ incompatible with the elastic energy contributions in ${\mathcal{F}}_{\varepsilon,\delta}$. As a remedy, we provide a more explicit construction  under the hypothesis of periodic surface integrands,   based on the density of finite polyhedral partitions \cite{BraConGar16} and the solution of the cell problem \eqref{eqhom-new}. At this point, we need to assume a compatibility condition of the form  $\eps \delta^{-\alpha/4} \to \infty$ as $\eps,\delta \to 0$ which means that that the scaling of the microstructure is not too small compared to the strain. This condition is not of technical nature but necessary as we indeed provide an example that, without such an assumption,  the equivalence by $\Gamma$-convergence of ${\mathcal{F}}_{\varepsilon,\delta}$ and $\mathcal{E}_{\varepsilon}$ (for $PR_{\R^{d\times d}_{\rm skew}}$) can fail. We provide an explicit construction, see Example~\ref{main-example}, which relies on a surface density which is not \emph{$BD$-elliptic} in the sense of \cite{FriPerSol20}. This effect, which as we believe has not been noted yet in the literature, is explicitly related to the bulk-surface nature of the problem since in elasticity theory the commutability of homogenisation and  linearisation  indeed can be checked \cite{ Neukamm2,Jesenko, Neukamm1}.

 Our  paper is  organised  as follows. In Section \ref{sec: not} we introduce basic notation. Section \ref{setting section} contains the setting and our main results. The proof of the stochastic homogenisation  results can be found in Section~\ref{sec 4}. Eventually, the $\Gamma$-convergence results for asymptotically piecewise rigid functions are addressed in Section \ref{sec: limit}. In the appendix we collect some auxiliary results, as well as details on Example~\ref{main-example}.

 \section{Notation}
 \label{sec: not}

We introduce basic notation. Let $d \in \lbrace 2,3 \rbrace$. Given $x \in \mathbb{R}^d$ we denote by $|x|$ its Euclidean norm. For every $x,y \in \mathbb{R}^d$, $\langle x, y \rangle$ denotes the standard inner product on $\mathbb{R}^d$ between $x$ and $y$, and $x \otimes y$ denotes their tensor product. For  $A,B \subset \mathbb{R}^d$ and $\lambda \in \mathbb{R}$,  we define
 \begin{equation*}
  A+B:=\{z \in \mathbb{R}^d: z=x+y, \ \ \   x   \in A\:\: \text{and}\:\: y \in B\}
 \end{equation*}
and 
 \begin{equation*}
 \lambda A:=\{ z \in \mathbb{R}^d: z=\lambda x,\:\: x \in A\}.
 \end{equation*}
By $A \triangle B$ we denote the symmetric difference of sets. In addition, we denote by $\chi_A$ the characteristic function of a set $A$.  
  We write $A \subset \subset B$ if $\overline{A} \subset B$.  By $\mathbb{S}^{d-1} =  \{x \in \mathbb{R}^d: |x|=1\}$ we denote the unit sphere in $\mathbb{R}^d$. 
Given $x \in \mathbb{R}^d$ and $\rho>0$ we indicate with $Q_{\rho}(x)$ the open cube  with center in $x$ and side length $\rho$, oriented according to the canonical orthonormal basis $\{e_1,...,e_d\}$, that is
 \begin{equation*}
 Q_{\rho}(x):=\Big\{ y \in \mathbb{R}^d: \max_{i=1,...,d}|y_i - x_i|<\frac{\rho}{2}\Big\}.
 \end{equation*}
 Given $\nu \in \mathbb{S}^{d-1}$  we fix an orthogonal matrix $R_{\nu}$ such that $R_{\nu}(e_d)=\nu$. Then, we denote by $Q^{\nu}_{\rho}(x)$ the open cube, with center in $x$ and side length $\rho$, oriented according to the orthonormal basis $\{ R_{\nu}(e_1),...,\nu\}$, that is
 \begin{equation}\label{eq: Qnot}
     Q_{\rho}^{\nu}(x)=R_{\nu}Q_{\rho}(0)+x.
 \end{equation}
   Similarly, for every $x \in \mathbb{R}^d$ and $\rho>0$, we indicate with $B_\rho(x)$ the open ball with center in $x$ and radius $\rho$. 
 We denote by $\mathbb{R}^{d\times d}$ the set of real $d \times d$ matrices and, given $M \in \mathbb{R}^{d\times d}$, we indicate with $\mathrm{det}(M)$ its determinant and with $M^T$ its transpose. By
$
  \mathbb{R}_{\mathrm{skew}}^{d \times d}:=\{M \in \mathbb{R}^{d \times d}: M=-M^T\}
$
we denote the set of  $d \times d$ skew-symmetric  matrices, and indicate the set of rotation matrices in $\mathbb{R}^d$ by
 \begin{equation*}
 SO(d):=\{M \in \mathbb{R}^{d \times d}\colon  \ \  \ M^TM = \mathbb{I}, \ \ \  \mathrm{det}(M)=1\},
 \end{equation*}
where   the   identity matrix will be represented with the symbol $\mathbb{I}$. We proceed with further notation for sets and measures: 
 
 \begin{enumerate}

     \item[$(\mathrm{a})$] By  $\mathcal{A}$ we denote the family of all open, bounded subsets of $\mathbb{R}^d$, and by $\mathcal{A}_0$  the family of all  open, bounded subsets of $\mathbb{R}^d$ with Lipschitz boundary. For open  bounded subsets we always use $A$, and we use $U$ in place of $A$ if the set has also Lipschitz boundary.
     \item[$(\mathrm{b})$] We denote by $\mathcal{M}(A;\mathbb{R}^{d \times d})$ the space of $\mathbb{R}^{d \times d}$-valued bounded Radon measures on $A$.   For every $\mu \in \mathcal{M}(A;\mathbb{R}^{d \times d})$ we denote by $\vert \mu \vert$ the corresponding total variation. 
     \item[$(\mathrm{c})$] For $E \subset \mathbb{R}^d$, we denote by $\partial E$ its topological boundary  and by $\partial^* E$  its essential boundary.
\item[$(\mathrm{d})$] By $\mathcal{L}^k$ and $\mathcal{H}^k$ we indicate respectively the $k$-dimensional Lebesgue and Hausdorff measure. 
\item[$(\mathrm{e})$] Let $X$ be a topological space. The Borel sets are the elements of the $\sigma$-algebra generated by the open sets of $X$. Such $\sigma$-algebra on $X$ (called also Borel $\sigma$-algebra on $X$) is denoted by $\mathcal{B}(X)$. When $X=\mathbb{R}^d$ or $X=\mathbb{S}^{d-1}$, we use a simplified notation to denote their corresponding Borel $\sigma$-algebras, namely $\mathcal{B}^d$ and $\mathcal{B}_{\mathbb{S}}^d$, respectively.
\item[$(\mathrm{f})$] Given $n$ measurable spaces $(X_1,\Sigma_1),...,(X_n,\Sigma_n)$, we denote with $\Sigma_1 \otimes ... \otimes \Sigma_n$ the product $\sigma$-algebra on $X_1 \times ... \times X_n$. 
 \end{enumerate}

 We proceed by introducing relevant function spaces and refer to  \cite{ambrosio2000fbv} for a more comprehensive discussion.  Given $A \subset \mathbb{R}^d$ open, we  denote by $L^0(A;\mathbb{R}^d)$ the set of measurable functions $u \colon A \to \mathbb{R}^d$. For every $u \in L^0(A;\mathbb{R}^d)$ we indicate by $J_u$ the set of its  (weak)   approximate jump points. Then, we denote by $u^+$ and $u^-$ the traces of $u$ on $J_u$, according to the orientation induced by a measure-theoretical unit   normal vector $\nu_u$ of $J_u$, and write $[u]: = u^+-u^-$ to denote the jump height.   For $u \in BV$, $Du$ and $D^su$ denote the distributional derivative of $u$ and its singular
part with respect to the Lebesgue measure, respectively. The density of the absolutely continuous part of $Du$ is denoted by $\nabla u$. 
  Finally, for $p \in (1,\infty)$, denoting by $(G)SBV$  the space of  \emph{(generalised) special functions of bounded variation} (see \cite[Section 4]{ambrosio2000fbv} and \cite[Section 2]{dalfratoa04}) and by $SBD$  the space of \emph{special functions of bounded deformation} (see e.g.\  \cite{Dal11}), we define  \EEE
\begin{equation*}
SBV^p(A;\mathbb{R}^d)=\{u \in SBV(A;\mathbb{R}^d): \nabla u \in L^p(A;\mathbb{R}^{d \times d}), \mathcal{H}^{d-1}(J_u)<\infty\},   
\end{equation*}
and 
\begin{equation}\label{GSBV}
 GSBV^p(A;\mathbb{R}^d)=\{u \in GSBV(A;\mathbb{R}^d): \nabla u \in L^p(A;\mathbb{R}^{d \times d}), \mathcal{H}^{d-1}(J_u)<\infty\}.  
\end{equation}
 We also recall the definition  of Caccioppoli partitions:  A \emph{Caccioppoli partition} of an open set $A\subset \R^d$  is a countable family $(A_j)_{j \in \mathbb{N}}$ such that $A_i \cap A_j = \emptyset$ if $i\neq j$,  $\bigcup_{j \in \mathbb{N}}A_j = A$, and $\sum_{j \in \mathbb{N}}\mathcal{H}^{d-1}(\partial^* A_j \cap A)<\infty$. We indicate by $\mathcal{P}(A)$  the set of Caccioppoli partitions on $A$.

\EEE


\section{Setting of the problem and main results}
\label{setting section}

In this section, we introduce the setting and formulate the main results.

\subsection{Setting of the problem}\label{sec: setting}

We start with introducing piecewise rigid functions. 

\begin{definition}[Piecewise rigid functions]\label{def: PR}
Let $A \subset \mathbb{R}^d$ be an open set. For $L=SO(d)$ or $L= \mathbb{R}_{\rm skew}^{d\times d} $, define the space
\begin{equation*}
\begin{split}
PR_{L}(A)= \Big\{ u\in L^0(A;\mathbb{R}^d)\colon u(x)=\sum_{j\in \mathbb{N}}(M_j\, x+b_j)\chi_{P_j}(x),\:\:\text{where}\:\: \ M_j \in L, \,  b_j \in \mathbb{R}^d,  (P_j)_{j\in \mathbb{N}} \in \mathcal{P}(A) \Big\}.  
\end{split}  
\end{equation*}
\end{definition}

Note that for $L = \lbrace 0 \rbrace$ the above definition corresponds to the space of \emph{piecewise constant functions}. Accordingly, in the sequel  we denote the  space of piecewise constant functions by $PR_0(A)$.  Notice that, by the properties of Caccioppoli partitions, given $u=\sum_{j\in \mathbb{N}}(M_j\, \cdot +b_j)\chi_{P_j}\in PR_L(A)$ it holds $\mathcal{H}^{d-1}(J_u \setminus \bigcup_{j} \partial^* P_j)=0$.  We also remark that representations of piecewise rigid functions are not unique. In this work, we often use their \emph{pairwise distinct} representation, i.e.\ for $u(x)=\sum_{j\in \mathbb{N}}q_j\chi_{P_j} \in PR_L(A)$, we  assume that the affine functions $(q_j)_j$ are pairwise distinct. In this case, it particularly holds $\mathcal{H}^{d-1}(J_u  \triangle  (\bigcup_{j}\partial^* P_j \setminus \partial U))=0$. Moreover, two piecewise rigid functions $u_1,u_2$ can always be represented by a single Caccioppoli partition as follows: if $u_1=\sum_j q_j^1 \chi_{P^1_j}$  and $u_2=\sum_j q_j^2 \chi_{P^2_j}$, we can construct a \emph{refined} Caccioppoli partition $(P_j)_j$ by taking the intersections $(P_j^1 \cap P_k^2)_{j,k}$. Then both $u_1$ and $u_2$ can be represented with $(P_j)_j$, but in general such a representation is not pairwise distinct. For more details we refer to \cite[Section 3.2]{Friedrich_2020}.

We proceed with the precise definition of random surface densities and energies.  Fix $c_1 \in (0,1)$, $c_2 \geq 1$, $c_0 \geq 1$, and an increasing modulus of continuity $\sigma \colon [0,\infty) \to [0,  \frac{1}{2}]$  with $\sigma(0)=0$. We define the following family of functions \begin{equation}\label{FFF}
 \mathcal{F}=\mathcal{F}(  c_0, c_1, c_2,  \sigma):= \Big\{  f \colon \mathbb{R}^d \times \mathbb{R}^{d} \setminus \{0\}\times  \mathbb{S}^{d-1}  \to [0,\infty): \:\: f \:\:\text{satisfies}\:\: (f1)-(f7)\Big\},  
 \end{equation}
 where $(f1)-(f7)$ are the following properties:     \\
 \begin{enumerate}
     \item[$({f1})$] (measurability) $f$ is $\mathcal{B}^d \otimes \mathcal{B}^d \otimes \mathcal{B}_{\mathbb{S}}^d$ measurable,
     \item[$({f2})$] (continuity in $\zeta$) the function $\zeta \to f(x,\zeta,\nu)$ is continuous, in particular
\begin{equation*}
|f(x,\zeta_2,\nu)-f(x,\zeta_1,\nu)|\leq \sigma(|\zeta_2-\zeta_1|)\big(f(x,\zeta_1,\nu)+f(x,\zeta_2,\nu)\big) 
\end{equation*}
for every $x \in \mathbb{R}^d$, $\zeta_1,\zeta_2 \in \mathbb{R}^{d}\setminus \{0\}$, and $\nu \in \mathbb{S}^{d-1}$,
\item[$({f3})$] (estimate for $|\zeta_1|\leq |\zeta_2|$) if $|\zeta_1|\leq|\zeta_2|$, for every $x \in \mathbb{R}^d$ and for every $\nu \in \mathbb{S}^{d-1}$, we have
\begin{equation*}
f(x,\zeta_1,\nu)\leq c_0 f(x,\zeta_2,\nu),
\end{equation*}
 \item[$({f4})$] (estimate for $c_0|\zeta_1|\leq |\zeta_2|$) if $c_0|\zeta_1|\leq|\zeta_2|$, for every $x \in \mathbb{R}^d$ and for every $\nu \in \mathbb{S}^{d-1}$ we have
 \begin{equation*}
 f(x,\zeta_1,\nu)\leq f(x,\zeta_2,\nu),      
 \end{equation*}
\item[$({f5})$] (lower bound) for every $x \in \mathbb{R}^d$, $\zeta \in \mathbb{R}^d \setminus \{0\}$, and for every $\nu \in \mathbb{S}^{d-1}$ it holds 
\begin{equation*}
c_1 \leq f(x,\zeta,\nu),    
\end{equation*}
\item[$({f6})$] (upper bound) for every $x \in \mathbb{R}^d$, $\zeta \in \mathbb{R}^d \setminus \{0\}$, and for every $\nu \in \mathbb{S}^{d-1}$ it holds 
\begin{equation*}
f(x,\zeta,\nu) \leq c_2,
\end{equation*}
\item[$({f7})$] (symmetry) for every $x \in \mathbb{R}^d$, $\zeta \in \mathbb{R}^d \setminus \{0\}$, and for every $\nu \in \mathbb{S}^{d-1}$ it holds 
\begin{equation*}
f(x,\zeta,\nu)=f(x,-\zeta,-\nu).   
\end{equation*}
 \end{enumerate} 

In the sequel, $(\Omega, \mathcal{I}, \mathbb{P})$ denotes a fixed probability space. 

\begin{definition}[Random surface density]
\label{randomsurfaceintegrand}
 A function $f \colon \Omega \times\mathbb{R}^d \times \mathbb{R}^{d} \setminus \{0\}\times \mathbb{S}^{d-1}\to [0,\infty)$ is said to be a \emph{random surface density} if the following two properties are satisfied:
\begin{enumerate}
    \item[$(\mathrm{i})$] $f$ is $\mathcal{I}\otimes \mathcal{B}^d \otimes \mathcal{B}^d \otimes \mathcal{B}_{\mathbb{S}}^d$ measurable, 
    \item[$(\mathrm{ii})$]  $f(\omega,\cdot,\cdot,\cdot) \in \mathcal{F}$ for every $\omega \in \Omega$.
\end{enumerate}
\end{definition}

\begin{definition}(Random surface energy)\label{defrandomenergy}
Let $L=SO(d)$ or $L=\mathbb{R}_{\rm skew}^{d\times d}$, and let $\eps>0$. A \emph{random surface energy} is a functional $\mathcal{E}_{\eps}\colon \Omega \times L^0(\mathbb{R}^d;\mathbb{R}^d)\times \mathcal{A}\to [0,\infty]$ of the form
\begin{equation}
\label{randomenergies}
\mathcal{E}_\eps[\omega](u,A)=\begin{cases}
 \int_{A \cap J_u}{f(\omega,  \tfrac{x}{\eps}, [u](x),\nu_{u}(x))\,\mathrm{d}\mathcal{H}^{d-1}(x)} & u_{|_A} \in PR_L(A),   \\
 +\infty  & \text{otherwise}
\end{cases}
\end{equation}
for every $A \in \mathcal{A}$ and $\omega \in \Omega$, where $f$ is a random surface density. 
\end{definition}
In the following, to simplify the notation when $\varepsilon=1$,  for the energy defined by \eqref{randomenergies} we write $\mathcal{E}$ in place of $\mathcal{E}_1$.

\begin{remark}[Properties $(f1)$--$(f7)$] \EEE {\normalfont
Notice that properties $(f1)$ and $(f7)$ are necessary for the well-posedness of \eqref{randomenergies}. In addition, as pointed out in \cite[Remark 3.2]{cagnetti2018gammaconvergence},  if $\zeta \to f(x,\zeta,\nu)$ is monotone with respect to $|\zeta|$, i.e.\ $f(x,\zeta_1,\nu)\leq f(x,\zeta_2,\nu)$  whenever $|\zeta_1| \leq |\zeta_2|$, then  properties $(f3)$ and $(f4)$ are automatically satisfied since $c_0\geq1$, but the converse is not true, i.e.\ $(f3)$ and $(f4)$  are a weaker condition compared to monotonicity. Properties $(f5)$ and $(f6)$ ensure that a control on the measure of the jump  set implies a control on the energy and vice versa. This is a standard assumption  for homogenisation of surface integrals and is indeed crucial for the so-called  ``fundamental estimate'' for a sequence of surface densities $(f_{\varepsilon})_\eps \subset \mathcal{F}$ and for a compactness result in the sense of $\Gamma$-convergence, see \cite{Friedrich_2020}. } 
 \end{remark}
 
In the next definitions, we introduce groups of $\mathbb{P}$-preserving transformations and discuss how they can be used to define the notions of stationarity and ergodicity for random surface densities.

\begin{definition}{(Group of $\mathbb{P}$-preserving transformations)} 
\label{def: pre}
A $d$-dimensional group of \emph{$\mathbb{P}$-preserving transformations} on $(\Omega,\mathcal{I}, \mathbb{P})$ is a family $(\tau_z)_{z \in \mathbb{Z}^d}$ (resp.\  $(\tau_z)_{z \in \mathbb{R}^d}$) of transformations $\tau_z \colon \Omega \to \Omega$ satisfying the following properties:
\begin{enumerate}
    \item[$(\mathrm{\tau}_1)$] (measurability) $\tau_{z}$ is $\mathcal{I}$ measurable for every $z \in \mathbb{Z}^d$ (resp.\  for every $z \in \mathbb{R}^d)$, 
    \item[$(\mathrm{\tau}_2)$] (bijectivity) $\tau_z \colon \Omega \to \Omega$ is bijective for every $z \in \mathbb{Z}^d$ (resp.\  for every $z \in \mathbb{R}^d$), 
    \item[$(\mathrm{\tau}_3)$] (group property) $(\tau_{z})_{z \in \mathbb{Z}^d}$ (resp.\  $(\tau_{z})_{z \in \mathbb{R}^d}$) is a group with respect to the composition operation, stable with the sum in $\mathbb{R}^d$, that is $\tau_{x} \circ \tau_{y} = \tau_{y} \circ \tau_{x} = \tau_{x+y}$ for every $x,y \in \mathbb{Z}^d $ (resp.\  for every $x,y \in \mathbb{R}^d)$, and $\tau_0 \colon \Omega \to \Omega$ is the identity, 
    \item[$(\mathrm{\tau}_4)$] (invariance) $(\tau_{z})_{z \in \mathbb{Z}^d}$ (resp.\  $(\tau_{z})_{z \in \mathbb{R}^d}$) preserves probability, that is $\mathbb{P}(\tau_z(E))=\mathbb{P}(E)$ for every $z \in \mathbb{Z}^d$ (resp.\ for every $z \in \mathbb{R}^d)$ and for every $E \in \mathcal{I}$. 
\end{enumerate}
 In addition, if it  also holds that   
 \begin{enumerate}
     \item[$(\mathrm{\tau}_5)$] given $E \in \mathcal{I}$, $\mathbb{P}(E \triangle \tau_z(E))=0$ for every $z\in \mathbb{Z}^d$ (resp.\  for every $z\in \mathbb{R}^d$) implies $\mathbb{P}(E)=0$ or $\mathbb{P}(E)=1$, 
 \end{enumerate}
 we say that $(\tau_z)_{z \in \mathbb{Z}^d}$ (resp.\  $(\tau_z)_{z \in \mathbb{R}^d})$ is \emph{ergodic}.
\end{definition}

We point out that $(\mathrm{\tau}_5)$   will be needed only to show that the homogenised $\Gamma$-limit is deterministic whereas  all other results in this paper hold also without this condition.    

\begin{definition}(Stationarity)
\label{stationarity}
We say that a random surface density $f$ is \emph{stationary} with respect to a ($d$-dimensional) group of $\mathbb{P}$-preserving transformations $(\tau_z)_{z \in \mathbb{Z}^{d}}$ (resp.\  $(\tau_z)_{z \in \mathbb{R}^{d}}$) on $(\Omega,\mathcal{I},\mathbb{P})$ if      
\begin{equation*}
\begin{split}
f(\tau_{z}\omega,x,\zeta,\nu)=f(\omega,x+z,\zeta,\nu), 
\end{split}
\end{equation*}
 for every $\omega \in \Omega$, $x \in \mathbb{R}^d, z \in \mathbb{Z}^d$ (resp.\  $z \in\mathbb{R}^d$), $\zeta \in \mathbb{R}^{d}\setminus\{0\}$, and $\nu \in \mathbb{S}^{d-1}$.
\end{definition}

\begin{definition}[Minimisation problem]
\label{definfimumproblem}
Let $L=SO(d)$ or $L=\mathbb{R}_{\mathrm{skew}}^{d \times d}$. For every $A \in \mathcal{A}$ and $v \in L^{0}(\mathbb{R}^d;\mathbb{R}^d)$ with $v|_A \in PR_0(A)$, we  define  $m^L_{\mathcal{E}}(v,A)$  by  
\begin{align*}
m^{SO(d)}_{\mathcal{E}}(v,A) &=\inf\big\{\mathcal{E}({\rm id} +  u,A): u\in PR_0(A)\:\: \text{and}\:\: u=   v  \:\: \text{near}\:\: \partial A\big\},\\
m^{\R^{d \times d}_{\rm skew}}_{\mathcal{E}}(v,A)& =\inf\big\{\mathcal{E}(u,A): u\in PR_{\R^{d \times d}_{\rm skew}}(A)\:\: \text{and}\:\: u=v\:\: \text{near}\:\: \partial A\big\},
\end{align*} 
where with ``$u=v\:\: \text{near}\:\: \partial A$'' we mean that there exists a neighbourhood $N \subset A$ of $\partial A$ such that $u=v$ on $N$.  
\end{definition}

For the definition of  $m^{SO(d)}_{\mathcal{E}}$, we  emphasise  that we do not consider the problem on the entire set $PR_{SO(d)}$ but for technical reasons (explained in   Remark \ref{remarkhomogenisationSO(d)}) only on the subset ${\rm id} + PR_0$. We refer to the discussion  in Remark \ref{rem: omportant}($\mathrm{ii}$)  why this change does not affect our analysis.    Definition \ref{definfimumproblem} will be usually used for competitors  of the form  
\begin{equation}\label{Lw}
u_{x,\zeta,\nu}(y):=\begin{cases}
\zeta & \text{if}\:\: \langle y-x, \nu \rangle \geq 0, \\
0  & \text{if}\:\: \langle y-x,\nu \rangle < 0
\end{cases}
\end{equation}
 for  $x \in \mathbb{R}^d$, $\zeta \in \mathbb{R}^{d}\setminus \{0\}$, and $\nu \in \mathbb{S}^{d-1}$.

\subsection{Stochastic homogenisation of surface energies for piecewise rigid functions}    In this section, we give our main result on the stochastic homogenisation of surface integrals. We start with  the existence of limits in asymptotic cell formulas that will be used
in the statement of the main result.

\begin{theorem}[Homogenisation   formula]
\label{HomogenizationFormula}
Let  $L=SO(d)$ or $L=\mathbb{R}_{\rm skew}^{d \times d}$. Let $f$ be a stationary random surface density with respect to a group $\{\tau_z\}_{z \in \mathbb{Z}^d}$ (resp.\  $\{\tau_z\}_{z \in \mathbb{R}^d}$) of $\mathbb{P}$-preserving transformations on $(\Omega,\mathcal{I}, \mathbb{P})$, and let $\mathcal{E}$ be the corresponding random surface energy, see \eqref{randomenergies}. In addition, for every $\omega \in \Omega$ let $m^L_{\mathcal{E}[\omega]}$ be defined as in Definition \ref{definfimumproblem}  with $\mathcal{E}[\omega]$ in place of $\mathcal{E}$.\\
 Then, there exists an event $\Omega' \in \mathcal{I}$, with $\mathbb{P}(\Omega')=1$, and a random surface density $f_{\mathrm{hom}} \colon \Omega \times \mathbb{R}^d \setminus \{0\} \times \mathbb{S}^{d-1} \to \mathbb{R}$, independent of $x$, such that for every $\omega \in \Omega'$, $x \in \mathbb{R}^d$, $\zeta \in \mathbb{R}^d \setminus \{0\}$,  $\nu \in \mathbb{S}^{d-1}$ and for every function $r\colon (0,\infty)\to (0,\infty)$, with $r(t)\geq t$ for every $t>0$, it holds
\begin{equation}
\label{eqhom}
f_{\mathrm{hom}}(\omega,\zeta,\nu)=\lim\limits_{t \to \infty}\frac{ m^L_{\mathcal{E}[\omega]}(u_{tx,\zeta,\nu},Q_{r(t)}^{\nu}(tx))}{{r(t)}^{d-1}},
\end{equation}
where $u_{tx,\zeta,\nu}$ is defined in \eqref{Lw}. Moreover, if $(\tau_z)_{z \in \mathbb{Z}^{d}}$ (resp.\  $(\tau_z)_{z \in \mathbb{R}^{d}}$) is ergodic, then $f_{\mathrm{hom}}$ does not depend on $\omega$ and we have 
\begin{equation}
\label{eqerg}
f_{\mathrm{hom}}(\zeta,\nu)=\lim\limits_{t \to \infty} \frac{1}{{r(t)}^{d-1}}\int_{\Omega}{  m^L_{\mathcal{E}[\omega]}(u_{0,\zeta,\nu},Q_{r(t)}^{\nu}(0))\,\mathrm{d}\mathbb{P}(\omega)}.
\end{equation}
\end{theorem}
\EEE

Theorem \ref{HomogenizationFormula} is the analogue of \cite[Theorem 3.12]{cagnetti2017stochastic} for piecewise rigid functions. It states that the blow up limit defining the cell formula exists $\mathbb{P}$-almost surely and,  as expected,  does not depend on $x$.  This homogenisation formula crucially enters in  the first main result of this paper, namely the almost sure $\Gamma$-convergence and integral representation result for the random functionals $(\mathcal{E}_\eps)_\eps$, under stationarity assumption for the density $f$. We observe that in the case $L = SO(d)$ the minimum problem defining $f_{\mathrm{hom}}$ is  written in terms of piecewise constant functions. Thus, the homogenisation formula actually  coincides with the one identified in  \cite[Theorem 3.12]{cagnetti2017stochastic}.

\begin{theorem}[$\Gamma$-convergence]
\label{gammaconv}
 Let   $L=SO(d)$  or $L=\mathbb{R}_{\rm skew}^{d \times d}$. Let $f$ be a stationary random surface density with respect to a group $(\tau_z)_{z \in \mathbb{Z}^d}$ (resp.\  $(\tau_z)_{z \in \mathbb{R}^d}$) of $\mathbb{P}$-preserving transformations on $(\Omega,\mathcal{I},\mathbb{P})$. Let $\mathcal{E}_{\varepsilon}$ be as in \eqref{randomenergies}, let $\Omega' \in \mathcal{I}$ (with $\mathbb{P}(\Omega')=1$), $f_{\mathrm{hom}}$ as in Theorem \ref{HomogenizationFormula}, and let $\mathcal{E}_{\mathrm{hom}}\colon \Omega \times L^0(\mathbb{R}^d;\mathbb{R}^d)\times \mathcal{A} \to [0,\infty]$ be the surface functional defined by
\begin{equation}
\label{homogenisedenergy}
\mathcal{E}_{\mathrm{hom}}[\omega](u,A)=\begin{cases}
\int_{J_u \cap A}f_{\mathrm{hom}}(\omega,[u](x),\nu_u(x))\,\mathrm{d}\mathcal{H}^{d-1}(x) & u_{|_A}\in PR_L(A),\\
+\infty &  \text{otherwise}
\end{cases}
\end{equation}
for $\omega \in \Omega$ and $A \in \mathcal{A}$. 
Then,   
\begin{equation}\label{gamma}
\mathcal{E}_{\varepsilon}[\omega](\cdot, U )\:\: \Gamma\text{-converge}\:\:\text{to}\:\: \mathcal{E}_{\mathrm{hom}}(\cdot,U ) \ \ \ \ \text{with respect to convergence in measure on $U$}, 
\end{equation}
for every $\omega \in \Omega'$ and every $U  \in  \mathcal{A}_0 $.
Further, if $(\tau_z)_{z \in \mathbb{Z}^d}$ (resp.\  $(\tau_z)_{z \in \mathbb{R}^d}$) is ergodic, then  $\mathcal{E}_{\mathrm{hom}}$ is a deterministic functional,  i.e.\  it does not depend on $\omega$.
\end{theorem}

We proceed with convergence of infima and almost  minimisers  for certain boundary value problems. To this end, we fix $\Psi \subset   \Psi'$ with  $\Psi$ convex (for technical reasons, cf.\ Lemma \ref{truncation} and Remark \ref{convexity} below), such that $U :=\Psi' \in \mathcal{A}_0$ and   $V := \Psi' \setminus  \overline{\Psi} \in \mathcal{A}_0$.  We let 
\begin{align}\label{BBBC}
PR^{u_0}_L(U) := \lbrace  u \in L^{0}(\mathbb{R}^d;\mathbb{R}^d) \colon  u_{|_U} \in PR_L(U)\:\: \text{and}  \, u = u_0 \text{ on }  V\rbrace, 
\end{align}
where $u_0 \in PR_L(U) \cap L^\infty(U;\R^d)$ plays the role of a prescribed boundary condition on $\partial_D \Psi := \partial \Psi \cap \Psi'$.  Here, we suppose that $u_0$ is a Lipschitz function in a neighborhood  of $V$ in  $U$. This is a standard way to prescribe boundary conditions for functions exhibiting discontinuities. Note that  $V$ should consist of at least two connected components as otherwise minimisers are trivially a rigid motion. Based  on a fundamental estimate in $PR_L$, boundary conditions can be incorporated in the $\Gamma$-convergence result of Theorem \ref{gammaconv}. We refer to Theorem \ref{th_ bundary data} in the appendix for details.

\begin{corollary}[Convergence of infima]
\label{convminima}
Let $L=SO(d)$ or $L=\mathbb{R}_{\rm skew}^{d \times d}$. Let $f$ be a stationary random surface density with respect to a group $(\tau_z)_{z \in \mathbb{Z}^d}$ (resp.\ $(\tau_z)_{z \in \mathbb{R}^d}$) of $\mathbb{P}$-preserving transformations on $(\Omega,\mathcal{I},\mathbb{P})$, and let $\mathcal{E}_{\varepsilon}$ be as in \eqref{randomenergies}. Consider $\Omega' \in \mathcal{I}$ (with $\mathbb{P}(\Omega')=1$) and $f_{\mathrm{hom}}$ as in Theorem \ref{HomogenizationFormula}, and let $\mathcal{E}_{\mathrm{hom}}$ as in \eqref{homogenisedenergy}. 
 Let $u_0 \in PR_L(U) \cap L^\infty(U;\R^d)$ and $\omega \in \Omega'$.  Then, it holds  
\begin{equation}\label{eq: 3.8X}
 \lim\limits_{\varepsilon \to 0  } \, \inf_{u \in PR^{u_0}_L(U)}\mathcal{E}_{\varepsilon}[\omega](u, U) = \inf_{ u \in PR^{u_0}_L(U)}\mathcal{E}_{\mathrm{hom}}[\omega](u,U). 
\end{equation}
\end{corollary}

\begin{corollary}[Convergence of almost minimisers]
\label{convergencealmostminimisers}
Under the hypotheses of Corollary \ref{convminima}, for every $\theta>0$ there exists a  sequence $(\eps_j)_j$ converging to zero and a sequence $(u^{\theta}_{\varepsilon_j})_{j} \subset  PR^{u_0}_L(U)$, uniformly bounded in $L^{\infty}(U;\mathbb{R}^d)$,   and $u^{\theta} \in PR^{u_0}_L(U) \cap L^{\infty}(U;\mathbb{R}^d)$ such that $u_{\varepsilon_j}^{\theta}\to u^{\theta}$ in measure on $U$ and 
\begin{align*}
\mathcal{E}_{\varepsilon_j}[\omega](u^{\theta}_{\varepsilon_j}, U) & \leq \inf_{v \in PR^{u_0}_L(U)}\mathcal{E}_{\varepsilon_j}[\omega](v, U)+ \theta,    \\ 
\mathcal{E}_{\mathrm{hom}}[\omega](u^{\theta},U) & \leq \inf_{v \in PR^{u_0}_L(U)}\mathcal{E}_{\mathrm{hom}}[\omega](v,U)+\theta, 
\end{align*} 
 
\end{corollary}

\begin{remark}[Compactness issue]
\label{remarkcompactness}
{\normalfont Note that $\Gamma$-convergence usually implies convergence of infima, as we state in  Corollary \ref{convminima}. Obtaining also   convergence of (almost) minimisers is a delicate issue since a compactness property is required  ensuring  that subsequences converge at least in measure.  In     \cite[Corollary 3.14]{cagnetti2017stochastic} this problem is circumvented by adding a fidelity term. Without such fidelity terms, compactness of minimising sequences has been shown for some free-discontinuity problems \cite{friedrichcompactness, FriPerSol20a}  which however does not cover the case of piecewise rigid functions. In  Corollary \ref{convergencealmostminimisers}, we provide a result in that direction for almost minimisers up to some error term $\theta$, resorting  to a truncation technique in Lemma \ref{truncation} below. The latter allows us to apply a compactness result in $PR_L$, see \cite[Lemma 3.3]{Friedrich_2020}, and thus to   mimic the proof of the fundamental theorem of $\Gamma$-convergence. }
\end{remark}
\EEE


The results announced in this subsection are proved in Section \ref{sec 4}.  

\subsection{Random surface energies defined on asymptotically piecewise rigid functions}\label{sec: resul2}

We now introduce a  nonlinear random Griffith model for nonsimple materials and discuss its limit when configurations become asymptotically piecewise rigid. Let $W\colon \Omega \times \mathbb{R}^d \times \mathbb{R}^{d \times d}\to [0,\infty)$ be  a random elastic energy density which has a single well and is  frame indifferent. More precisely,  we suppose  that there exists $c>0$ such that  
\begin{enumerate}
    \item[$(\mathrm{W1})$] (frame indifference) $W(\omega,x,RM)=W(\omega,x,M)$ for all $\omega \in \Omega$, $x \in \mathbb{R}^d$, $M \in \mathbb{R}^{d \times d}$ and $R \in SO(d)$, 
    \item[$(\mathrm{W2})$] (lower bound) $W(\omega,x,M)\geq c\,\mathrm{dist}^2(M,SO(d))$ for all $\omega \in \Omega$, $x \in \mathbb{R}^d$, and $M \in \mathbb{R}^{d \times d}$, and  $W(\omega,x,M)=0$  if and only if  $M \in SO(d)$.
\end{enumerate} 
For given $A \in \mathcal{A}$, recall the definition of $GSBV^2(A;\mathbb{R}^d)$ in \eqref{GSBV} and define 
\begin{equation}\label{eq: GSBV22}
    GSBV_2^2(A;\mathbb{R}^d):=\big\{y \in GSBV^2(A;\mathbb{R}^d): \nabla y \in GSBV^2(A;\mathbb{R}^{d\times d})\big\},
\end{equation}
where the approximate differential and the jump set of $\nabla y$ will be denoted by $\nabla^2 y$ and $J_{\nabla y}$, respectively.  For $\delta, \eps >0$ and $\beta \in (0,1)$, we introduce the energies $ \mathcal{F}_{\varepsilon, \delta} \colon \Omega \times L^0(\mathbb{R}^d;\mathbb{R}^{d})\times \mathcal{A}\to [0,\infty]$ as 
\begin{align}
\label{completenergy}
{\mathcal{F}}_{\varepsilon,\delta}[\omega](y,A)= 
 \int_{A} \frac{1}{\delta^{2}} W\Big(\omega,\frac{x}{\varepsilon}, \nabla y(x)\Big)+   \frac{1}{\delta^{2\beta}} |\nabla^2 y(x)|^2\, \mathrm{d}x + \int_{A}f\Big(\omega,\frac{x}{\varepsilon},[y](x),\nu_y(x)\Big)\, \mathrm{d}\mathcal{H}^{d-1}(x)
\end{align}
whenever $y_{|_A} \in GSBV_2^2(A;\mathbb{R}^d)$, $J_{\nabla y} \cap A \subset J_y \cap A$, and ${\mathcal{F}}_{\varepsilon}[\omega](u,A) = + \infty$ else.

As before, $\eps$ stands for the size of the microstructure, whereas $\delta$ represents the size of the strain. Indeed, since $W$ grows quadratically around $SO(d)$,  for a  configuration $y$ with finite energy \eqref{completenergy}, the strain $\nabla y$ is typically $\sim \delta$ close to the set of rotations, cf.\ e.g.\  \cite{dalnegper01, friedrich2017derivation, friedrich2019griffith}. The model is a variant of \cite{friedrich2019griffith}, which   we call a Griffith-type model for \emph{nonsimple materials}  \cite{Toupin1962ElasticMW, Toupin1962ElasticMW2} due to the presence of the second term. On the one hand, this term enhances the rigidity properties of the nonlinear model. At the same time, the scaling factor $ \frac{1}{2\beta}$ with $\beta<1$ ensures that this contribution vanishes in the small-strain limit, see \cite{friedrich2019griffith}.  \EEE  Eventually, we mention that in our model the  regularisation  effect acts on the entire intact region $A \setminus  J_y $ of the material which is modeled by the condition $J_{\nabla y} \subset   J_y$ (to be understood $\mathcal{H}^{d-1}$-a.e.).

In contrast to  \cite{friedrich2019griffith}, we treat  the case of a random surface energy, and study the simultaneous limit of small strains $\delta \to 0$ and homogenisation  $\eps \to 0$. As the effective limit is described purely by a surface energy, the exact form of the elastic energy density $W$ is irrelevant. For sake of generality, we still allow it to depend on   $\omega$ and $x$, although this does not play a role for our result.

We now present the $\Gamma$-limit of $\mathcal{F}_{\eps,\delta}$ in the simultaneous limit $\eps, \delta \to 0$. To this end, given a sequence $(\delta_\eps)_\eps$ with $\delta_\eps \to 0$ as $\eps \to 0$, we write $\mathcal{F}_\eps[\omega]:= \mathcal{F}_{\eps, \delta_\eps}[\omega]$ for each $\omega \in \Omega$ for simplicity. \EEE

\begin{theorem}[Homogenisation for asymptically piecewise rigid functions: the nonlinear case] \EEE
\label{brittlemateriallimit}
Let $L=SO(d)$ and $\beta \in (0,1)$. Let  $(\delta_\eps)_\eps\subset (0,1)$ with $\delta_\eps \to 0$ as $\eps \to 0$. Let $f$ be a  stationary random surface density with respect to a group $(\tau_z)_{z \in \mathbb{Z}^d}$ (resp.\  $(\tau_z)_{z \in \mathbb{R}^d}$) of $\mathbb{P}$-preserving transformations on $(\Omega,\mathcal{I},\mathbb{P})$. Let ${\mathcal{F}}_{\varepsilon}$ be as in \eqref{completenergy}, and  let $\Omega' \in \mathcal{I}$ with $\mathbb{P}(\Omega')=1$ as well as $f_{\mathrm{hom}}$ as in Theorem \ref{HomogenizationFormula}. Then, 
\begin{equation*}
{\mathcal{F}}_{\varepsilon}  [\omega](\cdot, U)\:\:  \text{$\Gamma$-converge}\:\:\text{to}\:\: \mathcal{E}_{\mathrm{hom}}[\omega](\cdot, U)  \ \ \ \ \text{with respect to convergence in measure on $U$} 
\end{equation*}
for every $\omega \in \Omega'$ and every  $U \in \mathcal{A}_0$,  where $\mathcal{E}_{\mathrm{hom}}\colon \Omega \times L^0(\mathbb{R}^d;\mathbb{R}^d)\times \mathcal{A} \to [0,\infty]$ is the surface functional defined by \eqref{homogenisedenergy} in the case $L=SO(d)$. 
\end{theorem}

The result is expectable since for a sequence $(y_\eps)_\eps$  with bounded energy the scaling of the elastic energy implies $\int_\Omega {\rm dist}^2(\nabla y_\eps, SO(d)) \, {\rm d}x \to 0$ as $\eps \to 0$. However, whereas it is well known that $GSBV^2$-functions $y$ with  ${\rm dist}^2(\nabla y, SO(d)) = 0 $ a.e.\ are piecewise rigid \cite{Chambolle2006PIECEWISERA}, the approximate case $\int_\Omega {\rm dist}^2(\nabla y_\eps, SO(d)) \, {\rm d}x \approx 0$ is more delicate. It relies on an approximation result of functions with small elastic energy by piecewise rigid functions, making use of the second-order regularisation. We refer to Proposition  \ref{piecewiserigidapproximation} below for details.

Eventually, we address the question how a simultaneous limit $\eps,\delta  \to 0$ can lead to a  linearised  model defined on $PR_L$ for $L = \R^{d \times d}_{\rm skew}$. As in \cite{friedrich2017derivation, friedrich2019griffith}, this should involve a suitable  linearisation  in terms of a rescaled displacement field $u$, related to the  deformation $y$ by
\begin{align}\label{uuu}
 u := \frac{y - {\rm id}}{\delta^\alpha}  \quad \quad \text{for some exponent $\alpha >0$}.
\end{align}
As seen in \cite{friedrich2017derivation, friedrich2019griffith}, the choice $\alpha =1$ leads to a Griffith model in the small-strain limit $\delta \to 0$ featuring both elastic and surface energy.   Consequently, to obtain pure surface energies in the limit, we suppose $\alpha \in (0,1)$ in the following. Heuristically, since  $ {\rm dist}(\nabla y, SO(d))  \sim \delta $, by a Taylor expansion at $\mathbb{I}$  (see \eqref{eq4.65XXX}  below for details) we  get 
\begin{align}\label{negligengli}
\delta  \sim     {\rm dist}(\nabla y, SO(d))  &  =  \frac{1}{2}\Big|  (\nabla y - \mathbb{I})^T +   (\nabla y - \mathbb{I}) \Big| + \mathcal{O}(|\nabla y - \mathbb{I}|^2)  =  \delta^\alpha  |  e(u)| +   \delta^{2\alpha }\mathcal{O}(|\nabla u|^2),    
\end{align} 
where we use the notation $e(u) \coloneqq \frac{1}{2}((\nabla u)^T +  \nabla u ) $. Supposing that the higher-order effect is negligible, this shows $|e(u)|  \sim \delta^{1-\alpha}$, and therefore with $\alpha <1$ we can expect $e(u) \to 0$ as $\delta \to 0$, i.e.\ $u$ is (asymptotically) piecewise rigid. Clearly, this approximation is wrong whenever the deformation gradient $\nabla y$ is not close to the identity. As shown in \cite{friedrich2017derivation}, this calls for a  linearisation  around various different rigid motions, in connection with a suitable partition of the domain. To rule out such  intricate formulation for simplicity, we will assume in the following that
\begin{align}\label{constraints}
|\nabla y(x) - \mathbb{I}| \le \delta^{ \alpha_*} \quad \text{a.e.\ in $A$, for some $\alpha_*  \in (\alpha/2, \alpha)$}. 
\end{align}
This allows us to  linearise around the identity and it also ensures that the higher-order term in \eqref{negligengli} is negligible. To simplify the exposition, we will only treat the case $\alpha_* = 3\alpha/4$ in the sequel.

We now write the energy \eqref{completenergy} in terms of the rescaled displacment fields $u$. Since jump heights $[u] \sim 1$ correspond to $[y] \sim \delta^\alpha$, it is also meaningful to rescale the random surface density,  i.e.\ to replace the surface part of \eqref{completenergy} by
\begin{align}\label{completenergyXXX}
\int_{A}f\Big(\omega,\frac{x}{\varepsilon},\delta^{-\alpha} [y](x),\nu_y(x)\Big)\, \mathrm{d}\mathcal{H}^{d-1}(x). 
\end{align} 
Then, plugging $u$ defined in \eqref{uuu} into \eqref{completenergy} (with surface part given in \eqref{completenergyXXX}), and respecting \eqref{constraints}, we define the energies $\mathcal{F}^{\rm lin}_{\varepsilon, \delta} \colon \Omega \times L^0(\mathbb{R}^d;\mathbb{R}^{d})\times \mathcal{A}\to [0,\infty]$ as 
\begin{align}
\label{completenergy-lin}
{\mathcal{F}}^{\rm lin}_{\varepsilon,\delta}[\omega]( u ,A)= 
 \int_{A} \frac{1}{\delta^{2}} W\Big(\omega,\frac{x}{\varepsilon}, \mathbb{I} + \delta^\alpha  \nabla u(x)\Big)+   \delta^{2(\alpha - \beta)} |\nabla^2 u(x)|^2\, \mathrm{d}x + \int_{A}f\Big(\omega,\frac{x}{\varepsilon},[u](x),\nu_u(x)\Big)\, \mathrm{d}\mathcal{H}^{d-1}
\end{align}
 whenever 
\begin{align}\label{for sure needed}
u _{|_A} \in GSBV_2^2(A;\mathbb{R}^d), \ J_{\nabla u} \cap A \subset J_u \cap A,  \ |\nabla u| \le \delta^{ -\alpha/4}  \ \text{a.e.\ in } A,
\end{align}
 and ${\mathcal{F}}^{\rm lin}_{\varepsilon,\delta}[\omega](u,A) = + \infty$ else.  Given $(\delta_\eps)_\eps$, we write ${\mathcal{F}}^{\rm lin}_{\varepsilon}[\omega] = {\mathcal{F}}^{\rm lin}_{\varepsilon,\delta_\eps}[\omega] $ for all $\omega \in \Omega'$. We now formulate our main result.  For technical reasons, we need a further assumption, namely that in the setting of Theorem \ref{gammaconv} for $L = \R^{d \times d}_{\rm skew}$ there exists  $\kappa>0$ such that  for each $\omega \in \Omega'$,  $U  \in  \mathcal{A}_0 $, and $u \in PR_L(U)$ there exists a recovery sequence $(u_{\varepsilon})_{\varepsilon} \subset PR_L(U)$ with controlled derivatives, i.e.\  $u_\eps \to u$  in measure on $U$, and  
 \begin{align}\label{Condi}
 \sup_{\varepsilon}\Vert \varepsilon^{1+\kappa}\nabla u_{\varepsilon}\Vert_{L^{\infty}(U; \R^{d\times d} )}<\infty, \quad \lim_{\varepsilon \to 0}\mathcal{E}_{\varepsilon}[\omega](u_{\varepsilon},U)=\mathcal{E}_{\mathrm{hom}}(u,U),  
 \end{align}
 see also Remark \ref{a remark} below.

\begin{theorem}[Homogenisation for asymptotically piecewise rigid functions: linearisation]
\label{brittlemateriallimit2}
Let $L=\mathbb{R}_{\rm skew}^{d \times d}$.  Let $ \alpha \in (0,1)$ and $\beta \in (\alpha,1)$. Suppose that \eqref{Condi} holds for $\kappa >0$.  {Let  $(\delta_\eps)_\eps\subset (0,1)$ with $\delta_\eps \to 0$ and  $\eps^{1+\kappa} \delta_\eps^{-\alpha/4} \to \infty$.} \EEE Let $f$ be a stationary random surface density with respect to a group $(\tau_z)_{z \in \mathbb{Z}^d}$ (resp.\  $(\tau_z)_{z \in \mathbb{R}^d}$) of $\mathbb{P}$-preserving transformation on $(\Omega,\mathcal{I},\mathbb{P})$, let $\Omega' \in \mathcal{I}$ with $\mathbb{P}(\Omega')=1$ as well as $f_{\mathrm{hom}}$ as in Theorem \ref{HomogenizationFormula}. 
 Then, 
\begin{equation*}
{\mathcal{F}}^{\mathrm{lin}}_{\varepsilon}[\omega](\cdot,U)\:\: \text{$\Gamma$-converge}\:\:\text{to}\:\: \mathcal{E}_{\mathrm{hom}}[\omega]  (\cdot,U) \ \ \ \ \text{with respect to convergence in measure on $U$}, 
\end{equation*}
for every $\omega \in \Omega'$ and $U \in \mathcal{A}_0$, where $\mathcal{E}_{\mathrm{hom}}$ is   defined by \eqref{homogenisedenergy} in the case $L=\R^{d \times d}_{\rm skew}$.   \EEE
\end{theorem}

\begin{remark}[Comments on linearisation result]\label{a remark}
{\normalfont
(a) Condition \eqref{Condi} is of technical nature and  is   needed to ensure that derivatives of recovery sequences are close to the identity, cf.\ \eqref{for sure needed}.   It is indeed  expectable that gradients do not oscillate faster than the microscale, i.e.\ one has $\Vert \nabla u_\eps \Vert_\infty \le C/\eps$, which complies with \eqref{Condi}.  Clearly, a property of this kind does not follow from the abstract $\Gamma$-convergence result in Theorem \ref{gammaconv}. At least in the case that $f$ is not stochastic (independent of $\omega$), this condition can be verified by constructing directly a more explicit recovery sequence. We refer to Proposition \ref{prop: recov} below for details.  

(b) In Example \ref{main-example}  below we show that a condition of the form $\eps \delta_\eps^{-\alpha/4} \to \infty$ appears do be necessary since otherwise one cannot expect $f_{\mathrm{hom}}$ to be the density of the limit. This explains the (slightly stronger) assumption $\eps^{1+\kappa} \delta_\eps^{-\alpha/4} \to \infty$ on $(\delta_\eps)_\eps$ in the theorem. 
 
 (c) One can show convergence of minima and  minimisers corresponding  to the energies in  Theorems \ref{brittlemateriallimit},  \ref{brittlemateriallimit2}, in the  same spirit of Corollaries \ref{convminima} and \ref{convergencealmostminimisers}. We do not repeat the details here, but refer to Remark \ref{remarkcompactness} for a short discussion.

}
\end{remark}

The results of this subsection will be proven in Section \ref{sec: limit}. We close the section with the announced example. 
 
\begin{example}\label{main-example}
{\normalfont

For $a >0$, we consider the density $f \colon [-\frac{1}{2},\frac{1}{2})^2 \times \R^2\setminus \lbrace 0 \rbrace \times \mathbb{S}^1$ defined by  
 \begin{equation}
\label{densitycounterexample} f(x,\xi,\nu) = \begin{cases} g(\xi,\nu) & \text{for } |x_2| \le 1/4, \\  a^3 & \text{for } |x_2| >1/4, \end{cases}    
\end{equation} 
 and periodically extended to $\R^2$, where  
$$g(\xi,\nu)   =    \min\big\{ 5  +  a|\xi_1|  + |\xi_2|,   a^2 \big\}  (a|\nu_2| + |\nu_1|) \quad \text{ for all $\xi = (\xi_1,\xi_2) \in \R^2\setminus \lbrace 0 \rbrace$ and $\nu = (\nu_1,\nu_2) \in \mathbb{S}^1$}.   $$
The density complies with our assumptions $(f1)$--$(f7)$. Its highly non isotropic nature is inspired by examples of densities which are not $BD$-elliptic \cite{FriPerSol20} and thus allow for lowering the energy by introducing microstructures. In the case $L = \R^{d \times d}_{\rm skew}$ and $\eps \delta_\eps^{-\alpha/4} \to 0$, for $\bar{\xi} := e_1$ and $\bar{\nu} := e_2$, one can show that for $a$ large enough and $\eps$ small enough it holds that
$$m^{L}_{\mathcal{E}_\eps}\big(u_{0,\bar{\xi},\bar{\nu}},Q_{\rho}(0)\big) \le c_a \inf\big\{\mathcal{E}_\eps\big(u,Q_\rho(0)\big)\colon \, u\in PR_{L}(Q_\rho(0)) \colon \,  |\nabla u | \le \delta_\eps^{-\alpha/4}, \    u=u_{0,\bar{\xi},\bar{\nu}}\:\: \text{near}\:\: \partial Q_\rho(0)\big\}$$ 
for a constant $0 <c_a <1$ only depending on $a$. This indicates that in this case the density of the $\Gamma$-limit in Theorem \ref{brittlemateriallimit2} does not coincide with $f_{\rm hom}$. In turn, this illustrates that  the limit depends on the ratio of $\eps$ and $\delta_\eps$, i.e.\ in general no commutability of linearisation and homogenisation can be expected.  For the detailed computation we refer to Appendix \ref{appendix-example}.

}

\end{example}
\EEE

\section{Stochastic homogenisation of surface energies}
\label{sec 4}

For the convenience of the reader, we divide this section into two parts: we first address  the homogenisation formula in Subsections \ref{sec homo}--\ref{sec: cor}, and then the  $\Gamma$-convergence result (Theorem \ref{gammaconv}) along with Corollaries \ref{convminima}-\ref{convergencealmostminimisers}  in Subsections \ref{sec: gamma}--\ref{sec: gamma2}. As several results in the following are formulated in a deterministic setting (i.e.\ $\omega$ is fixed), given $f \in \mathcal{F}$, we also use the notation $\mathcal{E}\colon L^0(\mathbb{R}^d;\mathbb{R}^d)\times \mathcal{A}\to [0,\infty]$ to indicate 
\begin{equation}
\label{energyequation}
\mathcal{E}(u,A)=\begin{cases}
    \int_{J_u \cap A}f(x,[u](x),\nu_u(x))\,\mathrm{d}\mathcal{H}^{d-1}(x) & u_{|_A}\in PR_L(A),\\
    +\infty & \text{otherwise}.
\end{cases} 
\end{equation}
This    indeed   corresponds to the functional $\mathcal{E}$ introduced below \eqref{randomenergies}, dropping the dependence on $\omega$. In addition, in the following, to simplify the notation, we will drop the superscript $L$ in $m_{\mathcal{E}}$ (see Definition~\ref{definfimumproblem}) if no confusion arises.   
In the proofs, we will frequently make use of the following \emph{gluing property} of piecewise rigid functions which follows directly from their definition.

\begin{remark}[Gluing property]
\label{gluingproperty} {\normalfont
Consider $A,B \in \mathcal{A}$ with $A \subset \subset B$. Let  $u,v \in L^0(\mathbb{R}^d;\mathbb{R}^d)$ be such that $u_{|_A} \in PR_L(A)$, $v_{|_B} \in PR_L(B)$ and $u=v$ in a neighbourhood of $N \subset A$ of $\partial A$. Then, the function $\Tilde{u}$ defined by
\begin{equation*}
\Tilde{u}(x)=\begin{cases}
u(x)\:\: \text{if}\:\: x \in A \\
v(x)\:\:  \text{if}\:\: x \in \mathbb{R}^d \setminus A,
\end{cases}    
\end{equation*}
satisfies the following properties: $\Tilde{u}\in L^0(\mathbb{R}^d;\mathbb{R}^d)$, $\Tilde{u}_{|_B}\in PR_L(B)$ and $\Tilde{u}=v$ in a neighbourhood $\Tilde{N} \subset B$ of $\partial B$. }
\end{remark}

\subsection{Homogenisation formula}\label{sec homo}

As observed before, to prove  Theorem \ref{HomogenizationFormula} we only need to treat the case $L = \R^{d \times d}_{\rm skew}$ as for $L=SO(d)$ the result can be deduced directly from  \cite[Theorem 3.12]{cagnetti2017stochastic}. The statement of  Theorem \ref{HomogenizationFormula} follows from the following two results.  
As a first ingredient, we get  that  the limes inferior and superior of the asymptotic cell formulas for $r(t) = t$ lie in $\mathcal{F}$, see \eqref{FFF}, and enjoy some continuity properties.  To formulate this, we define the sets $\hat{\mathbb{S}}_{\pm}^{d-1}:=\{ x \in \mathbb{S}^{d-1}: \pm x_{i(x)}>0\}$, where $i(x)$ is the largest $i \in \{1,...,d\}$ such that $x_i \neq 0$. Recall also the notation in \eqref{eq: Qnot}.
 \begin{lemma}
 \label{lemma2}
Let $L=\mathbb{R}_{\rm skew}^{d \times d}$. Let $f \in \mathcal{F}$, let $\mathcal{E}$ be as in \eqref{energyequation}, and let $m_{\mathcal{E}}$ as in Definition \ref{definfimumproblem}. Let $f',f''\colon \mathbb{R}^d \times \mathbb{R}^{d}\setminus\{0\}\times \mathbb{S}^{d-1}\to [ 0 ,  +\infty]$ be the functions defined by 
\begin{equation}\label{f1}
f'(x,\zeta,\nu)=\liminf\limits_{t\to \infty}\frac{m_{\mathcal{E}}(u_{tx,\zeta,\nu},Q_{t}^{\nu}(tx))}{t^{d-1}},    
\end{equation}
\begin{equation}\label{f2}
f''(x,\zeta,\nu)=\limsup\limits_{t \to \infty}\frac{m_{\mathcal{E}}(u_{tx,\zeta,\nu},Q_{t}^{\nu}(tx))}{t^{d-1}},
\end{equation}
where $u_{x,\zeta,\nu}$ is defined in \eqref{Lw}. Then $f', f'' \in \mathcal{F}$. Moreover, for every $x \in \mathbb{R}^d$ and $\zeta \in \mathbb{R}^d \setminus \{0\}$ the restriction of the functions $\nu \to f'(x,\zeta,\nu)$ and $\nu \to f''(x,\zeta,\nu)$ to the sets $\hat{\mathbb{S}}_{+}^{d-1}$ and $\hat{\mathbb{S}}_{-}^{d-1}$ are continuous. 
\end{lemma}

The second crucial ingredient is that Theorem \ref{HomogenizationFormula} holds in the special case $x= 0$ and $r(t) = t$.

\begin{theorem}[Homogenisation formula for $x=0$]
\label{corollary4.5}
Let $L=\mathbb{R}_{\rm skew}^{d \times d}$. Let $(\Omega,\hat{\mathcal{I}}, \hat{\mathbb{P}})$ be the completion of the probability space $(\Omega,\mathcal{I},\mathbb{P})$. Let $\mathcal{E}$ be a random surface energy and let, for every $\omega \in \Omega$, $m_{\mathcal{E}[\omega]}$ be as in Definition \ref{definfimumproblem}. Let $\hat{\Omega}$ be the set of all $\omega \in \Omega$ such that the limit
\begin{equation*}
\lim\limits_{\substack{t \to \infty  \\ t \in \mathbb{Q}}}\frac{m_{\mathcal{E}[\omega]}(u_{0,\zeta,\nu},Q_t^{\nu}(0))}{t^{d-1}} 
\end{equation*}
exists for every $\zeta \in \mathbb{Q}^{d} \setminus \{0\}$ and $\nu \in \mathbb{S}^{d-1}\cap \mathbb{Q}^{d}$. Then, there exists $\Tilde{\Omega} \in \mathcal{I}$, with $\Tilde{\Omega}\subset \hat{\Omega}$ and $\mathbb{P}(\Tilde{\Omega})=1$, and a random surface density $f_{\mathrm{hom}}\colon \Omega \times \mathbb{R}^d \setminus \{0\} \times \mathbb{S}^{d-1}\to \mathbb{R}$ such that
\begin{equation}
\label{eq4.28}
f_{\rm hom}(\omega,\zeta,\nu)=\lim\limits_{t \to \infty}\frac{m_{\mathcal{E}[\omega]}(u_{0,\zeta,\nu},Q_t^{\nu}(0))}{t^{d-1}} 
\end{equation}
for every $\omega \in \Tilde{\Omega}$, $\zeta \in  \mathbb{R}^{d}\setminus \{0\}$, and $\nu \in \mathbb{S}^{d-1}$.
\end{theorem}

We postpone the proofs of Lemma \ref{lemma2} and Theorem \ref{corollary4.5} to the next subsections, and briefly sketch  how the two results imply Theorem \ref{HomogenizationFormula}.  
\begin{proof}[Proof of Theorem \ref{HomogenizationFormula}]
First, we define $f_{\mathrm{hom}}\colon \Omega \times \mathbb{R}^d \setminus \{0\} \times \mathbb{S}^{d-1}\to \mathbb{R}$ as in Theorem~\ref{corollary4.5}.  Note that, because of Lemma \ref{lemma2},  $f_{\rm hom}$ is a random surface density and that \eqref{eq4.28} holds for all $\omega \in \Tilde{\Omega}$. The fact that   \eqref{eqhom} holds for general $x$ and for general function $r\colon (0,\infty)\to (0,\infty)$ with $r(t)\geq t$ is intricate, but it does not rely on the precise structure of the functionals, but rather on the properties $(f1)$--$(f6)$ and general techniques from probability theory, in particular Birkhoff's Ergodic Theorem \cite[Theorem~2.1.5]{krengel2011ergodic}   and the Conditional Dominated Convergence Theorem \cite[Section 9.7]{williams_1991}.    Indeed, we can verbatim follow   \cite[Theorem 6.1]{cagnetti2017stochastic},  by using  Theorem \ref{corollary4.5} in place of \cite[Theorem 5.1]{cagnetti2017stochastic} to show  that there exists a set $\Omega' \in \mathcal{I}$ with $\Omega' \subset \Tilde{\Omega}$  and $\mathbb{P}(\Omega')=1$ such that
\begin{align}\label{YYY}
f_{\rm hom}(\omega,\xi,\nu) = \lim\limits_{t\to \infty}\frac{m_{\mathcal{E}}(u_{tx,\zeta,\nu},Q_{t}^{\nu}(tx))}{t^{d-1}}, 
\end{align}
  for every $\omega \in \Omega'$, $x \in \mathbb{R}^d$, $\zeta \in \mathbb{Q}^{d}\setminus \{0\}$ and $\nu \in \mathbb{S}^{d-1} \cap \mathbb{Q}^d$.
  In these arguments, one makes also use of the gluing property in Remark \ref{gluingproperty}.   In a final step, we use \eqref{YYY} and Lemma \ref{lemma2} to show that \eqref{eqhom} holds for all $\zeta \in \mathbb{R}^{d}\setminus \{0\}$ and $\nu \in \mathbb{S}^{d-1}$, see \cite[Equations (5.16)--(5.17)]{cagnetti2017stochastic} for details. In fact, by  Lemma \ref{lemma2} we have continuity in $\zeta$ (see $(f2)$) and in $\nu$, in the sense described in  Lemma \ref{lemma2}. Eventually, in case that $(\tau_z)_{z \in \mathbb{Z}^d}$ (resp.\  $(\tau_z)_{z \in \mathbb{R}^d})$ is {ergodic}, we derive that $f_{\rm hom}$ does not depend on $\omega$ by repeating the proof of \cite[Corollary 6.3]{cagnetti2017stochastic}.  
\end{proof}

\subsection{Proof of Lemma \ref{lemma2}}

\EEE
This short subsection is devoted to the proof of Lemma \ref{lemma2}. 
\begin{proof}[Proof of Lemma \ref{lemma2}]
The proof is an adaptation of the one in \cite[Lemma~A.7]{cagnetti2018gammaconvergence} and  in \cite[Lemma 5.5]{cagnetti2017stochastic}. We only  highlight the necessary changes due to the setting of piecewise rigid functions $L=\R^{d \times d}_{\rm skew}$. 
 
We start with $(f2)$ for which we at least give the main idea. Fix $x \in \Omega$, $\nu \in \mathbb{S}^{d-1}$, and $\zeta_1, \zeta_2 \in \R^d \setminus \lbrace 0 \rbrace$. Given $\varepsilon>0$, we can consider $u_1 \in PR_L(Q_t^{\nu}(tx))$ such that $u_1=u_{tx,\zeta_1,\nu}$ in a neighbourhood of $\partial Q_t^{\nu}(tx)$ and 
\begin{equation*}
\mathcal{E}(u_1,Q_t^{\nu}(tx))\leq m_{\mathcal{E}}(u_{tx,\zeta_1,\nu},Q_t^{\nu}(tx))+\varepsilon t^{d-1}.  
\end{equation*}
Define $E=\{y \in Q^{\nu}_t(tx): u_{1}(y)=\zeta_1\}$.   By construction  $J_{\chi_E} \subset J_{u_1}$, and  $E$ is a set of finite perimeter in $Q_t^{\nu}(tx)$ with $\mathcal{H}^{d-1}(\partial^* E \cap Q_t^{\nu}(tx))\leq \mathcal{H}^{d-1}(J_{u_1}\cap Q_t^{\nu}(tx))<\infty$. (Here and in the following, inclusions are understood up to $\mathcal{H}^{d-1}$-negligible sets.) Define $u_2=u_1+(\zeta_2-\zeta_1)\chi_E$. Since $J_{\chi_E} \subset J_{u_1}$, we have $J_{u_2} \subset J_{u_1}$ and that $u_2$ is an admissible competitor for $m_{\mathcal{E}}(u_{tx,\zeta_2,\nu},Q_t^{\nu}(tx))$. From this and $(f2)$, one can deduce 
\begin{align*}
m_{\mathcal{E}}(w_{tx,\zeta_2,\nu},Q_t^{\nu}(tx))&  \le \mathcal{E}(u_2,Q_t^{\nu}(tx))  \le  \mathcal{E}(u_1,Q_t^{\nu}(tx)) +  \sigma(|\zeta_1 - \zeta_2|) \big(  \mathcal{E}(u_1,Q_t^{\nu}(tx))  +  \mathcal{E}(u_2,Q_t^{\nu}(tx))  \big) \\&
\le    m_{\mathcal{E}}(u_{tx,\zeta_1,\nu},Q_t^{\nu}(tx))+\varepsilon t^{d-1} + \sigma(|\zeta_1 - \zeta_2|) \big(  \mathcal{E}(u_1,Q_t^{\nu}(tx))  +  \mathcal{E}(u_2,Q_t^{\nu}(tx))  \big). 
\end{align*}
Then, $(f2)$ for $f'$ and $f''$ defined in \eqref{f1}--\eqref{f2} follows after multiplying with $t^{1-d}$ and sending $t \to \infty$, $\eps \to 0$, along with interchanging the roles of $\zeta_1$ and $\zeta_2$, see \cite[Lemma A.7]{cagnetti2018gammaconvergence} for details.

The proof of properties $(f3)$, $(f4)$, $(f6)$, and $(f7)$ is similar, and we can follow the arguments in \cite[Lemma A.7]{cagnetti2018gammaconvergence} by  defining suitable competitors. (For $(f3)$--$(f4)$ a rotation and dilation is performed and for $(f6)$ one uses $u_{tx,\zeta,\nu}$ itself as a competitor.)  Concerning the measurability property  $(f1)$, it is straightforward to verify that lemmas analogous to  \cite[Lemmas A.3--Lemma A.5]{cagnetti2018gammaconvergence} hold also when the space of piecewise constant functions is replaced with $PR_L(Q^{\nu}_t(tx))$ and when our growth condition on $f$ are imposed,  as the arguments do not really depend on the  space of competitors. (Note that the assumptions only differ in $(f6)$ which is slightly more restrictive in our case. Moreover, we note that one makes frequently use of the gluing property stated in Remark \ref{gluingproperty}.)

The adaptation of property $(f5)$ is slightly less obvious. It is based on reducing the problem to piecewise constant functions $PR_0$ as follows: given $\eps>0$, let $u \in PR_L(Q^{\nu}_{tx}(tx))$ be such that $u=u_{tx,\zeta,\nu}$ in a neighbourhood of $\partial Q^{\nu}_{tx}(tx)$ and $\mathcal{E}(u,Q_t^{\nu}(tx))\leq m_{\mathcal{E}}(u_{tx,\zeta,\nu},Q_{t}^{\nu}(tx))+\varepsilon t^{d-1}$.  We recall that each piecewise rigid function defined on some $A \in \mathcal{A}$ admits a \emph{piecewise dinstinct} representation, i.e.\ can be represented in such a way that $\mathcal{H}^{d-1}(J_u \triangle (\bigcup_j \partial^* P_j \cap A)) = 0$, where  $( P_j)_j$  is  the associated Caccioppoli partition.   For this reason, we can then find a function $\Tilde{u} \in PR_0(Q_t^{\nu}(tx))$   such that $\Tilde{u}=u_{tx,\zeta,\nu}$ in a neighbourhood of $\partial Q^{\nu}_{t}(tx)$ and $\mathcal{H}^{d-1}(J_u \triangle J_{\Tilde{u}})=0$.   So, using the fact that $f$ satisfies property $(f5)$, we have $\varepsilon t^{d-1}+m_{\mathcal{E}}(u_{tx,\zeta,\nu},Q_{t}^{\nu}(tx))\geq \mathcal{E}(u,Q^{\nu}_{t}(tx))\geq c_1 \mathcal{H}^{d-1}(J_{\Tilde{u}})$. Property $(f5)$ for $f'$ and $f''$ then follows by  \cite[Theorem 3.108]{ambrosio2000fbv}, since  for every   $\Tilde{u} \in  PR_0(Q_{t}^{\nu}(tx)) $, agreeing with $u_{tx,\zeta,\nu}$ on a neighbourhood of $\partial Q_{t}^{\nu}(tx)$, each straight line intersecting $Q_t^{\nu}(tx)$ and parallel to $\nu$ meets $J_{\Tilde{u}}$. 

The remaining part of the proof, regarding the continuity of the functions $\nu \to f'(x,\zeta,\nu)$ and $\nu \to f''(x,\zeta,\nu)$ restricted to   $\hat{\mathbb{S}}_{+}^{d-1}$ and $\hat{\mathbb{S}}_{-}^{d-1}$, follows by arguing like in \cite[Lemma 5.5]{cagnetti2017stochastic}. In fact, the proof does not rely on the exact nature of the competitors except for the fact that ``gluing'' a competitor with the boundary datum must still define a competitor on a larger set, cf.\ Remark~\ref{gluingproperty}.
\end{proof}

\subsection{Truncation results}\label{sec: trunc}

In this subsection, we derive a truncation result which will be vital in the proof of Theorem \ref{corollary4.5}. In particular, we show how to pass to a truncated version of the minimisation problem in Definition \ref{definfimumproblem}.

\begin{lemma}
\label{lemma1}
Let $L=\mathbb{R}_{\rm skew}^{d \times d}$. Let $f \in \mathcal{F}$ and let $\mathcal{E}\colon L^0(\mathbb{R}^d; \mathbb{R}^d) \times \mathcal{A}\to [0,\infty]$ be the corresponding functional as in \eqref{energyequation}.
Let $U \in \mathcal{A}_0$ and let $v \in L^{0}(\mathbb{R}^d;\mathbb{R}^d)$ be such that $v \in PR_L(U)\cap L^{\infty}(U;\mathbb{R}^d)$ and $\nabla v \in L^{\infty}(U;L)$. Then, given $m_{\mathcal{E}}$ as in Definition \ref{definfimumproblem}, we have 
\begin{equation}
\label{lim}
m_{\mathcal{E}}(v,U)=\lim\limits_{k \to \infty}m_{\mathcal{E}}^{k}(v,U),    
\end{equation}
where
\begin{equation}
\label{defm^kn}
\begin{split}
m^{k}_{\mathcal{E}}(v,U) &:= \inf\Big\{\mathcal{E}(u,U): u \in PR_L(U)\cap SBV(U;\mathbb{R}^d),  \  \  u=v\:\: \text{near}\:\: \partial U,     \\  & \ \ \ \  \quad \quad \quad \quad \quad \quad \quad \quad \quad \quad \quad \quad   \Vert u \Vert_{L^{\infty}(U;\mathbb{R}^d)} \le k \:\: \text{and}\:\:   \Vert \nabla u \Vert_{L^{\infty}(U;L)} \le  k \Big\}.
\end{split}
\end{equation}
\end{lemma}

\begin{proof}
We construct a minimising sequence $(u_n)_n$ for $m_{\mathcal{E}}(v,U)$ such that $(u_n)_n \subset L^{\infty}(U;\mathbb{R}^d)$ and $(\nabla u_n)_n \subset L^{\infty}(U;L)$. Let $(\hat{u}_n)_n$ be a minimising sequence for $m_{\mathcal{E}}(v,U)$. Let $(C_m)_m$ be an increasing sequence with $C_m \to + \infty$ as $m \to \infty$ such that $\Vert v \Vert_{L^{\infty}(U;\mathbb{R}^d)}+ \Vert \nabla v\Vert _{L^{\infty}(U;L)} \le C_m $ for every $m \in \mathbb{N}$. Fix $n \in \mathbb{N}$ and write $\hat{u}_n= \sum_{j \in \mathbb{N}}q^n_j \chi_{P_j^n}$ for suitable affine mappings $(q^n_j)_j$ and a Caccioppoli partition $(P_j^n)_j$. Moreover, let $(P_j^v)_j$  be a Caccioppoli partition  corresponding to a representation of $v$ in Definition \ref{def: PR}. It is not restrictive to assume  that $(P^n_j)_j$ is a refinement of $(P_j^v)_j$. (This can be achieved by representing $\hat{u}_n$ with the partition $(P^n_j \cap P_k^{v})_{j,k}$. For notational simplicity, we denote this partition still by $(P^n_j)_j$.) Then, we can also write $v=\sum_{j \in \mathbb{N}}r^{n}_j\chi_{P_j^{n}}$ for suitable affine mappings $(r^{n}_j)_j$. Define 
$$I^m_n :=\big\{j\in \mathbb{N}: \Vert q^n_j \Vert_{L^{\infty}(P_j^n;\mathbb{R}^d)}\leq C_m \:\: \text{and}\:\:\Vert \nabla q_j^n \Vert_{L^{\infty}(P_j^n;L)}\leq C_m\big\}$$   and ${u}^m_n$ as 
\begin{equation*}
{u}^m_n(x)=\begin{cases}
    q_j^n(x) &\text{if }\:\: x \in P_j^n \text{ and }   j  \in I^m_n, \\
    v(x) & \text{otherwise}.
\end{cases}
\end{equation*}
 In  particular, we can write ${u}^m_n=\sum_{j \in I^m_n} q^n_j \chi_{P^n_j}+\sum_{j \in \mathbb{N}\setminus I^m_n}r_j^n \chi_{P^n_j}$. Thus, ${u}_n^m$ is still a piecewise rigid function and it satisfies $\Vert  {u}^m_n \EEE \Vert _{L^{\infty}(U;\mathbb{R}^d)}\leq C_m$ and $\Vert \nabla   {u}_n^m \EEE \Vert_{L^{\infty}(U;L)} \le C_m$. Moreover, observe that by construction ${u}_n^m=v$ near $\partial U$ and therefore ${u}_n^m$ is admissible for $m_{\mathcal{E}}(v,U)$.  

It is straightforward to verify that  $I^m_n \subset I^{m+1}_n$ for every $m \in \mathbb{N}$. Consequently, for each $\varepsilon>0$ and $n \in \N$ there exists an $m(v,\varepsilon,n)$ such that, for every $m\geq m(v,\varepsilon,n)$, it holds
\begin{equation*}
 \sum_{j \in \mathbb{N} \setminus I^m_n} \mathcal{H}^{d-1}(\partial^* P^n_j) \le \varepsilon.   
\end{equation*}
Then, using $(f6)$ we can compare the energy of $\hat{u}_n$ and the energy of $u_n^m$, and we get that for every $m\geq m(v,\varepsilon,n)$ it holds
\begin{equation}
\label{correction}
\mathcal{E}(u^m_n,U)\leq \mathcal{E}(\hat{u}_n,U)+c_2 \sum_{j \in \mathbb{N} \setminus I^m_n} \mathcal{H}^{d-1}(\partial^* P^n_j)\leq \mathcal{E}(\hat{u}_n,U)+c_2 \varepsilon.
\end{equation}
Define $u_n:=   {u}_n^{m\big(v,\frac{1}{n},n\big)}  $   and observe that for every $n \in \mathbb{N}$ there exists $k(n)\in \mathbb{N}$ such that $\Vert u_n\Vert _{L^{\infty}(U;\mathbb{R}^d)}+\Vert \nabla u_n \Vert_{L^{\infty}(U;L)} \le k(n)$. In addition, by virtue of \cite[Theorem 2.7]{friedrich2018piecewise} (see also \cite[Theorem 2.2]{ConFocIur15}) notice that $PR_L(U)\cap L^{\infty}(U;\mathbb{R}^d)\subset SBD^2(U)\cap L^{\infty}(U;\mathbb{R}^d)\subset SBV(U;\mathbb{R}^d)$ and so $u_n$ is admissible for $m_{\mathcal{E}}^{k(n)}(v,U)$ for every $n \in \mathbb{N}$, cf.\ \eqref{defm^kn}. (Here, $SBD^2$ denotes the space of $SBD$ functions with $e(u) \in L^2(A;\mathbb{R}^{d \times d})$ and $\mathcal{H}^{d-1}(J_u)<\infty$, see e.g.\   \cite{Temam1980FunctionsOB, Dal11}. Since $(\hat{u}_n)_n$ is a minimising sequence for $m_{\mathcal{E}}(v,U)$ and $k \mapsto m^{k}_{\mathcal{E}}(v,U)$ is a monotone decreasing function,    \eqref{correction} (with $\eps = \frac{1}{n}$) implies \EEE
\begin{align*}
m_{\mathcal{E}}(  v ,  U)&=\lim\limits_{n \to \infty}\mathcal{E}(\hat{u}_n,U)\geq  \liminf\limits_{n \to \infty}  \mathcal{E}({u}_n,U)\geq \liminf\limits_{n \to \infty}   m^{k(n)}_{\mathcal{E}}(v,U)\geq  \lim\limits_{k \to \infty} m^{k}_{\mathcal{E}}(v,U)   \geq m_{ \mathcal{E}}(v,U).
\end{align*}
This yields \eqref{lim} and concludes the proof. 
\end{proof}  
\EEE

\subsection{Proof of Theorem \ref{corollary4.5}}\label{sec: cor}

The proof is based on the application of the  Subadditive Ergodic Theorem by {\sc Akcoglu and Krengel} \cite{Krengel1981} for  subadditive stochastic processes. Let us first introduce the main notions.   For any positive integer $k$ and for any $a,b \in \mathbb{R}^k$, with $a_i<b_i$ for all $i=1,...,k$, we define the cuboids  
\begin{equation*}
    [a,b)=\prod^{k}_{i=1}[a_i,b_i)=\{x \in \mathbb{R}^k: a_i \leq x_i < b_i \:\: \forall \, i=1,...,k\},
\end{equation*}
and 
\begin{equation*}
\mathcal{R}_k = \{[a,b): a_i<b_i \:\: \forall \, i=1,...,k\}.
\end{equation*}
 
\begin{definition}{(Subadditive process)}
\label{subadditivedef}
A \emph{subadditive process} with respect to a group $(\tau_{z})_{z \in \mathbb{Z}^k}$ (resp.\  $(\tau_{z})_{z \in \mathbb{R}^k}$) of $\mathbb{P}$-preserving transformations on $(\Omega, \mathcal{I}, \mathbb{P})$ is a function $\mu \colon \Omega \times \mathcal{R}_k  \to [0,\infty)$ satisfying:\\
\begin{enumerate}
    \item[$(\mathrm{i})$] (measurability) for any $R \in \mathcal{R}_k$ the function $\omega \to \mu(\omega,R)$ is $\mathcal{I}$ measurable, 
    \item[$(\mathrm{ii})$] (covariance) for any $\omega \in \Omega$, $R \in \mathcal{R}_k$, $z\in \mathbb{Z}^k$ (resp.\  $z\in \mathbb{R}^k$) it holds $\mu(\tau_z \omega,R)=\mu(\omega,z+R)$,
    \item[$(\mathrm{iii})$] (subadditivity) for any $R \in \mathcal{R}_k$ and for any finite family $(R_i)_{i \in I} \subset \mathcal{R}_k$ of pairwise disjoint sets  such that $\bigcup_{i\in I}R_i=R$, it holds
\begin{equation*}
    \mu(\omega,R)\leq \sum_{i \in I}\mu(\omega,R_i)\:\: \text{for every}\:\: \omega \in \Omega,
\end{equation*}
\item[$(\mathrm{iv})$] (uniform boundedness) there exists a constant $c>0$ such that $0\leq \mu(\omega,R) \leq c\,\mathcal{L}^k (R)$ for every $\omega \in \Omega$ and for every $R \in \mathcal{R}_k$.
\end{enumerate}
\end{definition}

\begin{theorem}[Subadditive Ergodic Theorem]\label{th: SET}
Let $k \in \N$ and let $(\tau_z)_{z \in \mathbb{Z}^d}$ (resp.\  $(\tau_z)_{z \in \mathbb{R}^d}$) be a group of $\mathbb{P}$-preserving transformations on $(\Omega,\mathcal{I},\mathbb{P})$. Let $\mu \colon \Omega \times \mathcal{R}_k \to \R $ be a subadditive process with respect to  $(\tau_z)_{z \in \mathbb{Z}^d}$ (resp.\  $(\tau_z)_{z \in \mathbb{R}^d}$). Then, there exists an  $\mathcal{I}$-measurable   function $\varphi\colon \Omega \to [0,\infty)$ and a set $\Omega' \in \mathcal{I}$ with $\mathbb{P}(\Omega') = 1$ such that 
\begin{equation}
\label{spaceaverage}
\lim_{t \to \infty} \frac{\mu(\omega,tQ)}{\mathcal{L}^k(tQ)} = \varphi(\omega)    
\end{equation}

for every $\omega \in \Omega'$ and for every   $Q \in \mathcal{R}_k$. If in addition  $(\tau_z)_{z \in \mathbb{Z}^d}$ (resp.\  $(\tau_z)_{z \in \mathbb{R}^d}$) is ergodic, then $\varphi$ is constant $\mathbb{P}$-a.e.
\end{theorem} 

  The notion of subadditive  processes  was introduced by {\sc Hammersley and Welsh} \cite{Hammersley1965FirstPassagePS} and their crucial property is that their space average in the sense of \eqref{spaceaverage} exists almost surely.    For a proof we refer to \cite[Theorem 2.7 and Remark p.~59]{Krengel1981} and \cite[Theorem 3.11]{cagnetti2017stochastic}. Our goal is to apply the Subadditive Ergodic Theorem for $k=d-1$. This is delicate as one needs to construct a $(d-1)$-dimensional  subadditive process starting from the $d$-dimensional set function   $U \in \mathcal{A}_0 \mapsto m_{\mathcal{E}[\omega]}(u_{0,\zeta,\nu},U)$. This issue has been solved in \cite{cagnetti2017stochastic} by suitably passing from $(d-1)$-dimensional to  $d$-dimensional cuboids.  To explain this, we introduce further notation: recall \eqref{eq: Qnot}. Since for $\nu \in \mathbb{S}^{d-1}\cap \mathbb{Q}^{d}$ we have that the orthogonal matrix $R_{\nu}$ lies in $\mathbb{Q}^{d \times d}$, see  \cite[Remark A.2]{cagnetti2018gammaconvergence}, there exists a positive integer $M_{\nu} \in \mathbb{N}$ such that $M_{\nu}R_{\nu} \in \mathbb{Z}^{d \times d}$.
Then, given $A'=\prod^{d-1}_{j=1}[a_j,b_j) \in \mathcal{R}_{d-1}$, for every $\nu \in \mathbb{S}^{d-1}\cap \mathbb{Q}^{d}$, we define the rotated $d$-dimensional cuboid $T_{\nu}(A')$ as
\begin{equation}
\label{Tnu}
\begin{split}
T_{\nu}(A'):= M_{\nu}R_{\nu}(A' \times [-c,c)),  \quad \text{where } c:=\frac{1}{2}\max_{1\leq j\leq d-1}(b_j-a_j).
\end{split}
\end{equation}

\begin{proposition}\label{subadditiveprocess}
Let $L=\mathbb{R}_{\rm skew}^{d \times d}$. Let $f$ be a stationary random surface density with respect to a group  $(\tau_z)_{z \in \mathbb{Z}^d}$ (resp.\  $(\tau_z)_{z \in \mathbb{R}^d}$) of $\mathbb{P}$-preserving transformations on $(\Omega,\mathcal{I},\mathbb{P})$, and let $\mathcal{E}$ be the corresponding random surface energy. Let $\zeta\in \mathbb{Q}^d \setminus \{0\}$ and let $\nu \in \mathbb{S}^{d-1}\cap \mathbb{Q}^d$. For every   cuboid   $A' \in \mathcal{R}_{d-1}$ and for every $\omega \in \Omega$, set
\begin{equation}\label{eq: the process}
\mu_{\zeta,\nu}(\omega,A'):= \frac{1}{M^{d-1}_{\nu}}m_{\mathcal{E}[\omega]}(u_{0,\zeta,\nu},T_{\nu}(A')),
\end{equation}
where $m_{\mathcal{E}[\omega]}$ is given in Definition \ref{definfimumproblem},  $M_{\nu}$ and $T_{\nu}$ are defined as in  \eqref{Tnu}, and $u_{0,\zeta,\nu}$ is given in \eqref{Lw}. Let $(\Omega,\hat{\mathcal{I}},\hat{\mathbb{P}})$ be the completion of the probability space $(\Omega,\mathcal{I},\mathbb{P})$. \\ Then, there exists a group $(\tau_{z'}^{\nu})_{z' \in \mathbb{Z}^{d-1}}$ (resp.\ $(\tau_{z'}^{\nu})_{z' \in \mathbb{R}^{d-1}}$) of $\hat{\mathbb{P}}$-preserving transformations on $(\Omega,\hat{\mathcal{I}},\hat{\mathbb{P}})$ such that $\mu_{\zeta,\nu} \colon \Omega \times \mathcal{R}_{d-1} \to [0,\infty)$  is a subadditive process on $(\Omega,\hat{\mathcal{I}},\hat{\mathbb{P}})$ with respect to $(\tau_{z'}^{\nu})_{z' \in \mathbb{Z}^{d-1}}$ (resp.\  $(\tau_{z'}^{\nu})_{z' \in \mathbb{R}^{d-1}}$). Moreover, for $\hat{\mathbb{P}}$-almost every $\omega \in \Omega$ and all $A' \in \mathcal{R}_{d-1}$ it holds  
\begin{equation}
\label{uniformboundedness}
0\leq \mu_{\zeta,\nu}(\omega,A')\leq c_2 \mathcal{L}^{d-1}(A').
\end{equation}
\end{proposition}

 We postpone the proof to the end of the section and proceed with the proof of Theorem \ref{corollary4.5}.

\begin{proof}[Proof of Theorem \ref{corollary4.5}]
 We follow the proof in  \cite[Theorem 5.1]{cagnetti2017stochastic}  and sketch the main steps for convenience of the reader.  Let $(\Omega,\hat{\mathcal{I}}, \hat{\mathbb{P}})$ be the completion of the probability space $(\Omega,\mathcal{I},\mathbb{P})$. Let $\zeta\in \mathbb{Q}^d \setminus \{0\}$ and let $\nu \in \mathbb{S}^{d-1}\cap \mathbb{Q}^d$. We apply Theorem \ref{th: SET} on the subadditive stochastic process $\mu_{\zeta,\nu}$ on $(\Omega,\hat{\mathcal{I}},\hat{\mathbb{P}})$ with respect to $(\tau_{z'}^{\nu})_{z' \in \mathbb{Z}^{d-1}}$, see \eqref{eq: the process}. We find  $\hat{\Omega}_{\zeta,\nu} \in \hat{\mathcal{I}}$ with $\hat{\mathbb{P}}(\hat{\Omega}_{\zeta,\nu}) = 1$, and an $\hat{\mathcal{I}}$ measurable function $g_{\zeta,\nu}\colon \Omega \to \R$ such that
\begin{align}\label{5.13}
\lim_{t \to \infty} \frac{\mu_{\zeta,\nu}(tQ')}{t^{d-1}} = g_{\zeta,\nu}(\omega) 
\end{align}
for every $\omega \in \hat{\Omega}_{\zeta,\nu}$, where we set $Q' := [0,1)^{d-1}$. By properties of the completion we then find $\Omega_{\zeta,\nu} \in \mathcal{I}$ with $\mathbb{P}(\Omega_{\zeta,\nu}) = 1$, and an $\mathcal{I}$ measurable function, still denoted by $g_{\zeta,\nu}$, such that \eqref{5.13} holds for all $\omega \in \Omega_{\zeta,\nu}$. Let $\Tilde{\Omega} \in \mathcal{I}$ be the intersection of the sets $\Omega_{\zeta,\nu}$ for $\zeta \in \mathbb{Q}^d \setminus \lbrace 0 \rbrace$ and $\nu \in \mathbb{S}^{d-1} \cap \mathbb{Q}^d$. Note that $\mathbb{P}(\Tilde{\Omega}) = 1$ and by the definition of $\mu_{\zeta,\nu}$ we have
$$g_{\zeta,\nu}(\omega) =  \lim_{t \to \infty} \frac{m_{\mathcal{E}[\omega]}\big(u_{0,\zeta,\nu},t T_{\nu}(Q')\big)}{(tM_{\nu})^{d-1}} =  \lim_{t \to \infty} \frac{m_{\mathcal{E}[\omega]}\big(u_{0,\zeta,\nu},t M_{\nu}Q^\nu(0)\big)}{(tM_{\nu})^{d-1}} $$
for all $\omega \in \Tilde{\Omega}$,  $\zeta\in \mathbb{Q}^d \setminus \{0\}$, and $\nu \in \mathbb{S}^{d-1}\cap \mathbb{Q}^d$, where we used the definition in \eqref{Tnu} and \eqref{eq: Qnot}.  In particular, this shows that the two auxiliary functions $\underline{f}, \overline{f} \colon \Tilde{\Omega} \times \R^d \setminus \lbrace 0 \rbrace\times  \mathbb{S}^{d-1} \to [0,\infty]$ defined by 
$$\underline{f}(\omega,\zeta,\nu) := \liminf_{t \to \infty}  \frac{m_{\mathcal{E}[\omega]}(u_{0,\zeta,\nu},  Q^\nu_t(0))}{t^{d-1}}, \quad \quad  \overline{f}(\omega,\zeta,\nu) := \limsup_{t \to \infty}  \frac{m_{\mathcal{E}[\omega]}(u_{0,\zeta,\nu},  Q^\nu_t(0))}{t^{d-1}} $$
coincide on $\Tilde{\Omega} \times \mathbb{Q}^d\setminus \lbrace 0 \rbrace \times (\mathbb{Q}^d \cap \mathbb{S}^{d-1})$, where we again use the notation for $Q^\nu_t(0)$ defined in \eqref{eq: Qnot}. Now,  by the continuity of $\underline{f}$ and $\overline{f}$  in $\zeta$ and $\nu$, see  Lemma \ref{lemma2} (for $x=0$), the functions $\underline{f}$ and $\overline{f}$ coincide on $\Tilde{\Omega} \times \R^d \setminus \lbrace 0 \rbrace \times \mathbb{S}^{d-1}$. Here, we particularly use that $(f2)$ holds as $\underline{f}, \overline{f} \in \mathcal{F}$. We also refer to the proof of Theorem \ref{HomogenizationFormula} for a similar argument and to \cite[Equations (5.16)--(5.17)]{cagnetti2017stochastic} for details. The continuity in $(\zeta,\nu)$ and the measurability of $g_{\zeta,\nu} \colon \Tilde{\Omega} \to \R$ also show that $\overline{f}$  is $\mathcal{I}  \times \mathcal{B}^d  \times \mathcal{B}^d_\mathbb{S}$ measurable on $\Tilde{\Omega} \times \R^d \setminus \lbrace 0 \rbrace \times \mathbb{S}^{d-1}$.  It now suffices to set $f_{\rm hom}(\omega,\zeta,\nu) = \overline{f}(\omega,\zeta,\nu)$ for $\omega \in \Tilde{\Omega}$ and $f_{\rm hom}(\omega,\zeta,\nu) = c_2$ for $\omega \in \Omega \setminus \Tilde{\Omega}$.  As we also have $\overline{f}(\omega,\cdot,\cdot) \in \mathcal{F}$ for all $\omega \in \Tilde{\Omega}$ by   Lemma~\ref{lemma2}, we conclude that $f_{\rm hom}$ is a random surface density in the sense of Definition \ref{randomsurfaceintegrand} and that \eqref{eq4.28} holds. 
\end{proof}

We now come to the proof of Proposition \ref{subadditiveprocess}. \EEE

\begin{proof}[Proof of Proposition \ref{subadditiveprocess}]
 The most delicate part of the proof is to show the $\hat{\mathcal{I}}$ measurability of the function $\omega \to \mu_{\zeta,\nu}(\omega,A')$. This  follows from the $\hat{\mathcal{I}}$ measurability of $\omega \to m_{\mathcal{E}[\omega]}(u_{0,\zeta,\nu},U)$ for every $U \in \mathcal{A}_0$, which we postpone to Theorem \ref{measurability} after this proof. 

The remaining part of the proof follows the same steps of \cite[Proposition 5.3]{cagnetti2017stochastic}, and we only include a short sketch.  Given a discrete group $(\tau_z)_{z \in \Z^d}$, one can define a group of $\hat{\mathbb{P}}$-preserving transformations on $(\Omega,\hat{\mathcal{I}},\hat{\mathbb{P}})$  by
$$(\tau^\nu_{z'})_{z' \in \Z^{d-1}} :=  (\tau_{z'_\nu})_{z' \in \Z^{d-1}}, \quad \text{where }   z'_\nu := M_\nu R_\nu (z',0) \in \Z^d. $$
By the stationarity of $f$ and by performing a change of variables for a function $u$ and a corresponding $\Tilde{u}(x)=    u(x+z'_{\nu}) $, one can check that $\mu_{\zeta,\nu}$ is covariant with respect to $(\tau^\nu_{z'})_{z' \in \Z^{d-1}}$. Here, it is crucial that for $u=u_{0,\zeta,\nu}$ the corresponding $\Tilde{u}$ satisfies $\Tilde{u} = u_{0,\zeta,\nu}$. We refer to \cite[Equations (5.5)--(5.6)]{cagnetti2017stochastic} for details. In case of a continuous group $(\tau_z)_{z \in \R^d}$, the same argument works, even simpler, for $z'_\nu :=   R_\nu (z',0) \in \R^d. $

For the subadditivity, we consider $A' =\bigcup_{i=1}^n A_i'$ for pairwise disjoint sets $(A_i')_{1 \le i \le n} \subset \mathcal{R}_{d-1}$ and glue together almost  minimisers  for the problems in Definition \ref{definfimumproblem} (on the sets $T_\nu(A_i')$ in place of $A$) by using Remark \ref{gluingproperty}. Then, we can exactly repeat the argument below \cite[Equations (5.9)]{cagnetti2017stochastic}. 
 Eventually, \eqref{uniformboundedness} follows by taking the function $u_{0,\zeta,\nu}$ as a competitor in the problem \eqref{eq: the process} and using the upper bound in $(f6)$. In view of Definition \ref{subadditivedef}, we conclude that $\mu_{\zeta,\nu}$ is a subadditive process on $(\Omega,\hat{\mathcal{I}},\hat{\mathbb{P}})$ with respect to $(\tau_{z'}^{\nu})_{z' \in \mathbb{Z}^{d-1}}$ (resp.\  $(\tau_{z'}^{\nu})_{z' \in \mathbb{R}^{d-1}}$). 
\end{proof}

We conclude the proof of Theorem \ref{corollary4.5} with the measurability   needed in the proof of Proposition~\ref{subadditiveprocess}. 
\begin{theorem}[Measurability]
\label{measurability}   
 Let $L=\mathbb{R}_{\rm skew}^{d \times d}$. Let $(\Omega,\hat{\mathcal{I}},\hat{\mathbb{P}})$ be the completion of the probability space $(\Omega,\mathcal{I},\mathbb{P})$, let $f$ be a random surface density, and let $U \in \mathcal{A}_0$. Let $\mathcal{E}$ be the random surface energy corresponding to $f$ and define, for every $\omega \in \Omega$, $m_{\mathcal{E}[\omega]}$ according to Definition \ref{definfimumproblem}. Let $v \in {L}^{0}(\mathbb{R}^d;\mathbb{R}^d)$ be such that $  v\in PR_L(U)\cap L^{\infty}(U;\mathbb{R}^d)$ and $\nabla  v  \in L^{\infty}(U;L)$. Then, the function $\omega \to m_{\mathcal{E}[\omega]}( v, U)$ is $\hat{\mathcal{I}}$ measurable.
\end{theorem}

\begin{proof}
We  follow  the strategy of the proof of \cite[Proposition A.1]{cagnetti2017stochastic}, up to some nontrivial modifications. We refer to Remark \ref{problemmeasurability}  below for some comments on the difference of the settings of piecewise constant and piecewise rigid functions. The proof is divided in two parts:  \EEE in the first part  (corresponding to \emph{Step~1} and \emph{Step~2}), we prove that the measurability problem can be reduced to a simpler one, based on the  truncation in \eqref{defm^kn}. In the second part  (\emph{Step 3}), we prove the measurability of the simplified problem.

\emph{Step 1}: For every $k \in \mathbb{N}$ and $\omega \in \Omega$, let $m^{k}_{\mathcal{E}[\omega]}$ be as in \eqref{defm^kn} with $\mathcal{E}[\omega]$ in place of $\mathcal{E}$. Thanks to Lemma~\ref{lemma1}, we have that for every $\omega\in \Omega$, $U \in \mathcal{A}_0$, and $v \in  PR_L(U) \cap L^{\infty}(U;\mathbb{R}^d)$ such that $\nabla v \in L^{\infty}(U;L)$, it holds that
\begin{equation}
\label{eq4.10}
m_{\mathcal{E}[\omega]}(v,U)=\lim_{k \to \infty}m^{k}_{\mathcal{E}[\omega]}(v,U).  
\end{equation}  
Hence, by virtue of \eqref{eq4.10}, in order to prove that $\omega \to m_{\mathcal{E}[\omega]}( v, U)$ is $\hat{\mathcal{I}}$  measurable, it is sufficient to show that the function
\begin{equation}
\label{eq4.12}
\omega \to m^{k}_{\mathcal{E}[\omega]}(v,U)\:\: \text{is}\:\: \hat{\mathcal{I}}\:\: \text{measurable for}\:\: k \:\: \text{large enough.}    
\end{equation}
Define the set
 \begin{equation*}
\chi^{k}:=\Big\{u \in PR_L(U)\cap SBV(U;\mathbb{R}^d)\colon \,  \:\: \Vert u \Vert_{L^{\infty}(U;\mathbb{R}^d)}\leq k,\:\:\Vert \nabla u \Vert_{L^{\infty}(U;L)} \le  k \:\: \text{and}\:\: u=v\:\: \text{near}\:\: \partial U\Big\}.  
\end{equation*}   
By virtue of  $(f6)$, it holds that $\mathcal{E}[\omega](v,U) \le c_2\mathcal{H}^{d-1}(  J_v \cap U)$, and then for $k$ large enough $(f5)$ implies  
\begin{equation}
\label{controljumpset}
m_{\mathcal{E}[\omega]}^{k}(v,U)=\inf\big\{\mathcal{E}[\omega](u,U): u \in \chi^{k}\:\: \text{and}\:\: \mathcal{H}^{d-1}(J_u \cap U)\leq \gamma\big\},   
\end{equation}
where $\gamma=\frac{c_2}{c_1}\mathcal{H}^{d-1}(J_v\cap U)$.  Let $(U_j)_j$ be an increasing sequence of open sets such that $U_j \subset \subset U$ and $U_j \uparrow U$. We consider the family of subsets of $\chi^k$ formed by the functions having jump set controlled by $\gamma$ and satisfying the boundary condition on $U \setminus U_j$ i.e.
\begin{equation*}
\begin{split}
 \chi_j^{k}&:=\big\{u \in  \chi^k\colon \,  \mathcal{H}^{d-1}(J_u\cap U)\leq \gamma \ \  \text{and}\:\: u=v \:\: \text{on}\:\: U \setminus {U}_j\big\}.    
\end{split}
\end{equation*}
Then, because of \eqref{controljumpset}, it holds
\begin{equation}
\lim\limits_{j \to \infty}\inf_{u \in \chi_j^{k}}\mathcal{E}[\omega](u,U)=m_{\mathcal{E}[\omega]}^{k}(v,U).
\end{equation}
Hence, in order to prove \eqref{eq4.12}, it is sufficient to show that the function
\begin{equation}
\label{eq4.15}
\omega \to \inf_{u \in \chi_j^{k}}\mathcal{E}[\omega](u,U)\:\: \text{is}\:\: \hat{\mathcal{I}}\text{ measurable}.
\end{equation}
\EEE
The proof of \eqref{eq4.15} will be achieved by using the projection theorem, see for example \cite[Theorem~$\mathrm{III}.13$ and $33(\mathrm{a})$]{meyer1966probability}.   Below we show that $\chi_j^{k}$ equipped with the weak* convergence of $BV(U;\mathbb{R}^d)$  
is homeomorphic to a  compact (and thus  separable and complete) metric space. Given the canonical projection  $\pi_\Omega \colon \Omega \times \chi_j^{k} \to \Omega$ of $\Omega \times \chi_j^{k}$ onto $\Omega$, for every $t \in \mathbb{R}$  we have
\begin{equation*}
\Big\{ \omega \in \Omega: \inf_{u \in \chi_j^k}\mathcal{E}[\omega](u,U)<t\Big\}=\pi_\Omega\Big(\{(\omega, u)\in\Omega   \times  \chi_j^{k} : \mathcal{E}[\omega](u,U)<t\}\Big). 
\end{equation*}
Hence, the projection theorem shows that  \eqref{eq4.15} is true if the function
\begin{equation}
\label{measurabilityclaim}
(\omega, u)\to  \mathcal{E}[\omega](u,U)\:\: \text{is}\:\: \mathcal{I}\otimes \mathcal{B}(\chi_j^{k}) \text{ measurable},
\end{equation}
hence $\hat{\mathcal{I}}\otimes \mathcal{B}(\chi_j^{k})$ measurable.  In the last step of the proof, we show \eqref{measurabilityclaim} which by the above reasoning implies that $\omega \to m_{\mathcal{E}[\omega]}(v,U)$ is $\hat{\mathcal{I}}$ measurable. We close this step by briefly explaining that $\chi_j^{k}$ is homeomorphic to a compact   metric space.
First, $BV(U;\mathbb{R}^d)$ is the dual of a separable Banach space, see  \cite[Remark 3.12]{ambrosio2000fbv}, so in particular its weak* topology is metrisable on bounded subsets. Notice that $\chi_j^{k}$ is bounded with respect to the  $BV$ norm as it holds
\begin{equation}\label{eq: compi}
\Vert u\Vert _{L^1(U;\mathbb{R}^d)}+|Du|(U)\leq  k  \mathcal{L}^d({U})+2k\gamma.    
\end{equation}
It remains to show compactness: Given $(u_n)_n \subset \chi_j^{k}$, in view of \eqref{eq: compi} and $\Vert \nabla u_n\Vert _{L^{\infty}(U;L)} \le  k$, by  \cite[Theorem 4.7]{ambrosio2000fbv} there exists a subsequence (not relabeled) and $u \in SBV(U;\mathbb{R}^d)$ such that $(u_n)_n$ converges weakly* to $u$ in $BV(U;\mathbb{R}^d)$. As  $u_n \to u$ in $L^1(U;\R^d)$, we get $u=v$ on $U \setminus {U}_j$, and by lower semicontinuity we have 
$$\Vert u\Vert _{L^{\infty}(U;\mathbb{R}^d)} \le  k, \quad \quad  \Vert \nabla u\Vert _{L^{\infty}(U;L)} \le  k, \quad \mathcal{H}^{d-1}(J_u\cap U) \le \gamma. $$ 
Eventually, using  \cite[Lemma 3.3]{Friedrich_2020} for  $\psi(t)= t$, we also find that  $u \in PR_L(U)$. This shows that $\chi_j^{k}$ is compact. For later purposes, we note that the argument also shows that   
\begin{align}\label{eq_ceeonti}
\Phi \colon  \chi_j^k  \to  \mathcal{M}(U;\mathbb{R}^{d \times d}), \quad   \Phi(u) = D^su \quad \text{ is continuous}
\end{align}
for  the weak* topologies on $BV(U;\mathbb{R}^d)$ and $\mathcal{M}(U;\mathbb{R}^{d \times d})$, respectively. Indeed, given   $u_n \to u$ in $L^1(U;\mathbb{R}^d)$ and   $Du_n \rightharpoonup^* Du$  in $U$,  by  \cite[Theorem 4.7]{ambrosio2000fbv} we find   $D^su_n \rightharpoonup^* D^su$   in $U$.

\emph{Step 2}: In this step, we show that it suffices to prove  \eqref{measurabilityclaim} for a certain class of densities $f$. Observe that there exists a one-to-one correspondence between the rank one $d \times d$ matrices and the quotient of $\mathbb{R}^d \setminus \{0\} \times \mathbb{S}^{d-1}$ with respect to the equivalence relation $(\zeta,\nu) \sim (-\zeta,-\nu)$. Therefore, thanks to $(f6)$--$(f7)$,   we can define a bounded $\mathcal{I}\times \mathcal{B}(U)\times \mathcal{B}^{d \times d}$ measurable function $\Tilde{f}_k \colon \Omega \times U \times \mathbb{R}^{d \times d}\to \mathbb{R}$ such that
\begin{equation}\label{eq: ftildeXX}
    \Tilde{f}_k(\omega,x,\zeta \otimes \nu)=f(\omega,x,\zeta,\nu)
\end{equation}
for every $\omega \in \Omega$, $x \in U$, $\zeta \in \mathbb{R}^{d \times d} \setminus \{0\}$ with $|\zeta|\leq 2k$   and $\nu \in \mathbb{S}^{d-1}$. Consequently, for every $u \in \chi_j^{k}$, we have the following equivalent expression of the energy:
\begin{equation}\label{eq: new en}
\mathcal{E}[\omega](u,U)=\int_{J_u \cap U}\Tilde{f}_k\big(\omega,x,[u]\otimes \nu_u(x)\big)\,\mathrm{d}\mathcal{H}^{d-1}(x).   
\end{equation}
By a monotone class argument, it is not restrictive to assume that for every $\omega \in \Omega$ and $x \in U$ the function $ \xi  \to \Tilde{f}_k(\omega,x,  \xi )$ is continuous. In fact, let $\mathcal{H}$ be the set of functions defined by
\begin{equation*}
\begin{split}
    \mathcal{H}&:=\Big\{g \colon \Omega \times U \times \mathbb{R}^{d \times d}\to [0,\infty): \:\: g \:\: \text{is}\:\:  \mathcal{I}\otimes \mathcal{B}(U)\otimes \mathcal{B}^{d \times d} \text{ measurable, bounded and}\\  & \quad \quad \quad \text{such that}\:\: (\omega,u)\to  \int_{J_u \cap U}g(\omega,x,[u]\otimes \nu_u(x))\,\mathrm{d}\mathcal{H}^{d-1}(x)\:\: \text{is}\:\: \mathcal{I}\otimes \mathcal{B}(\chi_j^{k,l})\:\: \text{measurable}\Big\}.   
\end{split} 
\end{equation*}
It can be verified that $\mathcal{H}$ is a monotone class in the sense of \cite[Definition 4.12]{aliprantis06}. As a consequence, if $\mathcal{H}$ contains the set  
\begin{equation*}
\begin{split}
 \mathcal{C} :=\Big\{g\colon \Omega \times  U  \times \mathbb{R}^{d \times d}\to & [0,\infty):\:\:g(\omega,x, \xi  )=\varphi(\omega,x)\psi( \xi  ),  \\ & \text{with}\:\:\varphi\:\: \text{bounded and}\:\: {\mathcal{I}}\otimes \mathcal{B}(U)\text{ measurable, and}\:\: \psi\in C_c^0(\mathbb{R}^{d \times d})\Big\}.   
\end{split}
\end{equation*}
then the functional form of the  Monotone Class Theorem, see \cite[Chapter 1, Theorem 21]{meyer1966probability},   applied to $\mathcal{C}$, shows that the set $\mathcal{H}$ coincides with the class of all bounded and $\mathcal{I}\otimes \mathcal{B}(U)\otimes \mathcal{B}^{d \times d}$ measurable functions. This shows that it suffices to prove \eqref{measurabilityclaim} for functions of the form $\Tilde{f}_k(\omega,x, \xi )=\varphi(\omega,x)\psi( \xi )$ as above, in particular we can assume that  $ \xi \to \Tilde{f}_k(\omega,x,\xi)$ is continuous for every $\omega \in \Omega$ and $x \in U$.

\emph{Step 3}: In this step, we prove \eqref{measurabilityclaim} for functions $\Tilde{f}_k$ of the above form.  In particular, we follow the proof of \cite[Proposition A.1]{cagnetti2017stochastic}, up to replacing $Du$ with the singular part $D^s u$.  Set  $\mu:= D^s u$. For every $B \in \mathcal{B}(U)$ we have
\begin{equation}
\label{mu}
\mu(B)= \int_{J_u \cap B}[u]\otimes \nu_u\, \mathrm{d}\mathcal{H}^{d-1}\:\: \text{and}\:\: |\mu|(B)= \int_{J_u \cap B}|[u]|\, \mathrm{d}\mathcal{H}^{d-1}, 
\end{equation}
which implies  
\begin{equation}
\label{H}
\mathcal{H}^{d-1}(J_u \cap B)  =\int_{J_u \cap B}\frac{1}{|[u]|}\,\mathrm{d}|\mu|.
\end{equation}
Consider, for every $\rho>0$, the measure $\mu^{\rho} \in \mathcal{M}(U;\mathbb{R}^{d \times d})$ defined by
\begin{equation*}
\mu^{\rho}(B)=\frac{\mu(B)}{\omega_{d-1}\rho^{d-1}}\:\: \text{for every}\:\: B \in \mathcal{B}(U),  
\end{equation*}
where $\omega_{d-1}$ is the measure of the unit ball in $\mathbb{R}^{d-1}$.  Given $u \in SBV(U;\mathbb{R}^d)$  and $\mu=D^s u$, by virtue of the Besicovitch derivation theorem and the rectifiabiliy of $J_u$, see \cite[Theorems 2.22, 2.83, and 3.78]{ambrosio2000fbv}, we have from \eqref{mu} that, when $\rho \to 0^+$,
\begin{equation}
\label{convmurho1}
\mu^{\rho}(B_{\rho}(x)\cap U)\to ([u]\otimes \nu_u)(x)\:\: \text{for}\:\: \mathcal{H}^{d-1}\text{-almost every}\:\: x \in J_u \cap U,   
\end{equation}
\begin{equation}
\label{convmurho2}
|\mu^{\rho}|(B_{\rho}(x)\cap  U )\to |[u](x)|\:\: \text{for}\:\: \mathcal{H}^{d-1}\text{-almost every}\:\: x \in J_u \cap U. 
\end{equation}
Since $\zeta \to \Tilde{f}_k(\omega,x,\zeta)$ is continuous and bounded, by Dominated Convergence,  \eqref{eq: new en}, and \eqref{H}--\eqref{convmurho2} it follows
that for every $u \in  \chi_j^{k} $ we have
\begin{equation}
\label{measurabilitylimit}
 \mathcal{E}[\omega](u,U)=\lim\limits_{\eta \to 0^+}\lim\limits_{\rho \to 0^+}\int_{U}\frac{\Tilde{f}_k\big(\omega,x,\mu^{\rho}(  U \cap B_{\rho}(x))\big)}{\max\{|\mu^{\rho}|( U  \cap B_{\rho}(x)),\eta\}}\,\mathrm{d}|\mu|(x).
\end{equation}
Let $R=2k\gamma$ and let $\mathcal{M}_R(U;\mathbb{R}^{d \times d})$ be the space of the $\mathbb{R}^{d \times d}$-valued Radon measures $\mu$ on $U$ such that $|\mu|(U)\leq R$.   By continuity of the mapping $u \mapsto D^s u$, see \eqref{eq_ceeonti}, and the fact that the image of $\chi_j^{k}$ under this map is contained in $\mathcal{M}_R(U;\mathbb{R}^{d \times d})$, the claim in \eqref{measurabilityclaim} is a direct consequence of \eqref{measurabilitylimit} and of the following property: for every $\eta>0$ and $\rho>0$ the function
\begin{equation}
\label{measurabilityfunction}
(\omega,\mu)\to \int_{U}\frac{\Tilde{f}_k\big(\omega,x,\mu^{\rho}( U  \cap B_{\rho}(x))\big)}{\max\{|\mu^{\rho}|( U \cap B_{\rho}(x)),\eta\}}\,\mathrm{d}|\mu|(x)\:\: \text{is}\:\: \mathcal{I}\otimes \mathcal{B}( \mathcal{M}_R(U;\mathbb{R}^{d \times d})) \text{ measurable}.
\end{equation}
Let us finally prove \eqref{measurabilityfunction}. To this end, we observe that
\begin{equation}
\label{eq4.21}
(x,\mu)\to | \mu^\rho  |(U \cap B_{\rho}(x))\:\: \text{is (jointly) lower semicontinuous on}\:\:  U \times \mathcal{M}_R(U  ;  \mathbb{R}^{d \times d}).
\end{equation}
In fact, by definition of total variation of a measure we have
\begin{equation}
\label{eq4.24}
|\mu|(B_{\rho}(x)\cap U)=\sup\Big\{\int_{U}\varphi(y-x)\,\mathrm{d}\mu(y): \varphi \in C_c^1(B_{\rho}(0);\mathbb{R}^{d \times d}), |\varphi|\leq 1\Big\}
\end{equation}
and the function $(x,\mu)\to \int_{U}\varphi(y-x)\, \mathrm{d}\mu(y)$ on $U \times \mathcal{M}_R(U;\mathbb{R}^{d \times d})$   is (jointly) continuous. Finally, we recall that the supremum of any collection of continuous functions is lower semicontinuous. Hence, equation \eqref{eq4.21} follows from \eqref{eq4.24}. In addition, also the $\mathbb{R}^{d \times d}$-valued function
\begin{equation}
\label{eq4.22}
(x,\mu)\to \mu^{\rho}(U \cap B_{\rho}(x))\:\:  \text{is}\:\: \mathcal{B}(U)\otimes \mathcal{B}(\mathcal{M}_R(U;\mathbb{R}^{d \times d}))\text{ measurable}.
\end{equation}
In fact, for a nondecreasing sequence $(\varphi_j)_j$ of nonnegative functions in $C_c^1(B_{\rho}(0))$ converging to $1$ on $B_\rho(0)$, it holds that
\begin{equation*}
\mu^{\rho}(U \cap B_{\rho}(x))=\frac{1}{\omega_{d-1}\rho^{d-1}}\lim\limits_{j \to \infty}\int_{U}\varphi_j(y-x) \,\mathrm{d}\mu(y),   
\end{equation*}
and each function $(x,\mu)\to \int_{U}\varphi_j(y-x)\, \mathrm{d}\mu(y)$ is (jointly) continuous on $U \times \mathcal{M}_R(U;\mathbb{R}^{d \times d})$. Since $\Tilde{f}_k$ is $\mathcal{I}\otimes \mathcal{B}(U)\otimes \mathcal{B}^{d \times d}$ measurable, from \eqref{eq4.21}-\eqref{eq4.22} we derive 
\begin{equation*}
(\omega,x,\mu)\to \frac{\Tilde{f}_k(\omega,x,\mu^{\rho}( U  \cap B_{\rho}(x)))}{\max\{|\mu^{\rho}|(  U  \cap B_{\rho}(x)),\eta\}}\:\: \text{is}\:\: \mathcal{I}\otimes \mathcal{B}(U)\otimes \mathcal{B}(\mathcal{M}_R(U;\mathbb{R}^{d \times d}))\:\: \text{measurable}.
\end{equation*}
Finally, \eqref{measurabilityfunction} follows from  \cite[Corollary A.3]{cagnetti2017stochastic}. 
\end{proof}

\begin{remark}[Comparison to the piecewise constant case \cite{cagnetti2017stochastic}]
\label{problemmeasurability} 
{\normalfont

(i) We emphasise  that for studying  the measurability of  $\omega \to m_{\mathcal{E}[\omega]}(w,U)$ it is crucial  to pass to a minimisation problem for truncated functions  \eqref{eq4.10}, where in contrast to \cite{cagnetti2017stochastic} we do not only control the functions but also their gradients   in $L^\infty$. In fact, otherwise we cannot guarantee the continuity of the mapping  $\Phi$ in \eqref{eq_ceeonti} which was fundamental for the proof, see before \eqref{measurabilityfunction}.   

We include an example  which shows that the mapping $\Phi(u) = D^su$ between  $PR_L(U) \cap \lbrace \Vert u \Vert_\infty \le 1 \rbrace \subset  BV(U;\R^d)$ and $\mathcal{M}(U;\mathbb{R}^{d \times d})$ is not continuous for the weak* topologies.  Consider $U=(0,1)^2$ and, for every $n \in \mathbb{N}$  and a given $M \in \R^{2 \times 2}_{\rm skew}$ with $0<|M| \le 1$,  define 
\begin{equation*}
u_n(x)=\sum^{k_n}_{i=1}k_n  M (x-x_i)\chi_{B(x_i,{\frac{1}{k_n}})}(x),    
\end{equation*}
where  $(k_n)_n$ satisfies $k_n \to \infty$ and the points $x_1,...,x_{k_n}$ are chosen in such a way that $B(x_i,{\frac{1}{k_n}}) \subset U$ and $B(x_i,{\frac{1}{k_n}})\cap B(x_j,{\frac{1}{k_n}})=\emptyset$ if $i \neq j$.  One can check that for all $n \in \N$ 
\begin{equation*}
\Vert u_n\Vert _{L^1(U;\mathbb{R}^2)}\leq \frac{\pi |  M | }{k_n} \le \frac{\pi  }{k_n}, \quad \quad \quad     \Vert u\Vert _{L^{\infty}(U; \mathbb{R}^2) }\leq | M | \le 1,
\end{equation*}
and 
\begin{equation*}
\sup_{n \in \mathbb{N}}|Du_n|(U)\leq 3\pi |M| \le 3\pi.   
\end{equation*}
In particular, we have $u_n \to 0$ in $L^1(U;\R^d)$ and so $Du_n \rightharpoonup^* 0$ by \cite[Theorem 3.23]{ambrosio2000fbv}. We now argue that $D^su_n$ does not converge weakly* to zero.  In fact, if it was converging, we would also get  $\nabla u_n \mathcal{L}^2 \mathrel{\ensurestackMath{\stackon[1pt]{\rightharpoonup}{\scriptstyle\ast}}} 0$, but this contradicts the fact that  $\int_U \nabla u_n \,\mathrm{d}\mathcal{L}^2 \to \pi M$ as $n \to \infty$. Let us note that in this example it is crucial that $(\nabla u_n)_n$ is not uniformly integrable, as otherwise we can indeed prove continuity of $\Phi$, cf.\ \eqref{eq_ceeonti}. \EEE


(ii) Another option  to study the problem could be to use the weak* topology of $BD$ instead of $BV$ as it is easier to obtain continuity of $u \mapsto E^s u$, where $E^su = \frac{1}{2} ((D^su)^T + D^su)$. (For instance, it has been proved that also $BD(U)$ is the dual of a Banach separable space \cite[Proposition 2.5]{Temam1980FunctionsOB}.)  This, however, leads to severe problems in Step 2 of the proof. Indeed, adapting the identification \eqref{eq: ftildeXX} to a setting with $BD$-topology, would require  to find a function $\Tilde{f}$ that allows us to pass from $f(\omega,x,[u],\nu_u)$ to $\Tilde{f}(\omega,x,E^s u)$. Note, however, that there is no  obvious one-to-one correspondence between the couples $(\zeta,\nu)$ (with respect to the equivalence relation $(\zeta,\nu) \sim (-\zeta,-\nu)$)  and the symmetric matrices of the form $\zeta \odot \nu :=    \frac{\zeta\otimes \nu+\nu \otimes \zeta}{2}$.  
 
}
\end{remark}

%
  
\subsection{$\Gamma$-convergence to homogenised functional}\label{sec: gamma}

We start with the proof of Theorem \ref{gammaconv}. The goal is to show that for each $\omega \in \Omega'$ the existence of the limit \eqref{eqhom} implies the $\Gamma$-convergence in \eqref{gamma}. In this sense, it relies on a purely deterministic argument (for fixed $\omega$) and we resort to the $\Gamma$-convergence results obtained in \cite{Friedrich_2020}. Still, some adaptations of the results in \cite{Friedrich_2020} are in order to cover the case of homogenisation.  First, we recall the main $\Gamma$-convergence result of  \cite{Friedrich_2020}.

\begin{theorem}[Compactness of $\Gamma$-convergence]\label{th: gamma}
Let $L=SO(d)$ or $L=\mathbb{R}_{\mathrm{skew}}^{d \times d}$. Let $(f_n)_n$ be a sequence in $\mathcal{F}$ and let $\mathcal{E}_n:   L^0(\mathbb{R}^d;\mathbb{R}^d)\times \mathcal{A}\to [0,\infty]$ be a sequence of functionals as in \eqref{energyequation} with $f_n$ in place of $f$. Then, there exists  a subsequence (not relabeled)    such that  
 \begin{equation}
 \label{compactnessequation}
 {\mathcal{E}}_0(\cdot,U) =\Gamma\text{-}\lim_{n \to \infty} \mathcal{E}_n(\cdot,U) \ \ \ \ \text{with respect to convergence in measure on $U$}    
 \end{equation}
for all $U \in \mathcal{A}_0$, where $\mathcal{E}_0$ is the functional defined in \eqref{energyequation} with density $f_0$  characterised by
\begin{align}\label{eq: cahrac}
f_0(x,\zeta,\nu) = \limsup_{\rho \to 0}  \frac{m_{{\mathcal{E}_0}}(u_{x,\zeta,\nu},Q_{\rho}^{\nu}(x))}{\rho^{d-1}}
\end{align}
for all $x \in \R^d$, $\zeta \in \R^d \setminus \lbrace 0 \rbrace$, and $\nu \in \mathbb{S}^{d-1}$. 
\end{theorem}

\begin{proof}
The result has been essentially proved in \cite[Theorems 2.3, 7.6]{Friedrich_2020}. We describe the slight adjustments needed for the above version of the statement. The proof is divided in two steps: we first show that all  functionals having density $f \in \mathcal{F}$ satisfy the hypotheses of \cite[Theorem 2.3, 7.6]{Friedrich_2020}.  Afterwards,  using a standard diagonal argument, we see that in   \cite[Theorem 2.3, 7.6]{Friedrich_2020} the open bounded Lipschitz set can be   replaced with $\mathbb{R}^d$ without affecting the validity of the statements. 

\emph{Step 1}: Fix $U \in \mathcal{A}_0$. Suppose $f \in \mathcal{F}$ and let us denote with $\mathcal{E}'$ the restriction of the corresponding functional to the set $PR_L(U) \times \mathcal{A}(U)$, namely the functional defined as 
\begin{equation*}
\mathcal{E}'(u,A)=\int_{J_u \cap A}{f(x,[u](x),\nu_u(x))}\, \mathrm{d}\mathcal{H}^{d-1}(x),    
\end{equation*}
for all $u \in PR_L(U)$ and $A \in \mathcal{A}(U)$. Then, we can extend $\mathcal{E}'(u,\cdot)$ on $\mathcal{B}(U)$. 
We prove that $\mathcal{E}'$ satisfies properties $(\mathrm{H1})$ and $\mathrm{(H3)}$--$(\mathrm{H6})$ of \cite{Friedrich_2020}. We start by noticing that $\mathrm{(H1)}$ and $\mathrm{(H3)}$ are immediately satisfied due to properties of integral functionals with densities in $L^1$.  Moreover, $\mathrm{(H4)}$ is a direct consequence of $(f5)$--$(f6)$. Let us now prove  $\mathrm{(H5)}$, namely that there exists an increasing modulus of continuity $\sigma' \colon [0,\infty) \to [0,c_2]$ with $\sigma'(0)=0$ such that for any $u,v \in PR_L(U)$ and $S \in \mathcal{B}(U)$ with $S \subset J_u \cap J_v$ we have
\begin{equation}
\label{H5}
{ |\mathcal{E}'(u,S)-\mathcal{E}'(v,S)|  \leq \int_{S}\sigma'(|[u](x)-[v](x)|)\, \mathrm{d}\mathcal{H}^{d-1}(x).}
\end{equation}
We prove that $\mathcal{E}'$ satisfies $(\mathrm{H5})$ with $\sigma'=2c_2 \sigma$. Indeed, by virtue of $(f2)$ we have
\begin{equation*}
\begin{split}
{|\mathcal{E}'(u,S)-\mathcal{E}'(v,S)| \leq \int_{S}\sigma(|[u](x)-[v](x)|)(|f(x,[u](x),\nu_u(x))|+|f(x,[v](x),\nu_u(x))|)\, \mathrm{d}\mathcal{H}^{d-1}(x),}
\end{split}
\end{equation*}
which along with $(f6)$  gives \eqref{H5}. Finally, 
 $(\mathrm{H6})$  (see  \cite[Section 7]{Friedrich_2020}) holds by  $(f4)$.

\emph{Step 2:} We consider a family of functionals corresponding to densities $(f_n)_n$, namely $\mathcal{E}_n \colon L^0(\mathbb{R}^d;\mathbb{R}^d) \times \mathcal{A} \to [0,\infty]$, which by  \emph{Step 1} satisfy $(\mathrm{H1})$ and $\mathrm{(H3)}$--$(\mathrm{H6})$ of \cite{Friedrich_2020}. In particular, given the sequence of balls $(B_j(0))_{j \in \mathbb{N}}$, it follows that for every $j \in \mathbb{N}$ there exists a subsequence $(n_k)_k$ (possibly depending on $j$) and a functional $\mathcal{E}'_j \colon PR_L(B_j(0))\times \mathcal{B}(B_j(0)) \to [0,\infty]$, admitting an integral representation with density as in \eqref{eq: cahrac}, such that
\begin{equation}
\Gamma-\lim\limits_{k \to \infty}\mathcal{E}_{n_k}(\cdot,U)=\mathcal{E}'_j (\cdot,U)  
\end{equation}
for all $U \in \mathcal{A}_0$ such that $U \subset B_j(0)$.
Finally, up to a standard diagonal argument, the $\Gamma$-converging subsequence can be chosen independently on $j$, and we find $ \mathcal{E}'_{j_1}(\cdot,U) =  \mathcal{E}'_{j_2}(\cdot,U)$ for all $j_1 \le j_2$ and  $U \in \mathcal{A}_0$ with $U \subset B_{j_1}(0)$. This concludes the proof since it allows us to redefine the $\Gamma$-limit energy $\mathcal{E}_0$ on all $L^0(\mathbb{R}^d; \mathbb{R}^d)\times \mathcal{A}_0$ by taking
\begin{equation*}
    \mathcal{E}_0(u,U)=\begin{cases}
        \int_{J_u \cap U}f_0(x,[u],\nu_u)\, \mathrm{d}\mathcal{H}^{d-1}\:\: \text{if}\:\: u_{|_U}\in PR_L(U)\\
        +\infty \:\: \text{otherwise},
    \end{cases}
\end{equation*}
for all $u \in L^0(\mathbb{R}^d;\mathbb{R}^d)$ and $U \in \mathcal{A}_0$. 
\end{proof}

\begin{remark}
\label{rem: omportant}
{\normalfont
When applying the results of \cite{Friedrich_2020}, we want to emphasise three aspects:

(i) In \cite{Friedrich_2020}, the minimisation problems of type \eqref{eq: cahrac} were formulated on balls in place of cubes $Q_{\rho}^{\nu}(x)$ oriented in direction $\nu$. The formulations are equivalent and all results in \cite{Friedrich_2020} could have been obtained with cubes in place of balls.

 (ii) For $L=SO(d)$, Definition \ref{definfimumproblem} differs from the one in \cite[(2.4)]{Friedrich_2020} since there $PR_L$ in place of ${\rm id} + PR_0$ is used. However, the density $f_0$ in \eqref{eq: cahrac} is not affected by this change. Indeed, for any $f \in \mathcal{F}$ and corresponding  energy $\mathcal{E}$, consider any  competitor $ u(y)=\sum_{j\in \mathbb{N}}(M_jy+b_j)\chi_{P_j}$ for $y \in Q_\rho^\nu(x)$, with   $M_j \in SO(d)$ and  $b_j \in \mathbb{R}^d$, satisfying $\mathcal{H}^{d-1}(J_u \cap Q^\nu_\rho(x)) \le  \frac{c_2}{c_1}\mathcal{H}^{d-1}(J_{u_{x,\zeta,\nu}}\cap Q^{\nu}_{\rho}(x))= \EEE \frac{c_2}{c_1}\rho^{d-1}$. Then,   for  $v(y)=\sum_{j\in \mathbb{N}}(\mathbb{I}(y-x)+  M_j\,x + b_j)\chi_{P_j} \in {\rm id} + PR_0(Q^\nu_\rho(x))$, by $(f2)$, $|R| \le \sqrt{d}$ for all $R \in SO(d)$, $|y-x| \le  \rho  \sqrt{d}$ for all $y \in Q_\rho^\nu(x)$, and $(f6)$ we find
\begin{align*}
\limsup_{\rho \to 0}   \frac{|{\mathcal{E}}(u,Q_{\rho}^{\nu}(x)) - {\mathcal{E}}(v,Q_{\rho}^{\nu}(x)) |}{\rho^{d-1}}  &= \frac{1}{\rho^{d-1}} \int_{J_u}\sigma(|[u]-[v]|)(|f(x,[u],\nu_u)|+|f(x,[v],\nu_u)|)\, \mathrm{d}\mathcal{H}^{d-1}(x),  \\
& = \limsup_{\rho \to 0}   \frac{1}{\rho^{d-1}}    \frac{c_2}{c_1}  \rho^{d-1}  \sigma(4d  \rho) 2c_2 = 0. 
\end{align*}  
The latter equation directly implies that given $m_{\mathcal{E}}^{SO(d)}$ as in Definition \ref{definfimumproblem}, it holds 
\begin{align*}
  \limsup\limits_{\rho\to 0}\frac{\inf\big\{\mathcal{E}({u,A): u\in PR_{SO(d)}(A), u=  {\rm id} + u_{x,\zeta,\nu} \:\: \text{near}\:\: \partial A\big\}}}{\rho^{d-1}}=\limsup\limits_{\rho \to 0}\frac{m^{SO(d)}_{\mathcal{E}}(u_{x,\zeta,\nu},  A )}{\rho^{d-1}}, 
\end{align*}
for every $x \in \mathbb{R}^d$, $\zeta \in \mathbb{R}^d \setminus \{0\}$,  $\nu \in \mathbb{S}^{d-1}$, and $A:= Q^{\nu}_{\rho}(x)$. \\ 
In this work, we have preferred to use the version in Definition \ref{definfimumproblem}   as with this the proof of  Theorem~\ref{HomogenizationFormula} is simpler and follows directly from \cite{cagnetti2017stochastic}.

(iii) In \cite{Friedrich_2020}, for simplicity also in the case $L = SO(d)$   the competitor $u_{x,\zeta,\nu}$ was used in place of ${\rm id} + u_{x,\zeta,\nu}$ (see \eqref{Lw}) although strictly speaking  this is not a function in $PR_L(\R^d)$. The reason is that, as seen in (ii), asymptotic cell formulas are not affected by the affine function $\mathbb{I}y$, or any other rigid motion $Ry$, $R \in SO(d)$. 
For clarification, in this paper we have decided to always add $\mathbb{I}y$ in the case $L=SO(d)$ to ensure that $ {\rm id} + u_{x,\zeta,\nu} \in PR_L(\R^d)$.}

\end{remark}

Note that the above compactness result   guarantees only that a subsequence converges. Next, we address a situation where the $\Gamma$-limit of the whole sequence exists. To this end, given a sequence of densities $(f_n)_n \subset \mathcal{F}$, we define 
\begin{align}
\label{f'}
f'(x,\zeta,\nu) = \limsup_{\rho \to 0} \liminf_{n \to \infty}  \frac{m_{\mathcal{E}_n}(u_{x,\zeta,\nu},Q_{\rho}^{\nu}(x))}{\rho^{d-1}}, 
\end{align}
\begin{align}
\label{f''}
f''(x,\zeta,\nu) = \limsup_{\rho \to 0} \limsup_{n \to \infty} \frac{m_{\mathcal{E}_n}(u_{x,\zeta,\nu},Q_{\rho}^{\nu}(x))}{\rho^{d-1}} 
\end{align}
for all $x \in \R^d$, $\zeta \in \R^d \setminus \lbrace 0 \rbrace$, and $\nu \in \mathbb{S}^{d-1}$.

\begin{proposition}[Identification of the $\Gamma$-limit] \label{prop4.13}
 Let $L=SO(d)$ or $L=\mathbb{R}_{\mathrm{skew}}^{d \times d}$. Let $(f_n)_n \subset \mathcal{F}$ be a sequence  and let $\mathcal{E}_n:   L^0(\mathbb{R}^d;\mathbb{R}^d)\times \mathcal{A}\to [0,\infty]$ be a sequence of functionals as in \eqref{energyequation} with $f_n$ in place of $f$. Moreover, let $f_\infty \in \mathcal{F}$ and suppose that for all $x \in \R^d$, $\zeta \in \R^d \setminus \lbrace 0 \rbrace$, and $\nu \in \mathbb{S}^{d-1}$ we have
\begin{equation}
\label{densitiesequality}
 f_\infty (x,\zeta,\nu)  := f'(x,\zeta,\nu) = f''(x,\zeta,\nu). 
\end{equation}
Then, for all   $U \in \mathcal{A}_0$ it holds that 
\begin{equation}
\label{gammaurysohn}
\mathcal{E}_\infty(\cdot,U) =\Gamma\text{-}\lim_{n \to \infty} \mathcal{E}_n(\cdot,U) \ \ \ \ \text{with respect to convergence in measure on $U$},
\end{equation}
where $\mathcal{E}_\infty$ is the functional defined in \eqref{energyequation} with density $f_\infty$.  
\end{proposition}

\begin{proof}
Thanks to Urysohn's lemma, to prove that \eqref{densitiesequality} implies \eqref{gammaurysohn}, it is sufficient to show that for every subsequence of $\mathcal{E}_n(\cdot,U)$ there exists a further subsequence (not relabeled) such that $\mathcal{E}_n(\cdot,U)$ $\Gamma$-converges to $\mathcal{E}_\infty(\cdot,U)$, where $\mathcal{E}_\infty(\cdot,U)$ is  defined as in \eqref{energyequation} with the density $f_{\infty}$ given in  \eqref{densitiesequality}. Furthermore, by virtue of Theorem \ref{th: gamma}, proving \eqref{gammaurysohn} is equivalent to prove that for every $x \in \mathbb{R}^d$, every $\zeta \in \mathbb{R}^{d}\setminus \{0\}$ and $\nu \in \mathbb{S}^{d-1}$, it holds that
\begin{equation}
\label{finfty=f0}
    f_{\infty}(x,\zeta,\nu)=f_0(x,\zeta,\nu),
\end{equation}
where $f_0$ is the density defined in \eqref{eq: cahrac}.  We apply Theorem \ref{th: gamma} and find a subsequence (not relabeled) $(\mathcal{E}_n)_n$ and a functional $\mathcal{E}_0$, with density $f_0$ defined by \eqref{eq: cahrac}, such that
\begin{equation*}
\mathcal{E}_0 (\cdot,U) =\Gamma\text{-}\lim_{n \to \infty} \mathcal{E}_n(\cdot,U) \ \ \ \ \text{with respect to convergence in measure on $U$},     
\end{equation*}
for every set $U \in \mathcal{A}_0$. As in Step 1 of the proof of Theorem \ref{th: gamma}, we get that the sequence $\mathcal{E}_n$ satisfies hypotheses $\mathrm{(H1)}$, $\mathrm{(H3)}$--$\mathrm{(H6)}$ of \cite{Friedrich_2020} with the same $0<c_1<c_2$, $c_0 \geq 1$, and $\sigma'\colon [0,\infty) \to [0, c_2]$.  Then, \cite[Lemma 6.3]{Friedrich_2020}   implies
\begin{equation}
\label{eq4.41}
\limsup\limits_{n\to \infty}m_{\mathcal{E}_n}(v,U)\leq m_{ \mathcal{E}_0}(v,U)   
\end{equation}
for all $U \in \mathcal{A}_0$ and $v \in L^0(\mathbb{R}^d;\mathbb{R}^d)$ such that  
$v_{|_U}\in PR_L(U)$. In particular, for fixed  $x \in \mathbb{R}^d$, $\zeta \in \mathbb{R}^d \setminus \{0\}$, and $\nu \in \mathbb{S}^{d-1}$, \eqref{eq4.41} holds for $U=Q_{\rho}^{\nu}(x)$ and $v=u_{x,\zeta,\nu}$. Hence, the inequality $f_{\infty}\leq f_0$ easily follows from \eqref{eq: cahrac}, \eqref{densitiesequality}, and \eqref{eq4.41}. Similarly, \cite[Lemma 7.5]{Friedrich_2020} (see also Remark \ref{rem: omportant}(i)) implies that for every cube $Q^{\nu}_{\rho}(x)$ and every $v \in L^0(\mathbb{R}^d;\mathbb{R}^d)$  with $v_{|_U}\in PR_L(U)$ it holds that
\begin{equation}
\label{eq4.42}
\sup_{0<\rho'<\rho}\liminf\limits_{n \to \infty}m_{\mathcal{E}_n}(v,Q_{\rho'}^{\nu}(x))\geq m_{\mathcal{E}_0}(v,Q_{\rho}^{\nu}(x)).
\end{equation}
Let $\varepsilon >0 $ and let $\Tilde{\rho}\in (0,\rho)$ be such that
\begin{equation}
\label{eq4.43}
\sup_{0<\rho'<\rho}\liminf\limits_{n \to \infty}m_{\mathcal{E}_n}(u_{x,\zeta,\nu}, Q_{\rho'}^{\nu}(x))\leq \liminf\limits_{n \to \infty}m_{\mathcal{E}_n}(u_{x,\zeta,\nu},Q_{\Tilde{\rho}}^{\nu}(x))+\varepsilon \rho^{d-1}.
\end{equation}
Combining \eqref{eq4.42}--\eqref{eq4.43} and using that $\Tilde{\rho}<\rho$ we get that
\begin{equation}
\label{eq4.44}
 \liminf\limits_{n \to \infty}\frac{m_{\mathcal{E}_n}(u_{x,\zeta,\nu},Q_{\Tilde{\rho}}^{\nu}(x))}{\Tilde{\rho}^{d-1}}+\varepsilon \geq \frac{m_{ \mathcal{E}_0}( u_{x,\zeta,\nu}, Q_{\rho}^{\nu}(x))}{\rho^{d-1}}.   
\end{equation}
Sending $\rho \to 0$ in \eqref{eq4.44} and using \eqref{eq: cahrac} and \eqref{densitiesequality} we finally get the other inequality $f_{\infty}\geq f_0$. This implies \eqref{finfty=f0} and concludes the proof.
\end{proof}
\EEE
 
We now proceed with a homogenisation result without periodicity assumption.  To this end, given $f \in \mathcal{F}$ we consider the energy $\mathcal{E}$ as defined in \eqref{energyequation}, as well as the sequence of energies  $\mathcal{E}_{\eps}\colon   L^0(\mathbb{R}^d;\mathbb{R}^d)\times \mathcal{A}\to [0,\infty]$ of the form
\begin{equation}
\label{not-randomenergies}
\mathcal{E}_\eps(u,A)=\begin{cases}
 \int_{A \cap J_u}{f(\omega, \tfrac{x}{\eps},  [u](x),\nu_{u}(x))\,\mathrm{d}\mathcal{H}^{d-1}(x)} & u_{|_A} \in PR_L(A),   \\
 +\infty &\text{otherwise}.
\end{cases}
\end{equation}

We now give a $\Gamma$-convergence result under the assumption that a homogenisation formula exists. Eventually, we will use that by Theorem \ref{HomogenizationFormula} such assumption holds $\mathbb{P}$-a.e.

\begin{theorem}[Homogenisation]\label{th: homog}
Let $L=SO(d)$ or $L=\mathbb{R}_{\mathrm{skew}}^{d \times d}$. Let $f \in \mathcal{F}$ and let $(\mathcal{E}_\eps)_\eps$ be as in \eqref{not-randomenergies}. Assume that for all $x \in \R^d$, $\zeta \in \R^d \setminus \lbrace 0 \rbrace$, and $\nu \in \mathbb{S}^{d-1}$ the limit  
\begin{align}\label{it needs to exist}
f_{\rm hom}(\zeta,\nu) :=  \lim_{t \to \infty} \frac{m_{\mathcal{E}}(u_{tx,\zeta,\nu},Q_{t}^{\nu}(tx))}{t^{d-1}}
\end{align}
 exists and is independent of $x$, where $\mathcal{E}$ is the energy with density $f$. 
Then, $f_{\rm hom} \in \mathcal{F}$ and for all $U \in \mathcal{A}_0$ it holds that
\begin{equation}
\label{eq4.49}
\mathcal{E}_{\rm hom}(\cdot,U) =\Gamma\text{-}\lim_{\eps \to 0} \mathcal{E}_\eps(\cdot,U) \ \ \ \ \text{with respect to convergence in measure on $U$},  
\end{equation}
where $\mathcal{E}_{\rm hom}\colon L^0(\mathbb{R}^d;\mathbb{R}^d)\times \mathcal{A}\to [0,\infty]$ is the functional defined by 
\begin{equation*}
\mathcal{E}_{\mathrm{hom}}(u,A)=\begin{cases}
    \int_{J_u \cap A}f_{\rm hom}([u](x),\nu_u(x))\,\mathrm{d}\mathcal{H}^{d-1}(x) & u_{|_A}\in PR_L(A),\\
    +\infty & \text{otherwise}, \end{cases}
\end{equation*} 

\end{theorem}

\begin{proof}   By virtue of  Proposition \ref{prop4.13}, to prove \eqref{eq4.49} is sufficient to show that
\begin{equation}
\label{eq4.50}    
f'(x,\zeta,\nu)=f''(x,\zeta,\nu)=f_{\mathrm{hom}}(\zeta,\nu)
\end{equation}
for every $x \in \mathbb{R}^d$, $\zeta \in \mathbb{R}^{d} \setminus \{0\}$, and $\nu \in \mathbb{S}^{d-1}$, where $f'$ and $f''$ are the functions defined in \eqref{f'} and \eqref{f''}, respectively. To this end, fix $x \in \mathbb{R}^d$, $\zeta \in \mathbb{R}^d \setminus \{0\}$ and $\nu \in \mathbb{S}^{d-1}$, $\rho>0$. Consider a competitor  $u \colon Q^\nu_\rho(x) \to \R^d$ with $u = u_{x,\zeta,\nu}$ near $\partial Q_{\rho}^{\nu}({x})$ which satisfies $u \in PR_L(Q^\nu_\rho(x))$ for $L = \R^{d\times d}_{\rm skew}$ or  $u \in PR_0(Q^\nu_\rho(x))$  for $L=SO(d)$, respectively.  For  $\varepsilon>0$, define ${u}_\varepsilon $ by $u_\varepsilon(z)=u(\varepsilon z)$ which lies in $PR_{\R^{d \times d}_{\rm skew}} (Q_{{\rho}/{\varepsilon}}^{\nu}({x}/{\varepsilon}))$ or $PR_0 (Q_{{\rho}/{\varepsilon}}^{\nu}({x}/{\varepsilon}))$, respectively.  Then, we have $J_{{u}_\varepsilon}=\frac{1}{\varepsilon}J_u$, $[u_\varepsilon](z)=[u](\varepsilon z)$ for $\mathcal{H}^{d-1}$-a.e $z \in J_{u_\varepsilon}$ and $u_\varepsilon = u_{{x}/{\varepsilon},\zeta,\nu}$ near $\partial Q_{{\rho}/{\varepsilon}}^{\nu}({x}/{\varepsilon})$. By a change of variables we get that $\mathcal{E}_{\varepsilon}(u,Q_{\rho}^{\nu}(x))=\varepsilon^{d-1}\mathcal{E}(u_\varepsilon,Q_{{\rho}/{\varepsilon}}^{\nu}({x}/{\varepsilon}))$ for $L = \R^{d\times d}_{\rm skew}$ and $\mathcal{E}_{\varepsilon}({\rm id} + u,Q_{\rho}^{\nu}(x))=\varepsilon^{d-1}\mathcal{E}({\rm id} + u_\varepsilon,Q_{{\rho}/{\varepsilon}}^{\nu}({x}/{\varepsilon}))$ for for $L=SO(d)$, respectively. Therefore, we deduce 
\begin{equation}
m_{\mathcal{E}_{\varepsilon}}(u_{x,\zeta,\nu},Q_{\rho}^{\nu}(x))=\varepsilon^{d-1} m_{\mathcal{E}}(u_{{x}/{\varepsilon},\zeta,\nu}, Q_{{\rho}/{\varepsilon}}^{\nu}({x}/{\varepsilon}))=\frac{\rho^{d-1}}{ r^{d-1}_\eps}m_{\mathcal{E}}(u_{r_\varepsilon {x}/{\rho}, \zeta,\nu}, Q_{r_\varepsilon}^{\nu}(r_\varepsilon {x}/{\rho})),    
\end{equation}
where $r_\varepsilon \defas \frac{\rho}{\varepsilon}$. By replacing $x$ with ${x}/{\rho}$ in \eqref{it needs to exist}, we obtain
\begin{equation*}
\lim\limits_{\varepsilon \to 0}\frac{1}{\rho^{d-1}}m_{\mathcal{E}_{\eps}}(u_{x,\zeta,\nu},Q_{\rho}^{\nu}(x))=f_{\mathrm{hom}}(\zeta,\nu).   
\end{equation*}
As  $\rho>0$ was arbitrary, by sending $\rho \to 0$ in the last equation we get \eqref{eq4.50}. 
\end{proof}

\begin{remark}[Minimisation problem for $L=SO(d)$]
\label{remarkhomogenisationSO(d)}
{\normalfont We point out that the rescaling argument in the previous proof,  replacing a competitor $u \in PR_L(Q^\nu_\rho(x))$ by $u_\varepsilon$, does not work for $L =SO(d)$ as in this case we would have  $\mathrm{det}(\nabla u_\varepsilon)=\varepsilon^d \neq 1$ a.e., and thus $u_\varepsilon \notin PR_L(Q^\nu_\rho(x))$. This is the reason why for $L=SO(d)$ we have replaced the minimisation problem on $PR_L$ by ${\rm id} + PR_0$, see Definition \ref{definfimumproblem} and also  Remark \ref{rem: omportant}(ii).}
\end{remark}


After these preparations, we are in the position to prove Theorem \ref{gammaconv}. 

\begin{proof}[Proof of Theorem \ref{gammaconv}]
Let $\Omega'$ be as in Theorem \ref{HomogenizationFormula} and fix $\omega \in \Omega'$. Then, the functionals $\mathcal{E}_\eps[\omega]$ defined in \eqref{randomenergies} satisfy all assumptions of Theorem \ref{th: homog}, in particular \eqref{it needs to exist} follows from \eqref{eqhom}. This shows \eqref{gamma}. Eventually, the fact that  $\mathcal{E}_{\mathrm{hom}}$ is deterministic under ergodicity assumption follows from Theorem~\ref{HomogenizationFormula}, cf.\ \eqref{eqerg}.  
\end{proof}

\subsection{Convergence of infima and  minimisers}\label{sec: gamma2}

This short subsection is devoted to the proofs of Corollaries \ref{convminima}--\ref{convergencealmostminimisers}.  As a  preparation, we recall the following truncation result, see \cite[Theorem 7.1]{Friedrich_2020}.

\begin{lemma}[Truncation]
\label{truncation}
Let $d=2$ or $d=3$. Let $\Psi \in \mathcal{A}_0$ and let $L=SO(d)$ or $L=\mathbb{R}_{\rm skew}^{d\times d}$. Let $\theta>0$ and let $\mathcal{E}$ be as in \eqref{energyequation}. Then, there exists $C_{\theta}=C_{\theta}(\theta, c_0, \Psi )>0$  (where $c_0$ is the constant in $(f3)$) such that for every $u \in PR_{L}(\Psi)$ and every $\lambda \geq 1$ the following holds: there exists a rest set $R\subset \R^d $   with
\begin{align}\label{eq: truncation-rest}
\mathcal{L}^d(R) \le \theta \big(  \mathcal{H}^{d-1}(J_u) + \mathcal{H}^{d-1}(\partial \Psi)\big)^{d/(d-1)}, \ \ \ \ \  \mathcal{H}^{d-1}(\partial^*R) \le \theta(\mathcal{H}^{d-1}(J_u) + \mathcal{H}^{d-1}(\partial \Psi)),
\end{align}
and a function $\Tilde{u} \in PR_{L}(\Psi)\cap L^{\infty}(\Psi; \mathbb{R}^d)$ such that
\begin{align}\label{eq: truncation-main}
{\mathrm{(i)}} & \ \  \lbrace u \neq \Tilde{u} \rbrace  \subset R \cup \lbrace |u| > \lambda \rbrace \ \ \ \text{up to a set of negligible $\mathcal{L}^d$-measure},\notag\\
{\mathrm{(ii)} }& \ \ \Vert \Tilde{u} \Vert_{L^\infty(U)} \le C_\theta \lambda  ,\notag\\
{\mathrm{(iii)}}& \ \ \mathcal{E}(\Tilde{u},\Psi) \le \mathcal{E}(u,\Psi) + c_2\mathcal{H}^{d-1}(\partial^*R).
\end{align}
\end{lemma}

\begin{remark} \label{convexity}
{\normalfont
As pointed out in \cite[Remark 7.2]{Friedrich_2020}, if $\Psi$ is convex, we get an additional relation between the measure $R\cap \partial \Psi$ and its boundary, namely 
\begin{equation}
\label{eqconvexity}
\mathcal{H}^{d-1}(R\cap \partial \Psi)\leq C \mathcal{H}^{d-1}(\partial^* R)    
\end{equation}
where $C$ is a universal constant not depending on $u$.  }
\end{remark}
\EEE

\begin{proof}[Proof of Corollary \ref{convminima}]
 Fix $\omega \in \Omega'$, and $\theta>0$. Let $(\eps_n)_n$ be an arbitrary sequence converging to $0$. Let $(u_{\eps_n})_n \subset PR_L^{u_0}(U)$ be a  minimising sequence in the sense that  $\mathcal{E}_{\varepsilon_n}[\omega](  u_{{\varepsilon}_n},U) \le  \inf_{v \in PR^{u_0}_L(U)}\mathcal{E}_{\varepsilon_n}[\omega](v,U)  + \min\lbrace  \eps_n, \theta \rbrace$. By Lemma \ref{truncation} applied for $\lambda = 1$ and  $\theta>0$ on $\Psi$, we find a sequence $\Tilde{u}_{\eps_n} \in PR_{L}(\Psi)\cap L^{\infty}(\Psi; \mathbb{R}^d)$ and corresponding rest sets $(R_{\eps_n})_n$. Then, setting  $(u^{\theta}_{\varepsilon_n})_n \subset PR^{u_0}_L(U)\cap L^{\infty}(U;\mathbb{R}^d)$ by
$$u^{\theta}_{\varepsilon_n} = \begin{cases} \tilde{u}_{\varepsilon_n}(x) & \text{ for } x \in \Psi, \\ u_0(x) & \text{ for } x \in U \setminus \overline{\Psi}, \end{cases}$$
and using Remark \ref{convexity} on $\Psi$ as well as $(f6)$ one can check that there is a universal constant $C>0$ and   $C_{\theta}>0$  such that $\sup_{n \in \mathbb{N}}\Vert u^{\theta}_{\varepsilon_n}\Vert _{L^{\infty}(U;\mathbb{R}^d)}\leq C_{\theta}+ \Vert u_0 \Vert_\infty$ and  
\begin{align}
\label{eq3.9}
\mathcal{E}_{\varepsilon_n}[\omega](u^{\theta}_{{\varepsilon}_n},U) &\leq   \mathcal{E}_{\varepsilon_n}[\omega](u_{{\varepsilon}_n},U) +  c_2\mathcal{H}^{d-1}\big(\partial^* R_{\eps_n}   \big) + c_2\mathcal{H}^{d-1}\big( R_{\eps_n} \cap \partial \Psi\big)  \notag  \\  & \leq   \mathcal{E}_{\varepsilon_n}[\omega](u_{{\varepsilon}_n},U) + C\mathcal{H}^{d-1}(\partial^*R_{\eps_n})   \le   \inf_{v \in PR^{u_0}_L(U)}\mathcal{E}_{\varepsilon_n}[\omega](v,U)+ C\theta  , 
\end{align}  \EEE
where we also have used a uniform bound $\mathcal{H}^{d-1}(J_{u_{\eps_n}} \cap U)$ induced by $(f5)$. Hence, applying  \cite[Lemma 3.3]{Friedrich_2020} we can extract a subsequence, not relabeled, such that $u_{\varepsilon_n}^{\theta}\to u^{\theta}$ in measure on $U$, for some $u^{\theta}\in PR^{u_0}_L(U) \cap L^\infty(U;\R^d)$. Then, due to Theorem \ref{gammaconv} and \eqref{eq3.9}, we have 
\begin{equation}
\label{eq4.30}
\liminf\limits_{n \to \infty} \inf_{v \in PR^{u_0}_L(U)}\mathcal{E}_{\varepsilon_n}[\omega](v,U)+  C \theta  \geq \mathcal{E}_{\mathrm{hom}}[\omega](u^{\theta}, U)\geq \inf_{v \in PR^{u_0}_L(U)}\mathcal{E}_{\mathrm{hom}}(v,U).
\end{equation}
 Applying this reasoning  for $\theta = \frac{1}{k}$, $k \in \N$, by a diagonal argument we find a subsequence $(\eps_k)_k$ of $(\eps_n)_n$ such that   
\begin{equation}
\label{eq4.31}
\liminf\limits_{k \to \infty} \inf_{v \in PR^{u_0}_L(U)}\mathcal{E}_{\varepsilon_k}[\omega](v,U)\geq \inf_{v \in PR^{u_0}_L(U)}\mathcal{E}_{\mathrm{hom}}(v,U).    
\end{equation}
Again given $\theta >0$, we now let $v^{\theta}\in PR^{u_0}_L(U)$ be such that 
\begin{equation*}
\mathcal{E}_{\mathrm{hom}}[\omega](v^{\theta}, U)\leq \inf_{v \in PR^{u_0}_L(U)}\mathcal{E}_{\mathrm{hom}}[\omega](v,U)+\theta.
\end{equation*}
Using Theorem \ref{th_ bundary data} we can find a recovery sequence $(u_{\varepsilon_n}^{\theta})_{n}$ for $v^{\theta}$. Hence, it holds
\begin{equation*}
\inf_{v \in PR^{u_0}_L(U)}\mathcal{E}_{\mathrm{hom}}[\omega](v,U)+\theta\geq \lim\limits_{n \to \infty}\mathcal{E}_{\varepsilon_n}[\omega]( u^{\theta}_{\varepsilon_n}, U)\geq   \limsup\limits_{n \to \infty} \inf_{v\in PR^{u_0}_L(U)}\mathcal{E}_{\varepsilon_n}[\omega](v, U).
\end{equation*}
As $\theta>0$ is arbitrary, this implies
\begin{equation} 
\label{eq4.32}
\inf_{v \in PR^{u_0}_L(U)}\mathcal{E}_{\mathrm{hom}}[\omega](v,U)\geq   \limsup\limits_{n \to \infty}  \inf_{v\in PR^{u_0}_L(U)}\mathcal{E}_{\varepsilon_n}[\omega](v, U).
\end{equation}
The thesis  follows by using \eqref{eq4.31}--\eqref{eq4.32} on the subsequence $(\eps_k)_k$ along with Urysohn's lemma.
\end{proof}
\begin{proof}[Proof of Corollary \ref{convergencealmostminimisers}]
Consider the sequence $(u_{\varepsilon_n}^{\theta})_n$ and its limit $u^\theta$ as given in the proof of Corollary~\ref{convminima} (applied for $\theta/C$ in place of $\theta$). Then, the first inequality follows from \eqref{eq3.9}. For the second, we  use  \eqref{eq4.30} and \eqref{eq: 3.8X}. 
\end{proof}

\section{Random surface energies defined on asymptotically piecewise rigid functions}\label{sec: limit}

This section is devoted to the proofs of Theorems \ref{brittlemateriallimit} and \ref{brittlemateriallimit2}.  Although we address the case of stochastic   homogenisation, having proved Theorems \ref{HomogenizationFormula}--\ref{gammaconv}, the remaining arguments are purely deterministic. Therefore, in this section  we fix $\omega \in \Omega'$, but we drop it in the notation, i.e.\ we write $\mathcal{F}_{\eps,\delta}$ and $\mathcal{F}^{\rm lin}_{\eps,\delta}$  in place of $\mathcal{F}_{\eps,\delta}[\omega]$ and $\mathcal{F}^{\rm lin}_{\eps,\delta}[\omega]$, see \eqref{completenergy} and \eqref{completenergy-lin},   and we also do not include the $\omega$-dependence of the functions $f$ and  $W$  defined in Subsections \ref{sec: setting}  and \ref{sec: resul2}. 

\subsection{Approximation with piecewise rigid functions}

 Before we come to the proofs of the main theorems, we present an auxiliary result of possible independent interest, namely an approximation   of configurations with small elastic energy by piecewise  rigid functions. In the following, we say that $r \colon \R^d \to \R^d$ is a rigid motion if $r(x) = M\,x + b$ for $M \in SO(d)$ and $b \in \R^d$.

\begin{proposition}[Approximation with piecewise rigid functions] \EEE
\label{piecewiserigidapproximation}

Let $U \in \mathcal{A}_0$. Let $\varepsilon>0$, $\delta \in (0,1)$, and  $\gamma \in   (0,  \beta)$.  Then, for every $y \in GSBV_2^2(U;\mathbb{R}^d)$ satisfying $\mathcal{F}_{\varepsilon, \delta} (y,U) \le C_0 $   for some $C_0>0$, there exists a Caccioppoli partition $({P}^{\delta}_j)_j$ and corresponding rigid motions $({r}^{\delta}_j)_j$    such that, defining
\begin{equation*}
{y}_{\delta}(x):= \sum_{j \in \mathbb{N}}{r}^{\delta}_j(x)\chi_{{P}^{\delta}_j}(x),
\end{equation*}
we have 
\begin{equation}
\label{eq5.1}
\Vert y-{y}_{\delta}\Vert _{L^{\infty}(U;\mathbb{R}^d)}\leq C  \delta^{2\gamma-\beta} 
\end{equation}
and 
\begin{equation}
\label{eq5.2}
\mathcal{H}^{d-1}\big((J_{y_{\delta}}\cap U)\setminus J_y\big)    \le C\delta^{\beta - \gamma} 
\end{equation}
for some constant $C=C(C_0,U,d)>0$.    The function $y_\delta$ can be chosen such that $J_{{{y}_{\delta}}}\cap U = \bigcup_{j \in \mathbb{N}}\partial^* P_j^{\delta} \cap U$ up to an $\mathcal{H}^{d-1}$-negligible set. 
\end{proposition}

\begin{proof}   The proof is divided in two steps: in the first one, we show that there exists a piecewise affine function which well approximates $y$ in the $L^{\infty}$-norm, up to altering the jump set  $J_y$ only by a set of small $\mathcal{H}^{d-1}$-measure. In the second step, we show that such piecewise affine function can be actually chosen piecewise rigid and we evaluate the approximation error, proving \eqref{eq5.1}--\eqref{eq5.2}. To simplify the notation, during the proof we   indicate   with $C$ a suitable, positive constant,  possibly depending on $C_0$, $U$, and $d$. 

\emph{Step 1}: Consider $y \in GSBV_2^2(U;\mathbb{R}^{d})$ with $ \mathcal{F}_{\eps,\delta} (y,U) \le C_0$  and $\delta>0$. Arguing like in  \cite[Theorem~2.3]{friedrich2019griffith}, more precisely, see \cite[Equations (2.10)$\mathrm{(ii)}$ and (4.10)]{friedrich2019griffith},    we can show that  there exists a Caccioppoli partition $(P_j^{1,\delta})_j$ and a sequence of matrices $(\bar{M}_j^{\delta})_j  \subset \mathbb{R}^{d \times d}$ such that, given the piecewise affine function
\begin{equation*}
    y_{\delta}^{1}(x):= \sum_{j \in \mathbb{N}}\bar{M}_j^{\delta}x\,\chi_{P_j^{1,\delta}}(x),
\end{equation*}
we have
\begin{equation}
\label{eq5.3}
\sum_{j \in \N}\mathcal{H}^{d-1}\big(  (\partial^* P_j^{1,\delta} \cap U)\setminus J_y\big)\leq C\delta^{\beta-\gamma}  
\end{equation}
and for every $j \in \mathbb{N}$
\begin{equation}
\label{eq4.37}
\Vert \nabla y - \bar{M}_j^{\delta}\Vert _{L^{\infty}(P_j^{1,\delta};\mathbb{R}^{d \times d})}=\Vert \nabla y - \nabla y_{\delta}^{1}\Vert _{L^{\infty}(P_j^{1,\delta};\mathbb{R}^{d \times d})}\leq C \delta^{\gamma}.
\end{equation}
(We note that the argument in \cite[Theorem 2.3]{friedrich2019griffith} works for any $\gamma \in (0, \beta )$, but it was just used there for $\gamma >\frac{2}{3}$.) 
Then, thanks to \cite[Theorem 2.3]{friedrich2018piecewise}  applied on $y - y_{\delta}^{1}$, with $\rho=\delta^{\beta - \gamma}$, and \eqref{eq4.37}, we can find another Caccioppoli partition $(P_j^{2,\delta})_j$ and corresponding translations $(\bar{b}^{\delta}_j)_j$ such that, for
\begin{equation}\label{to linfty0}
v_{\delta} := y - y_{\delta}^{1}-\sum_{j \in \mathbb{N}}\bar{b}_j^{\delta}\,\chi_{P_j^{2,\delta}},    
\end{equation}
we have 
\begin{equation}\label{to linfty}
\Vert v_{\delta}\Vert _{L^{\infty}(U; \mathbb{R}^d)}\leq  C\delta^{2\gamma-\beta} 
 \end{equation}
and
\begin{equation}
\label{eq5.6}
\sum_{j \in \mathbb{N}}\mathcal{H}^{d-1}\big((\partial^* P_j^{2,\delta}\cap U)\setminus J_{y-y_{\delta}^{1}} \big)\leq C \delta^{\beta- \gamma} 
\end{equation}
for some suitable constant $C=C(C_0,U,d)>0$. Note that both constructions above are essentially a consequence of the  $BV$ coarea formula. We now construct a refinement $(P_{j}^{3,\delta})_j$ of the two  previous Caccioppoli partitions $(P_{j}^{1,\delta})_j$ and $(P_{j}^{2,\delta})_j$ by letting  $(P_{j}^{3,\delta})_j$ be the nonempty sets in the family 
$$ P_j^{1,\delta} \cap   P_k^{2,\delta}, \quad j,k \in \N.   $$
Clearly,  by construction $(P_j^{3,\delta})_j$ is still a Caccioppoli partition of $U$. Since $(P_j^{3,\delta})_j$ is a refinement of $(P_j^{1,\delta})_j$ and $(P_j^{2,\delta})_j$, we find, in view of \eqref{to linfty0},
\begin{equation}\label{y3}
y_{\delta}^{3}(x):=   y(x) - v_\delta(x)  =  \sum_{j \in \mathbb{N}}(M_j^{\delta}x+b_j^{\delta})\chi_{P_{j}^{3,\delta}}(x)
\end{equation}
for suitable $(M_j^{\delta})_j \subset \R^{d \times d}$ and $(b_j^\delta)_j \subset \R^d$. By \eqref{eq4.37} and \eqref{to linfty} it  holds
\begin{equation}
\label{stimaLinfinito}
\Vert y_{\delta}^3 - y\Vert _{L^{\infty}(U;\mathbb{R}^d)}\leq C  \delta^{2\gamma-\beta},  \quad \quad  \Vert \nabla y - \nabla y_{\delta}^{3}\Vert _{L^{\infty}( U;  \mathbb{R}^{d \times d})}\leq C \delta^{\gamma}. 
\end{equation}
By construction we have (up to $\mathcal{H}^{d-1}$-negligible sets) 
$$ \bigcup_{j \in \N}  \partial^* P_j^{3,\delta} \cap U \subset \Big( \bigcup_{j \in \N}  \partial^* P_j^{1,\delta} \cup \bigcup_{j \in \N}  \partial^* P_j^{2,\delta}   \Big) \cap U  \quad \text{ and } \quad  J_{y_{\delta}^{1}} \subset \Big(\Big( \bigcup_{j \in \N}  \partial^* P_j^{1,\delta} \Big) \cap U \Big) \cup J_y. $$
Consequently,  by \eqref{eq5.3}, \eqref{eq5.6} and since  $J_{y-y_{\delta}^1} \cap U \subset (J_y \cup (J_{{y^1_\delta}} \setminus J_y)) \cap U$, we have  
\begin{align}
\label{eq5.8}
\mathcal{H}^{d-1}\Big( \bigcup_{j \in \N}  (\partial^* P_j^{3,\delta} \cap U)\setminus J_{y}\Big) & \leq \sum_{j \in \N} \mathcal{H}^{d-1}\big(  (\partial^* P_j^{1,\delta} \cap U)\setminus J_y\big) +    \sum_{j \in \N} \mathcal{H}^{d-1}\big(  (\partial^* P_j^{2,\delta} \cap U)\setminus J_y\big)\notag \\ &
\le 2 \sum_{j \in \N} \mathcal{H}^{d-1}\big( (\partial^* P_j^{1,\delta} \cap U)\setminus J_y\big) +    \sum_{j \in \N} \mathcal{H}^{d-1}\big(  (\partial^* P_j^{2,\delta} \cap U)\setminus J_{y-y_\delta^1}\big) \notag \\ 
&\le    C\delta^{\beta - \gamma}.
\end{align}

\emph{Step 2}: Using $\mathcal{F}_{\eps, \delta}(y,U)  \le C_0$ and $\mathrm{(W3)}$, we get 
\begin{equation}
\label{eq5.9}
c \int_{U}\mathrm{dist}^2(\nabla y(x), SO(d))\, \mathrm{d}x\leq \int_{U}W\big( x,  \nabla y(x)\big)\,\mathrm{d}x \leq  C_0  \delta^2.    
\end{equation}
For every $\delta>0$ and $j \in \mathbb{N}$ denote with $R_j^{\delta}$ the projection of $M^{\delta}_j$ onto $SO(d)$. Notice that, because of   $\beta>\gamma$, \eqref{y3} and \eqref{stimaLinfinito}, for every $x \in P^{3,\delta}_j$ it holds that 
\begin{equation}
\label{eq5.10}
 \mathrm{dist}(\nabla y (x),SO(d))\geq |M_j^{\delta}-R_j^{\delta}|-  C_*  \delta^{ \gamma} 
\end{equation}
for some $C_* \ge 1$   depending on $C_0$, $U$, and $d$.   Let $J_{\delta}:=\{ j \in \mathbb{N}: |M_j^{\delta}-R_j^{\delta}|> 2 C_*\delta^{ 2\gamma-\beta}\}$. Then, combining \eqref{eq5.9}--\eqref{eq5.10} we get
\begin{equation}
\label{eq5.11a}
\sum_{j \in J_{\delta} }C_*^2 \delta^{4\gamma-2\beta} \mathcal{L}^d (P_j^{3,\delta}) \le \sum_{j \in J_{\delta}}\frac{1}{4} |M_j^{\delta}-R_j^{\delta}|^2  \mathcal{L}^d  (P_j^{3,\delta})  \leq  \frac{C_0}{c}  \delta^2+ C\delta^{2\gamma}.  
\end{equation}
We now construct a further refinement of the partition $(P_j^{3,\delta})_j$ by cutting each $P_j^{3,\delta}$, $j \in  J_{\delta}$, into (subsets of) small cuboids.  To this end, for $j \in J_\delta$ let   
\begin{align}\label{eq: deltadef}
\tau_j^{\delta} :=\frac{{\delta}^{4\gamma-2\beta} }{|M_j^{\delta} -R_j^{\delta}|^2},
\end{align}
and note that $\tau_j^\delta \le 1$. Let  $s^{i,j}_k := k \tau^{\delta}_j$ for $k \in \Z$ and $i=1,\ldots,d$ and note that the cubes 
$$\big\{  [s^{1,j}_{k_1}, s^{1,j}_{k_1+1}) \times \ldots \times [s^{d,j}_{k_d}, s^{d,j}_{k_d+1}) \colon k_1,\ldots,k_d \in \Z\big\} $$
 form  a partition of $\R^d$. We now slightly change the grid introduced by the values $(s^{i,j}_k)_{i,k}$ as follows: for every $i=1,\ldots,d$ and $k \in \Z $ we can find $t^{i,j}_k \in [s_{k}^{i,j}, s_{k+1}^{i,j}]$ such that  
  $$ \mathcal{H}^{d-1}\big( P^{3,\delta}_j \cap \Pi^i(t_k^{i,j}) \big) \le  \frac{1}{\tau_j^{\delta}} \int_{s_k^{i,j}}^{s_{k+1}^{i,j}}  \mathcal{H}^{d-1}\big(  P^{3,\delta}_j \cap \Pi^i(t)\big)\, \mathrm{d}t,$$
where $\Pi^i(t) := \R \times \ldots  \times \R \times \{t\}\times \ldots \times \R$ with $t$ in the $i$-th entry. We note that   this  estimate is clearly trivial for $|k|$ large enough.  Then, summation, Fubini's theorem, and \eqref{eq5.11a}--\eqref{eq: deltadef} yield  
 \begin{align}\label{eq5.14}
 \sum_{j \in J_{\delta}} \sum^{d}_{i=1}\sum_{k \in \Z} \mathcal{H}^{d-1}\big( P^{3,\delta}_j \cap \Pi^i(t_k^{i,j}) \big) & \le   \sum_{j \in J_{\delta}} \sum^{d}_{i=1}\sum_{k \in \Z}  \frac{1}{\tau_j^{\delta}} \int_{s_k^{i,j}}^{s_{k+1}^{i,j}}  \mathcal{H}^{d-1}\big(  P^{3,\delta}_j \cap \Pi^i(t)\big)\, \mathrm{d}t  \\ 
& =d   \sum_{j \in   J_{\delta} } \frac{1}{\tau^{\delta}_j} \mathcal{L}^d(P_j^{3,\delta})   = d   \sum_{j \in J_{\delta}} \frac{1}{\delta^{4\gamma-2\beta}} | M_j^{\delta}   -R_j^{\delta}|^2 \mathcal{L}^d(P_j^{3,\delta})  \le  C   \delta^{2\beta-2\gamma}, \notag
\end{align}
for a constant $C=C(C_0, c, U,d)>0$.   We now consider the partition of $\R^d$ into the cuboids
$$\big\{  [t^{1,j}_{k_1}, t^{1,j}_{k_1+1}) \times \ldots \times [t^{d,j}_{k_d}, t^{d,j}_{k_d+1}) \colon k_1,\ldots,k_d \in \Z\big\}, $$ 
denoted by  $(Q_l^{j,\delta})_l$. With this, we  define a refined Caccioppoli partition $(P_j^{\delta})_j$ formed by the following sets:
\begin{equation*}
P_j^{3,\delta}\:\: \text{if}\:\: j \notin J_{\delta},
\end{equation*}
\begin{equation*}
P_j^{3,\delta} \cap Q_l^{j,\delta}\:\: \text{if}\:\: j \in J_{\delta}  \:\:\text{and if}\:\: l\:\: \text{is such that}\:\: P_j^{3,\delta} \cap Q_l^{j,\delta}\neq \emptyset.
\end{equation*}
In fact, notice that $(P_j^{\delta})_j$ is still a partition of $U$ and that we have (up to $\mathcal{H}^{d-1}$-negligible sets) 
 \begin{align}\label{eq5.14XXX}
 \bigcup_{j \in \N} \partial^* P_j^\delta \cap U \subset \Big( \bigcup_{j \in \N} \partial^* P_j^{3,\delta} \cap U \Big) \cup  \bigcup_{j \in J_{\delta}} \bigcup^{d}_{i=1}\bigcup_{k \in \Z}  \big( P^{3,\delta}_j \cap \Pi^i(t_k^{i,j}) \big).  
 \end{align}
Therefore, by  \eqref{eq5.14} and the fact that $(P_j^{3,\delta})_j$ is a Caccioppoli partition, we also find that $\sum_j \mathcal{H}^{d-1}(\partial^* P_j^{\delta})< +\infty $.  

Let us now come to the definition of suitable rigid motions. For every $j \in J_{\delta}$, and for every $l$ such that $P_j^{3,\delta} \cap Q_l^{j,\delta}\neq \emptyset$, we pick an arbitrary $x^{j,\delta}_l \in P_j^{3,\delta} \cap Q_l^{j,\delta}$ and we define
\begin{equation*}
d_l^{j,\delta}:= M_j^{\delta}x^{j,\delta}_l+b_j^{\delta}- R_j^{\delta}x^{j,\delta}_l,
\end{equation*} 
where  $ M_j^{\delta}$ and $b_j^{\delta}$ are given in \eqref{y3}, and $R_j^{\delta}$ in \eqref{eq5.10}.
Notice then that for every $j \in J_{\delta}$, and for every $x \in P_j^{3,\delta} \cap Q_l^{j,\delta}$, by \eqref{eq: deltadef},   $\tau_j^\delta \le  1$, and the fact that the cuboids have sidelength less than $2\tau^\delta_j$ it holds
\begin{equation}
\label{eq5.15}
|M_j^{\delta}x+b_j^{\delta}- (R_j^{\delta}x+d_l^{j,\delta})| =|(M_j^{\delta}-R_j^{\delta})(x-x^{j,\delta}_l)|\leq  C |M_j^{\delta}-R_j^{\delta}|\tau_j^\delta =   C|M_j^{\delta}-R_j^{\delta}|\sqrt{\tau_j^\delta} \sqrt{\tau_j^\delta} \leq C \delta^{2\gamma-\beta}.
\end{equation}
In a similar fashion, for each $j \in \N \setminus J_\delta$, the definition of $J_\delta$ implies
\begin{equation}
\label{eq5.15XXX}
|M_j^{\delta}x+b_j^{\delta}- (R_j^{\delta}x+b_j^{\delta})|  \le C \delta^{2\gamma-\beta} \quad \quad  \text{for all $x\in P^{3,\delta}_j$}. 
\end{equation}
Now, we define the piecewise rigid function
\begin{equation*}
   y_{\delta}(x)=\sum_{k \in \mathbb{N}} r_k^{\delta}(x)\chi_{P_k^{\delta}}(x),
\end{equation*}
where the rigid motions $(r_k^{\delta})_k$ are defined by 
\begin{equation*}
    \begin{cases}
       r_k^{\delta}(x)=R_j^{\delta}x+b_j^{\delta} &  \text{if }    P_k^{\delta}=P_j^{3,\delta}\:\: \text{for some}\:\: j \in \mathbb{N}\setminus J_{\delta} \\
        r_k^{\delta}(x)=R_j^{\delta}x+d_l^{j,\delta} &  \text{if }    P_k^{\delta}=P_j^{3,\delta}\cap Q_l^{j,\delta}\:\: \text{for some}\:\: j \in J_{\delta}\:\: \text{and for some cube}\:\: Q_l^{j,\delta}.
    \end{cases}
\end{equation*}
In particular, recalling the definition of $y_\delta^3$ in \eqref{y3},      \eqref{eq5.15} and \eqref{eq5.15XXX}   imply  $\Vert y_\delta - y^3_\delta \Vert_{L^\infty(U;\R^d)}\le C\delta^{2\gamma-\beta}$. This along with \eqref{stimaLinfinito} shows \eqref{eq5.1}.
In addition, \eqref{eq5.14}--\eqref{eq5.14XXX} imply
\begin{equation}
\label{eq5.17}
 \mathcal{H}^{d-1}\Big((J_{y_{\delta}}\cap U)\setminus   \bigcup_{j \in \N}  \partial^* P_j^{3,\delta}\Big)\leq C \delta^{2\beta-2\gamma}. 
\end{equation}
 Hence, because of \eqref{eq5.8} and \eqref{eq5.17} we obtain \eqref{eq5.2} as 
\begin{equation*}
\mathcal{H}^{d-1}((J_{y_{\delta}}\cap U)\setminus J_y) \le \mathcal{H}^{d-1}\Big( \bigcup_{j \in \N}  (\partial^* P_j^{3,\delta} \cap U)\setminus J_{y}\Big) + \mathcal{H}^{d-1}\Big((J_{y_{\delta}}\cap U)\setminus   \bigcup_{j \in \N}  \partial^* P_j^{3,\delta}\Big) \le   C\delta^{\beta - \gamma}  + C\delta^{2\beta - 2\gamma}.  
\end{equation*}
Eventually, we remark that by an infinitesimally small change of the rigid motions one can also guarantee that   $J_{{{y}_{\delta}}}\cap U = \bigcup_{j \in \mathbb{N}}\partial^* P_j^{\delta} \cap U$ up to an $\mathcal{H}^{d-1}$-negligible set.  
\end{proof}

\begin{remark}[Configurations close to the identity]\label{nice remark}
{\rm For later purposes,  we remark that an inspection of the proof also yields that, given a $\kappa>0$, the condition  
$$\Vert \nabla y  -\mathbb{I}\Vert _{L^{\infty}(U;\mathbb{R}^{d \times d})}\leq \kappa$$
implies
$${\Vert  \nabla  {y}_{\delta}  - \mathbb{I}\Vert _{L^{\infty}(U;\mathbb{R}^{d\times d})}\leq  C\kappa }$$
for some universal $C>0$. Indeed,  under this assumption, each $\bar{M}_j^{\delta}$ in \eqref{eq4.37} and thus each $M^\delta_j$ in \eqref{y3}  can be chosen such that $|\mathbb{I} -  M_j^{\delta}| \le C \kappa$. Thus, also the rotation $R_j^{\delta}$ chosen in \eqref{eq5.10} satisfies  $|\mathbb{I} - R_j^{\delta}| \le C\kappa$. 

}
\end{remark}

\subsection{The nonlinear case: Proof of  Theorem \ref{brittlemateriallimit}}

In this short subsection, we address the homogenization result for asymptotically piecewise rigid functions.

\begin{proof}[Proof of  Theorem \ref{brittlemateriallimit}] 
Let $   U  \in \mathcal{A}_0$. Let  $(\delta_\eps)_\eps\subset (0,1)$ with $\delta_\eps \to 0$ as $\eps \to 0$. We will assume without loss of generality $\varepsilon \in (0,1)$. 
The proof is divided into two steps. In the first one, we show the $\Gamma$-$\liminf\limits$ inequality and in the second one the construction of recovery sequences. 

\emph{Step 1: $\Gamma$-$\liminf$ inequality}.   Let $({y}_{\varepsilon})_{\varepsilon}$ and $y \in PR_L(U)$ be such that ${y}_{\varepsilon} \to y$ in measure on $U$ and such that 
\begin{equation}
\label{boundednessenergy}
\sup_{\varepsilon>0}{\mathcal{F}}_{\varepsilon} ({y}_{\varepsilon},U) \le C_0   
\end{equation}
for some $C_0>0$. Applying Proposition \ref{piecewiserigidapproximation} for $\gamma = 3\beta/4$ we can find a sequence $(\Tilde{y}_{\varepsilon})_{\varepsilon} \subset PR_L(U)$ such that 
\begin{equation}
\label{eq5.18}
\Vert {y}_{\varepsilon}-\Tilde{y}_{\varepsilon}\Vert _{L^{\infty}( U ;\mathbb{R}^d)}\leq \bar{C} \delta_\varepsilon^{\beta/2} 
\end{equation}
and
\begin{equation}
\label{eq5.19}
\mathcal{H}^{d-1}(J_{\Tilde{y}_{\varepsilon}}\setminus J_{{{y}}_{\varepsilon}})\leq \bar{C}\delta_\varepsilon^{\beta/4}. 
\end{equation}
We claim that 
\begin{equation}
\label{claimbrittle}
\liminf\limits_{\varepsilon \to 0}\mathcal{E}_{\varepsilon} (\Tilde{y}_{\varepsilon},U)\leq \liminf\limits_{\varepsilon \to 0} {\mathcal{F}}_{\varepsilon}( {y}_{\varepsilon},U).
\end{equation}
Once \eqref{claimbrittle} is proved, the $\Gamma$-$\liminf\limits$ inequality just follows by Theorem \ref{gammaconv} and by noticing that, because of \eqref{eq5.18}, also $\Tilde{y}_{\varepsilon}\to y$ in measure on $ U $.
Notice that there exists $C_0'>0$ such that
\begin{equation}
\label{boundednessmodifiednenergy}
\sup_{\varepsilon>0}\mathcal{E}_{\varepsilon}(\Tilde{y}_{\varepsilon}, U)\leq C_0'.
\end{equation}
Indeed, by virtue of $(f5)$, $(f6)$, \eqref{boundednessenergy}, and \eqref{eq5.19} we have 
\begin{equation*}
 \mathcal{E}_{\varepsilon}(\Tilde{y}_{\varepsilon},U)\leq c_2  \Big[\mathcal{H}^{d-1}((J_{\Tilde{y}_{\varepsilon}}\setminus J_{{y}_{\varepsilon}})\cap U )+\mathcal{H}^{d-1}(J_{\Tilde{y}_{\varepsilon}}\cap J_{{y}_{\varepsilon}}\cap U )\Big]\leq  c_2 \Big(C+\frac{C_0}{c_1}\Big)=: C_0'. 
\end{equation*}
Using $(\mathrm{W3})$, the fact that  $\Tilde{y}_\varepsilon$ is a piecewise rigid motion, $\nabla^2 \Tilde{y}_{\varepsilon}=0$ a.e.,  $(f6)$, and \eqref{eq5.19}, we get
\begin{align}
\label{eq5.23}
\begin{split}
\mathcal{E}_{\varepsilon}(\Tilde{y}_{\varepsilon},U ) &\leq \frac{1}{\delta_\varepsilon^2}\int_{ U }W\Big(\frac{x}{\varepsilon}, \nabla {y}_{\varepsilon}(x)\Big)\, \mathrm{d}x+\frac{1}{\delta_\varepsilon^{2 \beta }}\int_{U}|\nabla^2 {y}_{\varepsilon}(x)|^2\, \mathrm{d}x+
 c_2 \bar{C}\delta_\varepsilon^{\beta/4} \\ & \quad \quad \quad  +  \int_{J_{\Tilde{y}_{\varepsilon}}\cap J_{{y}_{\varepsilon}}\cap U }f\Big(\frac{x}{\varepsilon},[\Tilde{y}_{\varepsilon}],\nu_{\Tilde{y}_{\varepsilon}}(x)\Big)\, \mathrm{d}\mathcal{H}^{d-1}(x),
\end{split}
\end{align}
where we note that the measure-theoretic normals can be chosen such that $\nu_{{y}_{\varepsilon}} = \nu_{\Tilde{y}_{\varepsilon}}$ $\mathcal{H}^{d-1}$-a.e.\ on $J_{y_{\varepsilon}}\cap J_{\Tilde{y}_{\varepsilon}}$.
Using $(f2)$, \eqref{boundednessenergy}, \eqref{eq5.18}, and \eqref{boundednessmodifiednenergy} we can then estimate the  latter term   in \eqref{eq5.23} by 
\begin{equation}
\label{eq5.24}
\begin{split}
\int_{J_{\Tilde{y}_{\varepsilon}}\cap J_{{y}_{\varepsilon}}\cap U }f\Big(\frac{x}{\varepsilon},[\Tilde{y}_{\varepsilon}],\nu_{\Tilde{y}_{\varepsilon}}(x)\Big)\, \mathrm{d}\mathcal{H}^{d-1}(x) & \leq  \int_{ J_{{y}_{\varepsilon}}\cap U }f\Big(\frac{x}{\varepsilon},[{y}_{\varepsilon}],\nu_{{y}_{\varepsilon}}(x)\Big)\, \mathrm{d}\mathcal{H}^{d-1}(x)     +  \sigma(  2  \bar{C}  \delta_\varepsilon^{\beta/2})(C_0+C_0').
\end{split}
\end{equation}
Hence, \eqref{claimbrittle} follows by combining \eqref{eq5.23} and \eqref{eq5.24}. \\
\emph{Step 2: Existence of recovery sequences.}    The $\Gamma$-$\limsup\limits$ inequality is a direct consequence of the $\Gamma$-$\limsup\limits$ inequality of Theorem \ref{gammaconv}. Indeed, notice that for every $y \in PR_L( U )$ we obtain a recovery sequence $(y_\varepsilon)_{\varepsilon} \subset PR_L(U )$ such that $y_\varepsilon \to y$ in measure on $U$ and $\mathcal{E}_{\varepsilon}(y_{\varepsilon},U )={\mathcal{F}}_{\varepsilon}(y_{\varepsilon}, U )\to \mathcal{E}_{\mathrm{hom}}(y, U )$.
\end{proof}

\subsection{Linearisation: Proof of Theorem \ref{brittlemateriallimit2}}

We now come to the proof of linearisation result stated in Theorem \ref{brittlemateriallimit2}. As before, the random variable $\omega$ is omitted in the notation for simplicity.

\begin{proof}[Proof of Theorem \ref{brittlemateriallimit2}] Assume without restriction that  $\varepsilon \in (0,1)$, and let $\beta \in (\alpha,1)$ for $\alpha \in (0,1)$. We divide the proof into two steps.

\emph{Step 1:  $\Gamma$-$\liminf$   inequality.}  Let $(u_{\varepsilon})_{\varepsilon} \subset GSBV_2^2(U;\mathbb{R}^d)$ and $u \in PR_L(U)$ be such that $u_{\varepsilon} \to u$ in measure on $U$ and
\begin{equation}
\label{boundednesslinearizedneergies}
\sup_{\varepsilon>0} {\mathcal{F}}^{\mathrm{lin}}_{\varepsilon}(u_{\varepsilon},U) \le C_0 
\end{equation}
for some $C_0>0$. We claim that
\begin{equation}
\label{liminfinequalitylinearised}
\mathcal{E}_{\mathrm{hom}}(u,U)\leq \liminf\limits_{\varepsilon \to 0} {\mathcal{F}}^{\mathrm{lin}}_{\varepsilon}(u_{\varepsilon},U). 
\end{equation}
For every $\varepsilon>0$, define $
y_{\varepsilon}={\rm id}  +\delta^\alpha_{\varepsilon}u_{\varepsilon}$. By \eqref{for sure needed} and \eqref{boundednesslinearizedneergies} it holds
\begin{equation}
\label{eq4.59}
\Vert \nabla y_{\varepsilon}-\mathbb{I}\Vert _{L^{\infty}(U;\mathbb{R}^{d \times d})}\leq  \delta^{3\alpha/4}_{\varepsilon}. 
\end{equation}
As $\beta >\alpha$, we can choose $\gamma$ with $\beta > \gamma > \frac{1}{2}(\alpha + \beta)$. 
Thanks to Proposition \ref{piecewiserigidapproximation}, for every $\varepsilon>0$, there exists a function $\Tilde
{y}_{\varepsilon}$ of the form
\begin{equation}\label{thansl}
\Tilde{y}_{\varepsilon}=\sum_{j \in \mathbb{N}}(R_j^{\varepsilon}x+b_j^{\varepsilon})\chi_{P_j^{\varepsilon}}(x),    
\end{equation}
where $R_j^{\varepsilon} \in SO(d)$, $b_j^{\varepsilon} \in \mathbb{R}^d$, and $(P_j^{\varepsilon})_j$ is a Caccioppoli partition of $U$ such that
\begin{equation}
\label{eq4.61}
\Vert {y}_{\varepsilon}-\Tilde{y}_{\varepsilon}\Vert _{L^{\infty}(U;\mathbb{R}^d)}\leq C \delta_\varepsilon^{2\gamma-\beta},
\end{equation}
\begin{equation}
\label{eq4.62}
 \Vert \nabla \Tilde{y}_{\varepsilon} - \mathbb{I}\Vert _{L^{\infty}( U;  \mathbb{R}^{d \times d})}\leq   C  \delta^{3\alpha/4}_{\varepsilon}, 
\end{equation}
where for the second estimate we use Remark \ref{nice remark} and \eqref{eq4.59}.  Moreover, we have 
and
\begin{equation}
\label{eq4.63}
\mathcal{H}^{d-1}\big((J_{\Tilde{y}_{\varepsilon}}\cap U)\setminus J_{y_{\varepsilon}}\big)\leq C \delta_\eps^{\beta-\gamma}. 
\end{equation}
By \eqref{eq4.62}, we get that for every $j \in \mathbb{N}$ and $\varepsilon>0$ it holds $|R_j^{\varepsilon}-\mathbb{I}| \le C \delta_{\varepsilon}^{3\alpha /4}$. Thus, by the linearisation formula \cite[(4.12)]{friedrich2019griffith}, for every $j \in \mathbb{N}$ and $\varepsilon>0$ we can find a skew-symmetric matrix $M_j^{\varepsilon} \in \mathbb{R}_{\rm skew}^{d \times d}$ such that
\begin{equation}
\label{eq4.65XXX}
R_j^{\varepsilon}=\mathbb{I}+ \delta_\eps^{3\alpha/4}  M_j^{\varepsilon}+ \mathcal{O}  (\delta_\eps^{3\alpha/2}).
\end{equation}
Now, for every $\varepsilon>0$ we define the piecewise rigid function
\begin{equation}\label{hatudef}
\hat{u}_{\varepsilon}(x)=\sum_{j \in \mathbb{N}}{(  \delta_{\varepsilon}^{-\alpha/4} M_j^{\varepsilon}x+ \delta_\eps^{-\alpha}  b_j^{\varepsilon})}\chi_{P_j^{\varepsilon}}(x).
\end{equation}
Using \eqref{eq4.63} and the fact that $J_{{\Tilde{y}_{\varepsilon}}}\cap U = \bigcup_{j \in \mathbb{N}}\partial^* P_j^{\varepsilon} \cap U$ up to an $\mathcal{H}^{d-1}$-negligible set, it is straightforward to show that
\begin{equation}
\label{eq4.67}
\mathcal{H}^{d-1}((J_{\hat{u}_{\varepsilon}}\cap U)\setminus J_{u_{\varepsilon}})\leq \mathcal{H}^{d-1}((J_{\Tilde{y}_{\varepsilon}}\cap U)\setminus J_{y_{\varepsilon}}) \leq C \delta_\eps^{\beta-\gamma}. 
\end{equation}
 Finally, recalling the definitions in  \eqref{thansl} and \eqref{hatudef}, \EEE by   combining \eqref{eq4.61} and \eqref{eq4.65XXX} we get
\begin{align}
\label{eq4.69}
\delta^\alpha_{\varepsilon} \Vert u_{\varepsilon}-\hat{u}_{\varepsilon}\Vert _{L^{\infty}(U;\mathbb{R}^d)}  &  \le \Vert y_\eps - \big({\rm id} + \delta^\alpha_\eps \hat{u}_\eps  \big)\Vert _{L^{\infty}(U;\mathbb{R}^d)}  \notag \\ 
&\le \Vert \Tilde{y}_{\varepsilon} - {y}_{\varepsilon} \Vert_{L^{\infty}(U;\mathbb{R}^d)}  + \Big\|  \sum_{j \in \mathbb{N}} \big( R^\eps_j  - \mathbb{I}  - \delta_{\varepsilon}^{3\alpha/4}M_j^{\varepsilon} \big) \cdot \chi_{P_j^{\varepsilon}}\Big\|_{L^{\infty}(U;\mathbb{R}^d)} \notag \\ 
& \leq  C\delta_\varepsilon^{2\gamma-\beta}+ C  \delta^{3\alpha/2}_{\varepsilon}. 
\end{align}
   Then, $\Vert u_{\varepsilon}-\hat{u}_{\varepsilon}\Vert _{L^{\infty}(U;\mathbb{R}^d)}  \to 0$  as $\varepsilon \to 0$ since we indeed have
\begin{align}\label{for sonn}
\delta_\varepsilon^{2\gamma-\beta-\alpha} \to 0
\end{align}
 by $\gamma > \frac{1}{2}(\alpha + \beta)$ and $\delta_\eps \to 0$. This shows that $\hat{u}_{\varepsilon} \to u$ in measure on $U$. Define the sequence of functionals ${\mathcal{E}}_{\varepsilon}$ like in \eqref{randomenergies}.
Due to \eqref{eq4.67},   arguing as in  \eqref{boundednessmodifiednenergy},  there exists a $C_0'>0$ such that
\begin{equation}
\label{eq4.70}
\sup_{\varepsilon>0}{\mathcal{E}}_{\varepsilon}(\hat{u}_{\varepsilon},U) \le C_0'.
\end{equation}
Then, combining equations \eqref{eq4.67}, \eqref{eq4.69}, and \eqref{eq4.70}, using $(f2)$, $(f6)$, and \eqref{boundednesslinearizedneergies} and arguing as in \emph{Step 1} of the proof of Theorem \ref{brittlemateriallimit}, we get
\begin{equation}
\label{eq4.71}
\begin{split}
 {\mathcal{E}}_{\varepsilon}(\hat{u}_{\varepsilon},U)\leq  {\mathcal{F}}^{\rm lin}_{\varepsilon}(u_{\varepsilon},U)+C c_2 \delta_\varepsilon^{\beta-\gamma}  +\sigma\Big( 2 C(\delta_\varepsilon^{2\gamma-\beta-\alpha}+ C  \delta^{\alpha/2}_{\varepsilon})  \Big)(C_0+ C_0').   
\end{split}
\end{equation}
Then applying the lower bound inequality of Theorem \ref{gammaconv} to ${\mathcal{E}}_{\varepsilon}$ and using \eqref{for sonn}, $\beta >\gamma$, as well as  $\hat{u}_{\varepsilon} \to u$ in measure on $U$,  we  conclude  
\begin{equation*}
{\mathcal{E}}_{\mathrm{hom}}(u,U)\leq \liminf\limits_{\varepsilon \to 0}{\mathcal{E}}_{\varepsilon}(\hat{u}_{\varepsilon},U)\leq \liminf\limits_{\varepsilon \to 0} \mathcal{F}_{\varepsilon}^{\mathrm{lin}}({u}_{\varepsilon}, \EEE U).   
\end{equation*}

\emph{Step 2: Existence of recovery sequences.}  Due to  \eqref{Condi},   there exists a  sequence $(u_{\varepsilon})_{\varepsilon} \subset PR_L(U)$ such that
$u_{\varepsilon}$ converges to $u$ in measure on $U$ and 
\begin{equation}
\label{eq4.72}
\lim\limits_{\varepsilon \to 0^+}\mathcal{E}_{\varepsilon}(u_{\varepsilon},U)=\mathcal{E}_{\text{hom}}(u,U), \quad \quad  \quad  \sup_{\eps >0}\Vert    \eps^{1+\kappa} \nabla u_\eps  \Vert_{L^\infty(U;\R^{d \times d})}   < +\infty
\end{equation}
for given $\kappa>0$. By the assumption  $\eps^{1+\kappa} \delta_\eps^{-\alpha/4} \to \infty$  we get  
\begin{align}\label{eq: verwenden!}
\delta_\eps^{\alpha/4}  \Vert \nabla u_\eps\Vert_{L^\infty(U)}  \to 0.
\end{align}
In particular, this implies ${\mathcal{F}}^{\rm lin}_{\varepsilon}(u_\eps,U) \EEE <\infty$ for  every $\varepsilon$ small enough as  $\Vert \nabla u_\eps\Vert_{L^\infty(U)} \le \delta_\eps^{-\alpha/4}$, cf.\   \eqref{for sure needed}. Define $y_\eps = {\rm id} + \delta_\eps^\alpha u_\eps$. As before choose $\gamma$  with $\beta > \gamma > \frac{1}{2}(\alpha + \beta)$.   By applying Proposition \ref{piecewiserigidapproximation},  we can find a sequence $(\Tilde{y}_{\varepsilon})_{\varepsilon} \subset PR_{SO(d)}(U)$ such that it holds
\begin{equation}
\label{eq4.75}
\Vert {y}_{\varepsilon}-\Tilde{y}_{\varepsilon}\Vert _{L^{\infty}(U;\mathbb{R}^d)}\leq C\delta_\varepsilon^{2\gamma - \beta},
\end{equation}
\begin{equation}
\label{eq4.76}
\mathcal{H}^{d-1}(J_{\Tilde{y}_{\varepsilon}}\setminus J_{{y}_{\varepsilon}})\leq C\delta_\varepsilon^{\beta-\gamma}
\end{equation}
for some  constant ${C}={C}(U,d )>0$. Moreover,  taking also  Remark \ref{nice remark} into account and using  \eqref{eq: verwenden!} we get 
\begin{equation}
\label{eq4.77}    
\delta_\eps^{ - 3\alpha/4} \Vert   \nabla \Tilde{y}_{\varepsilon} - \mathbb{I}\Vert _{L^{\infty}(U;\mathbb{R}^{d \times d})} \le C \delta_\eps^{ - 3\alpha/4} \Vert   \nabla  {y}_{\varepsilon} - \mathbb{I}\Vert _{L^{\infty}(U;\mathbb{R}^{d \times d})}   =   C  \delta_\eps^{\alpha/4}  \Vert \nabla u_\eps\Vert_{L^\infty(U)} \to 0.
\end{equation} 
 Hence, defining $\Tilde{u}_{\varepsilon}\colon U \to \mathbb{R}^d$ as $\Tilde{u}_{\varepsilon}=\delta^{-\alpha}_{\varepsilon}(\Tilde{y}_{\varepsilon}- {\rm id}) $, and using \eqref{eq4.75}--\eqref{eq4.76} we get 
\begin{equation}
 \label{eq4.78}
 \limsup_{\eps\to 0} \Vert {u}_{\varepsilon}-\Tilde{u}_{\varepsilon}\Vert _{L^{\infty}(U;\mathbb{R}^d)}\leq \lim_{\eps\to 0}  C \delta_\varepsilon^{2\gamma-\beta- \alpha} = 0,
\end{equation}
\begin{equation}
\label{eq4.79}
\limsup_{\eps\to 0}  \mathcal{H}^{d-1}(J_{\Tilde{u}_{\varepsilon}}\setminus J_{{u}_{\varepsilon}})\leq \lim_{\eps\to 0}   {C}\delta_\varepsilon^{\beta-\gamma} = 0,
\end{equation}
where we used $\beta > \gamma > \frac{1}{2}(\alpha + \beta)$. In particular, by \eqref{eq4.78}, $\Tilde{u}_{\varepsilon}$ converges to $u$ in measure on $U$. Moreover, by \eqref{eq4.77}     it holds   $\Vert \nabla\Tilde{u}_{\varepsilon}\Vert _{L^{\infty}(U;\mathbb{R}^{d \times d})} \le  \delta^{-\alpha/4}_{\varepsilon}$ for $\eps$ small enough, i.e.\   ${\mathcal{F}}^{\rm lin}_{\varepsilon}( \Tilde{u}_{\varepsilon}, \EEE U)<\infty$, see \eqref{for sure needed}. 
Hence, since $\Tilde{y}_{\varepsilon} \in PR_{SO(d)}(U)$, by $(\mathrm{W3})$ and \eqref{completenergy-lin}  we have
\begin{equation}
\label{eq4.82}
{\mathcal{F}}^{\mathrm{lin}}_{\varepsilon}(\Tilde{u}_{\varepsilon},U)=\int_{J_{{\Tilde{u}_{\varepsilon}}}\cap U}f\Big(\frac{x}{\varepsilon}, [\Tilde{u}_{\varepsilon}](x), \nu_{\Tilde{u}_{\varepsilon}}(x)\Big)\, \mathrm{d}\mathcal{H}^{d-1}(x).    
\end{equation}
Arguing as in \eqref{boundednessmodifiednenergy}, we can show that there is a constant $C_0'$ such that 
\begin{equation}
\label{eq4.83}
\sup_{\varepsilon \in (0,1)} {\mathcal{F}}^{\mathrm{lin}}_{\varepsilon}(\Tilde{u}_{\varepsilon},U) \le C_0'.
\end{equation}
Finally, using $(f2)$, $(f6)$, \eqref{eq4.78}, \eqref{eq4.79},   and \eqref{eq4.83} we obtain
\begin{equation}
\label{eq4.84}
\begin{split}
 {{\mathcal{F}}}^{\mathrm{lin}}_{\varepsilon}( \tilde{u}_{\varepsilon}, U)={{\mathcal{E}}}_{\varepsilon}(  \tilde{u}_{\varepsilon},  U)\leq {\mathcal{E}}_{\varepsilon}( {u}_{\varepsilon}, U)+{C}c_2 \delta_\varepsilon^{\beta- \gamma }  + \sigma\Big( 2 C \delta_\varepsilon^{2\gamma-\beta-\alpha}  \Big)(C_0+C_0'),
\end{split}
\end{equation}
where $C_0$ is chosen such that   $\mathcal{E}_{\varepsilon}(u_{\varepsilon},U) \le C_0$ for all $\eps >0$, see \eqref{eq4.72}. This along with \eqref{eq4.72}, \eqref{eq4.78}--\eqref{eq4.79},  and the fact that $\Tilde{u}_{\varepsilon}$ converges to $u$ in measure on $U$, concludes the proof.
\end{proof}

 We now turn our attention to condition \eqref{Condi} and show that it can be verified if \EEE $f$ does not depend on $\omega$.

\begin{proposition}[Recovery sequences with controlled  derivatives]\label{prop: recov}
 Let $L = \R^{d \times d}_{\rm skew}$.  In the setting of Theorem \ref{gammaconv} {with $f$ independent of $\omega$}, given $ U \in \mathcal{A}_0  $, $\kappa>0$, and $u \in PR_L(U)$  there exists a recovery sequence $(u_\eps)_\eps \subset PR_L(U) $  such that $u_\eps \to u$ in measure on $U$,  $\mathcal{E}_{\varepsilon}(u_\eps,U) \to \mathcal{E}_{\mathrm{hom}}(u,U)$, and 
\begin{equation}
\label{eq5.28XXX}
 \sup_{\eps >0}\Vert   \eps^{\kappa+1}  \nabla u_\eps  \Vert_{L^\infty(U ;\R^{d \times d})}   < +\infty.   
\end{equation}

\end{proposition}

\begin{proof}
 By assumption we get that  $f$ is periodic along the directions $e_1,\ldots,e_d$ with period one.   In fact, for   the case $(\tau_z)_{z \in \mathbb{Z}^d}$   this follows from \EEE the stationarity of $f$ with respect to $(\tau_z)_{z \in \mathbb{Z}^d}$,   whereas in the case of   a continuous group $(\tau_z)_{z \in \mathbb{R}^d}$ the density  $f$  is even   independent of the $x$ variable.

\emph{Step 1 (Preliminaries):}  
First, by the density result in  Theorem \ref{prop5.3} and a standard diagonal argument it is not restrictive to assume that the jump of $u \in PR_L(U)$  is a finite polyhedral Caccioppoli partition and that  $\nu_u \in \mathbb{Q}^{d}\cap \mathbb{S}^{d-1}$ $\mathcal{H}^{d-1}$-a.e.\ on $J_u$. We fix parameters $\eta,\rho>0$ to be specified below. Our goal is to construct a sequence $(u^{\eta,\rho}_\eps)_\eps \subset PR_L(  U )  $   such that $u^{\eta,\rho}_\eps \to u$ in measure on $U$ as $\eps \to 0$ and 
\begin{align}\label{Ant5}
{\rm (i)} & \ \  \limsup_{\eps \to 0}  \mathcal{E}_\eps( u^{\eta,\rho}_\eps , U ) \le \mathcal{E}_{\rm hom}(u, U ) + C\eta + C\sigma(\rho),\notag\\
{\rm (ii)} & \ \    \sup_{\eps >0}\Vert   \eps   \nabla   u^{\eta,\rho}_\eps    \Vert_{L^\infty(U ;\R^{d \times d})}  \le C_{\rho,\eta}
\end{align}
for a constant $C$ depending on $u$,  $C_{\rho,\eta}$ depending also  on $\rho$ and $\eta$, and $\sigma$ from $(f2)$. 
Then, the statement follows by a standard diagonal argument: for every $\kappa >0$, we can find sequences $\eta_{\varepsilon}$ and $\rho_{\varepsilon}$ converging to $0$ slow enough such that  $u_{\varepsilon}:=u_{\varepsilon}^{\eta_{\varepsilon},\rho_{\varepsilon}}$ converges to $u$ in measure and satisfies \eqref{eq5.28XXX} by \eqref{Ant5}(ii). Eventually, $\mathcal{E}_{\varepsilon}(u_\eps, U) \to \mathcal{E}_{\mathrm{hom}}(u,U)$ follows from \eqref{Ant5}(i), the fact that $\eta_\eps,\rho_\eps \to 0$, and the $\Gamma$-$\liminf$ inequality. 

Now, given  $\eta,\rho>0$, it is not restrictive to suppose that $\eps$ is small compared to $\rho$ and $\eta$.   In the following, $C>0$ always denotes a generic constant which may depend on $u$, but is independent of $\rho$, $\eta$, and $\eps$.

 We write $J_u = \bigcup_{k=1}^N \Gamma_k$, where each $\Gamma_k$ is contained in a $(d-1)$-dimensional plane with normal vector $\nu_k \in \mathbb{S}^{d-1} \cap \mathbb{Q}^d$.  We can choose rotations $R_k \in SO(d)$ and integers $t_k \in \mathbb{N}$ such that $R_ke_d=\nu_k$ and $t_kRe_i \in \mathbb{Z}^d$ for every $i=1,...,d$,   see e.g.\ \cite[Lemma 3.4]{Cristoferi}.  This implies that $f$ is periodic along the directions $R_ke_1 \ldots,R_ke_d$ with period $t_k$, i.e.\   
\begin{equation}\label{eq: perodii}
{f}(x,\zeta,\nu)={f}(x  + t_k R_k e_i,\zeta,\nu) \quad \text{for all $x \in \R^d$,  $\zeta \in \R^d\setminus \lbrace 0 \rbrace$, and $\nu \in \mathbb{S}^{d-1}$}.   
\end{equation}

\emph{Step 2 (Covering of $J_u$ with coarse- and fine-scale cubes):} In this step, we cover $J_u$ with two types of cubes: a first family with side length  $\sim \rho$ on which we will find almost minimisers of the homogenisation formula \eqref{eqhom}. Then, by a scaling argument we will transfer these optimal profiles to a second family of cubes with side length $\sim \eps$ which cover the jump set up to a small portion. This two-step procedure is necessary since, in order to get \eqref{Ant5}, we need to ensure that we work with a number of almost minimisers to problem \eqref{eqhom} which is \emph{bounded with respect to $\eps$}. Let us come to the details. Many objects in the following depend on $\rho$, $\eta$, and $\eps$, but we do not include this in the notation for convenience.  

For each $k = 1,\ldots,N$, we choose an arbitrary $p_k \in \Gamma_k$.  Given  $\rho>0$,   we select  a  finite number of  \emph{coarse scale}  base points $(x^k_i)_i \subset   p_k  +  R_k \rho  (\Z^{d-1}\times \lbrace 0 \rbrace)$  depending on $\rho$    such that for all $k=1,\ldots,N$ 
\begin{align}
\label{eq5.31bis}
{\rm (i)}  & \ \ {\rm dist}\big(x,  (x^k_i)_i  \big) \le C\rho \quad \text{for all $x \in \Gamma_k$},  \notag \\
{\rm (ii)} & \ \ \#  (x^k_i)_{i} \le C\rho^{-(d-1)}. 
\end{align}
As in Lemma \ref{lemma1}, we   exploit the property that $m_{\mathcal{E}}(u,U)=\lim\limits_{S \to \infty}m^S_{\mathcal{E}}(u,U)$, where 
\begin{equation}\label{eq: MMSSS}
m^S_{\mathcal{E}}(u,U)=\inf \big\{ \mathcal{E}(v,U): v \in PR_L( U ), \:\: \vert \nabla v \vert \leq S \:\: \text{and}\:\: v=u \:\: \text{near}\:\: \partial  U \big\}.   
\end{equation}
Thus, given $\eta>0$ and letting $\zeta_i^k = [u](x^k_i)$,  by virtue of \eqref{eqhom}, we can find   $t  \in \N$,  depending on $\eta$ and $\rho$, which is an integer multiple of  each $t_k$ for $k=1,\ldots, N$,   and  some  $S>0$  depending on $\rho$ and $\eta$ such that  for all $k=1,\ldots,N$ 
\begin{equation}
\label{eq5.31}
f_{\mathrm{hom}}(\zeta_i^k,\nu_k)\geq \frac{{m}^{S}_{\mathcal{E}}(u_{ tp_k,\zeta_i^k,\nu_k}, Q^{\nu_k}_{t}( t p_k  ))}{ t^{d-1}}-\eta \quad  \text{for every $i$.  } 
\end{equation}
Here, we note that the constants $t$ and $S$ depend on $\rho$ due to  \eqref{eq5.31bis}(ii). We are now in the position to introduce the second family of cubes.  Define 
\begin{align}\label{YYYYY}
\mathcal{Y}^k_i =  \lbrace  y \in \Gamma_k \colon    \argmin_{\iota}   |x^k_\iota - y |  = i \rbrace.  
\end{align}
Choose a plane $\Pi_{k}$ (depending on $\eps$),  with normal vector $\nu_k$, which has distance smaller than $2  t\eps$ to $\Gamma_k$ such that $\Pi_k$ contains points of $t\eps p_k  +    \eps R_k t \Z^d$, where $p_k$ was defined before \eqref{eq5.31bis}. Based on this, we   consider \emph{fine scale}  points $(y^{k}_j)_j \subset ( t\eps p_k  +    \eps R_k t \Z^d) \cap \Pi_k$ depending on $\rho$, $\eta$, and $\eps$  such that $(Q^{\nu_k}_{t\varepsilon}(y^{k}_j))_{j}$ is a finite family of  pairwise disjoint, adjacent   cubes  and such that the sets 
 $$T^k_i := {\rm int} \Big( \bigcup\nolimits_{y^k_j \in \mathcal{Y}^k_i}Q_{t \eps}^{\nu_k}(y_j^{k})  \Big), \quad \quad T^k = {\rm int}\Big(  \bigcup\nolimits_i   \overline{T^k_i} \Big)   $$ 
 satisfy
\begin{equation}
\label{eq5.31tris}
\mathcal{H}^{d-1}\big( \Gamma_k \setminus T^k  \big)\leq  C_{\rho,k} t\eps, \quad \quad   \mathcal{H}^{d-1}\big( \partial^\nu  T^k \big)  \le C_{\rho,k} t\eps, \quad \quad   \mathcal{L}^d( T^k  ) \le C_{\rho,k} t\eps   
\end{equation}
for all $k=1,\ldots,N$,  where the constant $C_{\rho,k}$ depends on $\rho$ and $\Gamma_k$, and $\partial^\nu$ denotes the faces of the boundary for which $\nu_k$ is \emph{tangential} to the boundary.  This can be done in such a way that the sets   $(\overline{T^k})_k$ are pairwise disjoint.

\emph{Step 3 (Competitors on coarse-scale and fine-scale cubes):}  
   We now introduce almost optimal competitors related to the coarse-scale cubes.  Recalling \eqref {eq: MMSSS}, we let $w^k_i \in PR_L(Q^{\nu_k}_{t}( tp_k))$ be such that $w^k_i=u_{tp_k,\zeta_i^k,\nu_k}$ in a neighbourhood of $\partial Q^{\nu_k}_{t}( tp_k)$, $\Vert \nabla w^k_i \Vert_\infty\leq S$, and
\begin{equation}
\label{eq5.32}
m^S_{\mathcal{E}}(u_{ tp_k,\zeta^k_i,\nu_k}, Q^{\nu_k}_{t}( tp_k ))\geq \mathcal{E}\big(w^k_i,Q^{\nu_k}_{ t}( tp_k )\big)-\eta.
\end{equation}
Let $v^{k,\varepsilon}_i(x):=w^k_i(x/\varepsilon) \in PR_L(Q^{\nu_k}_{t\varepsilon}( t \eps p_k))$ and let us extend it periodically with period $t\varepsilon$ along the directions described by the  basis  $R_ke_1,...,R_ke_d$.  By a scaling argument, we find
\begin{align}\label{eq: scalllll}
\eps^{d-1} \mathcal{E}\big(w^k_i,Q^{\nu_k}_{ t}( t p_k )\big) =    \mathcal{E}_{\varepsilon}\big( v^{k,\varepsilon}_i , Q^{\nu_k}_{ t\eps}( \eps tp_k )    \big).
\end{align}
 Since $f$ is periodic along the directions $R_ke_1,...,R_ke_d$ with period   $t$, see \eqref{eq: perodii} and recall $t/t_k \in \N$,   and $y^{k}_j \in   t  \eps p_k+    \eps R_k t \Z^d$,   we have 
\begin{align}\label{Ant1}
\mathcal{E}_{\varepsilon}\big( v^{k,\varepsilon}_i , Q^{\nu_k}_{ t\eps}( t\eps  p_k)    \big)= {\mathcal{E}_{\varepsilon}}\big(v^{k,\varepsilon}_i,Q^{\nu_k}_{ t\eps}(y^{k}_j))\big) \quad \text{for all } y^{k}_j. 
\end{align} 
For later, we also note that
\begin{align}\label{Ant2}
\text{$v^{k,\eps}_i=u_{y^k_j,\zeta_i^k,\nu_k}$ in a neighbourhood of $\partial Q^{\nu_k}_{ t\eps}(y^k_j)$}, \quad \quad \Vert \nabla  v^{k,\varepsilon}_i \Vert_{L^\infty(\R^d)} \le S/\eps.  
\end{align}

\emph{Step 4 (Definition of the recovery sequence):}   
We introduce some further notation. We denote by  $u^\pm_k(y) =  M^\pm_k y + b^\pm_k$, $M^\pm_k \in \R^{d \times d}_{\rm skew}$, $b^\pm_k \in \R^d$,   the two rigid motions related to $u$ on both sides of $\Gamma_k$. We define    $(\zeta_i^{k})^\pm :=(u(x^k_i))^\pm = M^\pm_k x^k_i + b^\pm_k$ and note that  $\zeta_i^k  = (\zeta_i^{k})^+ -  (\zeta_i^{k})^-$, with $\zeta_i^k$ introduced after \eqref{eq: MMSSS}. 
Moreover, recalling \eqref{YYYYY}, for each $i$ and $k$ we  define the two sets 
\begin{align}\label{EEEEEEE}
E^{k,\pm}_{i }= \bigcup_{y^k_j \in \mathcal{Y}^k_{i}}  \{ y \in Q^{\nu_k}_{t\eps }(y^k_j) \colon  v^{k,\varepsilon}_{i}(y)=(\zeta^k_{i})^{\pm} -  (\zeta_i^{k})^- \},
\end{align} 
and note that the boundaries  $\partial Q^{\nu_k}_{t\eps}(y^k_j)$ are contained in the closure of $E^{k,+}_{i } \cup E^{k,-}_{i }$ by the first property in \eqref{Ant2}. 
 Now, we define  $u^{\eta,\rho}_{\varepsilon} \in PR_{L}( U )$ as 
\begin{equation}
\label{eq5.34}
u^{\eta,\rho}_{\varepsilon}(y)= \begin{cases}
 v^{k,\varepsilon}_{i}(y)+ (\zeta_i^{k})^-  +    (u^+_k(y)-  (\zeta^k_i)^+)  \chi_{E^{k,+}_{i }}(y) +   (u^-_k(y)-  (\zeta^k_i)^-)  \chi_{E^{k,-}_{i }}(y)  & \text{if } y \in T^k_i \text{ $\forall \, i$, $k$}    \\
   u(y)    & \text{if }   y \in U \setminus  \bigcup\nolimits_{i,k}  T^k_i.
\end{cases}
\end{equation}
Since each function $(v^{k,\varepsilon}_{i})_{i,k}$ lies in $PR_L$ and the sets $(E^{k,\pm}_{j })_{j,k}$ have finite perimeter, we clearly get that $u^{\eta,\rho}_{\varepsilon} \in PR_{L}(U)$. In view of \eqref{Ant2} and the choice of the sets in \eqref{EEEEEEE}, this definition ensures that $u^{\eta,\rho}_{\varepsilon}$ has no jump on $\bigcup_j  \partial Q^{\nu_k}_{t\eps}(y^k_j) \setminus \partial^\nu  T^k $.

\emph{Step 5 (Estimate on the energy and convergence in measure):} 
We start estimating the energy on the sets $ (T^k)_{k=1}^N$ separately.   First, by  \eqref{eq5.31}, \eqref{eq5.32}, \eqref{eq: scalllll},  and \eqref{Ant1} we get 
\begin{equation}
\label{eq5.33}
\eps^{d-1} t^{d-1} \big(f_{\mathrm{hom}}(\zeta_i^k,\nu_k) + 2\eta\big)  \ge   {\mathcal{E}_{\varepsilon}}\big(v^{k,\varepsilon}_i,Q^{\nu_k}_{ t\eps}(y^{k}_j)\big) \quad \text{for all } y^{k}_j \in \mathcal{Y}^k_i.
\end{equation}
Since $\nabla u$ takes only finitely many values and is thus bounded,  by using \eqref{eq5.31bis}(i) and the definition of $\mathcal{Y}^k_i $ in \eqref{YYYYY},   it holds that  
\begin{align}\label{Ant3}
|(u^\pm_k(y)-  (\zeta^k_i)^\pm| \le C\rho \quad \text{for all $y \in \mathcal{Y}^k_i $}, \quad  |[u](y) -   \zeta_i^k| \le C\rho \quad \text{for all $y \in \mathcal{Y}^k_i  \cap \Gamma_k $}.
 \end{align}
 Then, using $(f2)$ for $f_{\rm hom}$    we find
 \begin{equation*}
(1- \sigma(C \rho) )\eps^{d-1} t^{d-1}  f_{\mathrm{hom}}(\zeta_i^k,\nu_k)\leq  (1+\sigma(C \rho)){\mathcal{E}_{\rm hom}}\big(u, Q^{\nu_k}_{ t\eps}(y^{k}_j)\big)  \quad \quad \text{for all } y^{k}_j \in \mathcal{Y}^k_i.
 \end{equation*} 
For $\rho$ small enough we find $(1+\sigma(C \rho)) / (1-\sigma(C \rho)) \le (1+C\sigma(C \rho))$, and thus 
 \begin{equation}
 \label{eq5.35}
\eps^{d-1} t^{d-1}  f_{\mathrm{hom}}(\zeta_i^k,\nu_k)\leq  (1+C\sigma(C \rho)){\mathcal{E}_{\rm hom}}\big(u, Q^{\nu_k}_{ t\eps}(y^{k}_j)\big)  \quad \quad \text{for all } y^{k}_j \in \mathcal{Y}^k_i.
 \end{equation} 
In a similar fashion, again using \eqref{Ant3}, $(f2)$, and \eqref{eq5.34} we get 
 \begin{equation}
 \label{eq5.35XXX}
{\mathcal{E}_\eps}\big(u^{\eta,\rho}_{\varepsilon}, Q^{\nu_k}_{ t\eps}(y^{k}_j)\big)  \leq  (1+C\sigma(C \rho)){\mathcal{E}_\eps}\big(  v^{k,\varepsilon}_{i},  Q^{\nu_k}_{ t\eps}(y^{k}_j)\big)  \quad \quad \text{for all } y^{k}_j \in \mathcal{Y}^k_i.
 \end{equation} 
 Then, summing over all $y^{k}_j \in \mathcal{Y}^k_i$, by \eqref{eq5.33}, \eqref{eq5.35}, \eqref{eq5.35XXX},  and the fact that  $ \# \mathcal{Y}^k_i = \mathcal{H}^{d-1}(J_u \cap T^k_i)  (t\eps)^{-(d-1)}$, we find
\begin{align*}
{\mathcal{E}_{\varepsilon}}\big(u^{\eta,\rho}_{\varepsilon},  T^k_i\big) & \le     (1+C\sigma(C \rho))   {\mathcal{E}_{\rm hom}}\big(u,  T^k_i\big) +   C\mathcal{H}^{d-1}(J_u \cap T^k_i)\eta  \\
& \le      {\mathcal{E}_{\rm hom}}\big(u,  T^k_i\big) +   C\mathcal{H}^{d-1}(J_u \cap T^k_i)\big( \eta +\sigma(C \rho)\big),
 \end{align*}
 where in the second step we  used $(f6)$, and we  used \eqref{Ant2} to see that   $u^{\eta,\rho}_{\varepsilon}$ does not exhibit jumps on $\bigcup_j \partial Q^{\nu_k}_{ t\eps}(y^k_j)  \setminus\partial^\nu  T^k $.  Summing over $i$ and using again \eqref{Ant2}, \eqref{eq5.34}, $(f6)$, as well as the second property of \eqref{eq5.31tris}   we find 
$$ {\mathcal{E}_{\varepsilon}}\big(u^{\eta,\rho}_{\varepsilon},  \overline{T^k}  \big) \le      {\mathcal{E}_{\rm hom}}\big(u,  \overline{T^k}\big) +   C\mathcal{H}^{d-1}(J_u \cap T^k)\big( \eta +\sigma(C \rho)\big) +  C_{\rho,k} t\eps   ,$$
where we again used \eqref{Ant2} to see that   $u^{\eta,\rho}_{\varepsilon}$ does not exhibit jumps on $\partial T^k_i \setminus\partial^\nu  T^k $. 
Summing over all $k =1, \ldots, N$ and using the first property of  \eqref{eq5.31tris} together with $(f6)$  we find
$$ {{\mathcal{E}_{\varepsilon}}\big(u^{\eta,\rho}_{\varepsilon},   U  \big) \le      {\mathcal{E}_{\rm hom}}\big(u, U \big) +   C\mathcal{H}^{d-1}(J_u)\big( \eta +\sigma(C \rho)\big) +  2NC_{\rho,k} t\eps   , } $$
where we have used that the sets $(\overline{T^k})_k$ are pairwise disjoint. Replacing $\rho$ by $C/\rho$, this is the desired estimate \eqref{Ant5}(i). Moreover, \eqref{Ant5}(ii) follows from \eqref{Ant2}, \eqref{eq5.34}, the dependence of $S$ on $\eta$ and $\rho$,  and the fact that $\nabla u$ is uniformly bounded.   Eventually, by the third bound in  \eqref{eq5.31tris} and  \eqref{eq5.34} we get that $\lim_{\varepsilon \to 0}\mathcal{L}^d(\{x \in  U \colon u_{\varepsilon}^{\eta,\rho}(x) \neq u(x)\})=0$, i.e.\    $u_{\varepsilon}^{\eta,\rho}$ converges to $u$ in measure on $ U $ as $\eps \to 0$.  This concludes the proof. 
\end{proof}

\section*{Acknowledgements} 
This work was supported by the RTG 2339 “Interfaces, Complex Structures, and Singular Limits”
of the German Science Foundation (DFG). The support is gratefully acknowledged. This work was supported  by the Deutsche Forschungsgemeinschaft (DFG, German Research Foundation) under Germany's Excellence Strategies EXC 2044 -390685587, Mathematics M\"unster: Dynamics--Geometry--Structure.

\appendix

\section{Density of functions with polyhedral jump sets}
 
In this section we state a density result in $PR_L$ which is an adaptation of \cite{BraConGar16}. 

\begin{theorem}[Density of functions with polyhedral jump sets]
\label{prop5.3}
Let $L=SO(d)$ or $L=\mathbb{R}_{\mathrm{skew}}^{d \times d}$.   Given $U \in \mathcal{A}_0$ and a continuous and bounded function $\psi \colon U \times \R^d\setminus \lbrace 0 \rbrace \times \mathbb{S}^{d-1}$,  consider the  energy  
\begin{align}\label{E functi}
\mathcal{E}(u) =  \int_{U \cap J_u}{\psi(x, [u](x),\nu_{u}(x))\,\mathrm{d}\mathcal{H}^{d-1}(x)}
\end{align}
for $u \in PR_L(U)$.  Then for every $u \in PR_L(U)$  there exists a sequence $(u_n)_n \subset PR_L(U) \cap L^{\infty}(U;\mathbb{R}^d)$ such that $u_n$ has polyhedral jump set, $\nu_{u_n}\in \mathbb{Q}^d \cap \mathbb{S}^{d-1}$, $u_n \to u$ in  measure,   and
\begin{equation}\label{use later}
\limsup\limits_{n \to \infty}\mathcal{E}(u_n)\leq \mathcal{E}(u).    
\end{equation}
\end{theorem}
\begin{proof} We divide the proof intro three steps. 

 \emph{Step 1:} Let $u(x)=\sum_{j \in \mathbb{N}}(M_j x + b_j)\chi_{P_j}(x)$, with $M_j \in L$ and $b_j \in \mathbb{R}^d$, and a Caccioppoli partition $(P_j)_j$   of $U$. For every $N \in \mathbb{N}$ define $u^N(x)= \sum^N_{j=1}(M_j x+b_j)\chi_{P_j}(x)+\sum_{j >N}   M_1x   \chi_{P_j}(x)$\EEE. Notice that $u^N \to u$ in measure on $U$ as $N \to \infty$. We have 
 \begin{equation*}
 \mathcal{E}(u^N,U)\leq \mathcal{E}(u,U)+ \Vert \psi \Vert_\infty \sum_{j>N}\mathcal{H}^{d-1}(\partial^* P_j),   
 \end{equation*}
and since $\sum_{j \in \mathbb{N}}\mathcal{H}^{d-1}(\partial^* P_j)<\infty$ we obtain
 \begin{equation}
 \label{eq5.28}
  \limsup\limits_{N \to \infty}\mathcal{E}(u^N,U)\leq \mathcal{E}(u,U).
 \end{equation}
 Hence, since $u^N$ is essentially defined on $N+1$ Caccioppoli sets, up to a diagonal argument, it suffices to prove the proposition   for functions with finite Caccioppoli partitions. 
 
 \emph{Step 2}: Consider $u(x)=\sum^N_{j=1}(M_j\,x+b_j)\chi_{P_j}(x)$. Without restriction we choose a pairwise distinct representation. i.e.\ $\mathcal{H}^{d-1}(J_u \triangle (U \cap \bigcup^N_{j=1}\partial^* P_j))=0$. Define 
 $$p=\sum^N_{j=1}M_j \chi_{P_j} \in PR_0(U;\mathbf{M}), \quad \quad  q=\sum^N_{j=1}b_j \chi_{P_j} \in PR_0(U; \mathbf{b} ),$$
 where $\mathbf{M} = (M_j)_{j=1}^N \subset  L$ and $\mathbf{b} = (b_j)_{j=1}^N \subset  \R^d$. 
  Notice that $J_p \cup J_q \subset  J_u$ up to  $\mathcal{H}^{d-1}$-negligible sets. By virtue of \cite[Theorem 2.1]{BraConGar16}, there exist sequences $(p_n)_n \subset PR_0(U; \mathbf{M})$ and sequence $(q_n)_n \subset PR_0(U; \mathbf{b})$,  having polyhedral jump set, such that $p_n \to p$ in $L^1(U;\mathbb{R}^{d \times d})$ and $q_n \to q$ in $L^1(U;\mathbb{R}^d)$. A careful inspection of the proof of \cite[Theorem 2.1]{BraConGar16} shows that, since $p$ and $q$ share the same Caccioppoli partition, also the two sequences $(p_n)_n$ and $(q_n)_n$ can be chosen such that $p_n$ and $q_n$ have the same  partition. In fact, the construction of the polyhedral jump set  takes into account only the   partition on which the original function is defined, not its values. Now, define  $u_n(x)=p_n(x) x +q_n(x) \in PR_L(U) \cap L^{\infty}(U;\mathbb{R}^d)$. Since $p_n \to p$ in $L^1(U;\mathbb{R}^{d \times d})$ and $q_n \to q $ in $L^1(U;\mathbb{R}^d)$, we have $u_n \to u$ in $L^1(U;\mathbb{R}^d)$. Arguing like in the proof of \cite[Corollary 2.4]{BraConGar16}, it can be shown that $\limsup_{n \to \infty}\mathcal{E}(u_n)\leq \mathcal{E}(u)$,  i.e.\ \eqref{use later} holds. We now give the details of this argument. \\
  By \cite[Theorem 2.1]{BraConGar16} there also exists a sequence of functions $(f_n)_n \subset C^1(\mathbb{R}^d;\mathbb{R}^d)$, with inverse also in $C^1$, such that $f_n \to \mathrm{id}$ strongly in $W^{1,\infty}(\mathbb{R}^d;\mathbb{R}^d)$ and $ \vert   D  \tilde{p}_n   - D  p_n \vert (U) \to 0$, $\vert   D   \tilde{q}_n   - D  q_n  \vert (U) \to 0$,  where $\tilde{p}_n := p \circ f_n$ and $\tilde{q}_n := q \circ f_n$ are functions defined on $f_n^{-1}(U)$, and  the measures $D    \tilde{p}_n $, $D \tilde{q}_n$ are implicitly extended by $0$ to $\R^d$. Since $ \vert D \tilde{p}_n - Dp_n \vert (U) \to 0$ and the functions $p_n$ attain only the finitely many different values $\mathbf{M}$, we have  
\begin{equation}\label{eq: andnow}
\mathcal{H}^{d-1}\big(U \cap (J_{\Tilde{p}_n}\triangle J_{p_n})\big)  +  \mathcal{H}^{d-1}\big(\big\{x \in J_{ \tilde{p}_n}\cap J_{p_n}\cap U: p^{\pm}_n(x)\neq  \tilde{p}_n^{\pm}  (x)\big\}\big)  \to 0 
\end{equation}
 as $n \to \infty$. The same holds for $q$ in place of $p$.  Define  also $\Tilde{u}_n(x)=\tilde{p}_n(x) x + \tilde{q}_n (x)$ and note  $J_{\Tilde{u}_n}\subset J_{\tilde{p}_n} \cup J_{\tilde{q}_n}$.  This along with \eqref{eq: andnow} (also for $q$)  implies
\begin{equation*}
\mathcal{H}^{d-1}\big(U \cap (J_{\Tilde{u}_n}\triangle J_{u_n})\big)+ \mathcal{H}^{d-1}\big(\big\{x \in J_{\Tilde{u}_n}\cap J_{u_n}\cap U: u^{\pm}_n(x)\neq \Tilde{u}^{\pm}_n\big\}\big) \to 0.
\end{equation*}
Then, by the boundedness of $\psi$,   in order to conclude \eqref{use later}, it suffices to show
\begin{equation}
\label{eq5.29}
\limsup_{n \to \infty} \int_{J_{\Tilde{u}_n}\cap U}\psi(x,[\Tilde{u}_n],\nu_{\Tilde{u}_n})\, \mathrm{d}\mathcal{H}^{d-1}(x) \le \mathcal{E}(u).
\end{equation}
It can be checked that $\nu_{\tilde{u}_n }(x)=\frac{Df_n^T (x) \nu_u(f_n(x))}{\vert Df_n^T (x) \nu_u(f_n(x))\vert}$ and then  by the change of variable formula \cite[Theorem 2.91]{ambrosio2000fbv} we  have
\begin{equation*}
 \int_{ J_{\tilde{u}_n}\cap U} \hspace{-0.1cm}\psi(x,[\Tilde{u}_n],\nu_{\tilde{u}_n})\, \mathrm{d}\mathcal{H}^{d-1}(x)=\int_{ J_{u}\cap  f_n(U)} \psi  \big(f^{-1}_n(x),[p(x)f^{-1}_n(x)+q(x)],\nu^u_{n}  (x)  \big)\, J_{d-1}\mathrm{d}^{J_u} f^{-1}_n\mathrm{d}\mathcal{H}^{d-1} (x), 
\end{equation*} 
where $\nu^u_n := \nu_{\tilde{u}_n} \circ f_n^{-1} $ is the normal to $J_{\tilde{u}_n}$ transformed by $f_n$  and $J_{d-1} \, \mathrm{d}^{J_u}f_n^{-1}$ is the Jacobian of the tangential differential of $f^{-1}_n$. Notice now that $\nu^u_n \to \nu_u$,  $f_n^{-1} \to \mathrm{id}$, and $\nabla f_n \to \mathbb{I}$ uniformly. Hence, by dominated convergence,  $\mathcal{H}^{d-1}(J_u \setminus f_n(U)) \to 0$,  and the continuity and boundedness  of $\psi$ we conclude \eqref{eq5.29}.

\emph{Step 3}: Finally, as  $\mathbb{S}^{d-1}\cap \mathbb{Q}^d$ is dense in $\mathbb{S}^{d-1}$, each function $w$ with polyhedral jump set can be approximated by functions $(w_n)_n$,   still with polyhedral jump set,  such that $\nu_{w_n}\in \mathbb{Q}^d \cap \mathbb{S}^{d-1}$ for $\mathcal{H}^{d-1}$-a.e.\ point of $J_{w_n}$, $w_n \to w$ in measure on $U$ and $\mathcal{E}(w_n) \to \mathcal{E}(w)$. Then, by a  diagonal argument we can assume that the sequence in Step 2 has this additional property.
\end{proof}

\section{$\Gamma$-convergence with boundary data}

In this section we present and prove a version of Theorem \ref{gammaconv} featuring boundary data.

\begin{theorem}[$\Gamma$-convergence with boundary data]\label{th_ bundary data}
Let $L=SO(d)$ or $L=\mathbb{R}_{\rm skew}^{d \times d}$.  Let $f$ be a stationary random surface density with respect to a group $(\tau_z)_{z \in \mathbb{Z}^d}$ (resp. $(\tau_z)_{z \in \mathbb{R}^d}$) of $\mathbb{P}$-preserving transformations on $(\Omega,\mathcal{I},\mathbb{P})$. Let $U \in \mathcal{A}_0$ and let $V \subset U$  be   such that $V  \in \mathcal{A}_0$. Consider  $u_0 \in PR_L(U)$ as introduced below \eqref{BBBC}. Finally, for every $\varepsilon>0$, define
 $\mathcal{E}_{\varepsilon}', \mathcal{E}'_{\mathrm{hom}} \colon \Omega \times L^{0}(\mathbb{R}^d;\mathbb{R}^d)\to [0,\infty]$ as
\begin{equation}
\label{randomenergies_boundary data}
\mathcal{E}'_\eps[\omega](u)=\begin{cases}
\mathcal{E}_{\varepsilon}[\omega](u,U) &  u =u_0\:\: \text{on}\:\: V,   \\
 +\infty  & \text{otherwise},
\end{cases}
\end{equation}
and
\begin{equation}
\label{homogenisedenergy_boundary data}
\mathcal{E}'_{\mathrm{hom}}[\omega](u)=\begin{cases}
\mathcal{E}_{\mathrm{hom}}[\omega](u,U)  &     u  =u_0\:\: \text{on}\:\: V,\\
+\infty &  \text{otherwise}.
\end{cases}
\end{equation}
Then, with $\Omega'$ from Theorem \ref{HomogenizationFormula}, for every $\omega \in \Omega'$ it holds that 
\begin{equation}\label{gamma_boundarydata}
\mathcal{E}'_{\varepsilon}[\omega]\:\: \Gamma\text{-converge}\:\:\text{to}\:\: \mathcal{E}'_{\mathrm{hom}}[\omega]  \ \ \ \ \text{with respect to convergence in measure on $U$}.
\end{equation}
\end{theorem}
 
\begin{proof}
In order to prove \eqref{gamma_boundarydata}, we need  to check that
\begin{equation}
\label{eq3.13}
\Gamma-\limsup\limits_{\eps \to 0}\mathcal{E}_{\varepsilon}'[\omega](u)\leq \mathcal{E}_{\mathrm{hom}}[\omega](u),    
\end{equation}
for every $\omega \in \Omega'$ and $u \in PR_L(U)$. Let $u \in PR_L(U)$ and let $U':= U \setminus \overline{V} \in \mathcal{A}_0$. We fix $\omega \in \Omega'$, which we drop from the notation for simplicity. 

  By Lemma \ref{prop5.3-cor} (see below)   and a standard diagonal argument, we can assume that  $u = u_0$ in a neighbourhood of $V$ and thus  $\mathcal{H}^{d-1}(J_u \cap \partial V) =  \mathcal{H}^{d-1}(J_{u_0} \cap \partial V) = 0 $. Using Theorem \ref{gammaconv} we know that there exists a recovery sequence $(u_\eps)_{\eps}$ for $u$ such that $u_{\eps}\to u$ in measure on $U$ and
\begin{equation}
\label{eq4.64}
    \lim\limits_{\eps \to 0}\mathcal{E}_{\varepsilon}(u_{\eps},U)=\mathcal{E}_{\mathrm{hom}}(u,U).
\end{equation}
Let $\eta>0$. We choose sets  $A,A',B \in \mathcal{A}_0$ with    $A \subset \subset U' = U \setminus \overline{V}$,   $A' \subset \subset A$,  and $U \cap \overline{U'\setminus A'} \subset  B \subset U$ such that $U \setminus \overline{A} \in \mathcal{A}_0$, $U \setminus \overline{A'\cup B} \in \mathcal{A}_0$, and 
\begin{align}\label{3cond}
\mathcal{H}^{d-1}\big(J_u \cap B \big)\le  \eta,  \quad  \quad \mathcal{H}^{d-1}\big((\partial A \cup \partial B) \cap J_u\big) = 0. 
\end{align}
Define the function $v_\eps \in PR_L(B)$ by $v_\eps = u$. We apply the fundamental estimate in \cite[Lemma 4.5]{friedrich2018piecewise}  for the functions $u_\eps \in PR_L(A)$ and $v_\eps \in PR_L(B)$ as well as  the parameter $\eta >0$. We find a function $w_\eps \in PR_L(A'\cup B)$ such that $w_\eps = v_\eps$ on $B\setminus A$ and 
\begin{equation}
\label{eq4.65}
\mathcal{E}_{\varepsilon}(w_\eps, A' \cup B)\leq (1+C\eta)\big(\mathcal{E}_{\varepsilon}(u_\eps,A) +  \mathcal{E}_{\varepsilon}(v_\eps,B)\big) + C\eta + C_\eta\mathrm{err}(\varepsilon),  
\end{equation}
where $C$ depends on $A',A',B$, and  $C_\eta$ depends additionally on $\eta$, as well as $\mathrm{err}(\varepsilon) \to 0$ as $\eps \to 0$. Here, we use that $u_\eps\to u$ on $A $, $v_\eps = u$ on $B$, and thus $u_\eps - v_\eps \to 0$ on  $(A \setminus A') \cap B$. (We also note that \cite[(4.6)]{friedrich2018piecewise} can be verified   since $v_\eps =u$ on $B \setminus \overline{A'}$.)  
In particular, we have $w_\eps = u = u_0$ on $(A' \cup B) \cap V = B \cap V \EEE \subset B \setminus A$. As $A' \cup B \supset   U'$ (since $U \cap \overline{U'\setminus A'} \subset  B$), we can extend $w_\eps$ to $U$ by setting
\begin{equation*}
{u}^{\eta}_\eps(x)=\begin{cases}
        w_\eps(x) &\text{if}\:\: x \in A' \cup B, \\
        u_0(x) & \text{if}\:\: x \in U \setminus (A' \cup B)
    \end{cases}
\end{equation*}
and we get
\begin{align}\label{tocombine1}
\mathcal{E}_{\varepsilon}(u^\eta_\eps, U) \le \mathcal{E}_{\varepsilon}\big(w_\eps, A' \cup B\big) + \mathcal{E}_{\varepsilon}\big(u_0, U \setminus \overline{A' \cup B}\big),
\end{align}
i.e.\ no additional jump is introduced on $\partial B \cap V$. Since $w_\eps = u_0$ on $(A' \cup B) \cap V$, we find $u^\eta_\eps = u_0$ on $V$. We finally estimate the energy. For this, we will use that 
\begin{align}\label{tocombine2}
& \lim_{\eps\to 0} \mathcal{E}_{\varepsilon}(u_\eps, A) = \lim_{\eps\to 0} \mathcal{E}_{\rm hom}(u, A),
\quad \quad 
\lim_{\eps\to 0} \mathcal{E}_{\varepsilon}(u_\eps,  U\setminus \overline{A' \cup B}) = \lim_{\eps\to 0} \mathcal{E}_{\rm hom}(u, U\setminus \overline{A' \cup B}). 
\end{align}
In fact, as $A$ and $U \setminus \overline{A}$ are open and Lipschitz,  we find  
$$\liminf_{\eps\to 0 } \mathcal{E}_{\varepsilon}(u_\eps, A) \ge \mathcal{E}_{\mathrm{hom}}(u,A), \quad\quad \liminf_{\eps \to 0} \mathcal{E}_{\varepsilon}(u_\eps, U \setminus \overline{A}) \ge \mathcal{E}_{\mathrm{hom}}(u,U \setminus \overline{A}).   $$
By $\mathcal{H}^{d-1}(\partial A \cap J_u) = 0$, see  \eqref{3cond}, we get $\mathcal{E}_{\rm hom}(u, U  ) = \mathcal{E}_{\rm hom}(u, A) + \mathcal{E}_{\rm hom}(u, U \setminus \overline{A})$. This  along  with  
the inequality
\begin{equation*}
\mathcal{E}_{\mathrm{hom}}(u,A)+\mathcal{E}_{\mathrm{hom}}(u,U \setminus \overline{A})\geq\limsup\limits_{\varepsilon \to 0}\big(\mathcal{E}_{\varepsilon}(u_\eps, A)+\mathcal{E}_{\varepsilon}(u_\eps, U \setminus \overline{A})\big)\geq \limsup\limits_{\varepsilon \to 0}\mathcal{E}_{\varepsilon}(u_\eps, A)+\mathcal{E}_{\mathrm{hom}}(u,U\setminus \overline{A}),
\end{equation*}
  (use \eqref{eq4.64} in the first step) 
shows \eqref{tocombine2} on $A$. The argument for   $U\setminus \overline{A' \cup B}$ is the same, again using  \eqref{3cond}.  
Now, combining  \eqref{3cond}, \eqref{eq4.65}, \eqref{tocombine1}, \eqref{tocombine2}, $(f6)$, and using $v_\eps = u$ we find 
$$\limsup_{\eps \to 0} \mathcal{E}_{\varepsilon}(u^\eta_\eps, U) \le (1+C\eta)\big(\mathcal{E}_{\rm hom}(u,A) +  c_2 \eta \big) + C\eta + \limsup_{\eps \to 0} \mathcal{E}_{\varepsilon}\big(u_0, U \setminus \overline{A' \cup B}\big).  $$
As $u_\eps = u_0$ on $U \setminus \overline{A' \cup B}$, using again \eqref{tocombine2} and $A \cap (U \setminus \overline{A' \cup B}) = \emptyset$, we deduce
\begin{align*}
\limsup_{\eps \to 0} \mathcal{E}_{\varepsilon}(u^\eta_\eps, U) & \le (1+C\eta)\big(\mathcal{E}_{\rm hom}(u,A) +  c_2 \eta \big) + C\eta + \mathcal{E}_{\rm hom}(u, U \setminus \overline{A' \cup B})   \le (1+C\eta) \mathcal{E}_{\rm hom}(u, U )   + C\eta.
\end{align*}
As $\eta>0$ was arbitrary, the proof is concluded. 
\end{proof}

 Arguments similar to the ones of \cite {BraConGar16} and Theorem \ref{prop5.3} lead to the following lemma.

\begin{lemma}[Density]
\label{prop5.3-cor}
Let $L=SO(d)$ or $L=\mathbb{R}_{\mathrm{skew}}^{d \times d}$.  Let $U
\in \mathcal{A}_0$ and let  $V \subset U$ be    such that $U \setminus
\overline{V} \in \mathcal{A}_0$. Let $u_0 \in PR_L(U)$ be 
Lipschitz   in a neighborhood  of $V$ in  $U$. Then, for every
$u \in PR^{u_0}_L(U)$   there exists a sequence $(u_n)_n \subset
PR_L(U)$ and a sequence of neighbourhoods $N_n \supset V$  of $V$ in $U$
such that $N_n \in \mathcal{A}_0$, $u_n = u_0$ on $N_n$, $u_n \to u$ in 
measure on $U$,  and
$\limsup_{n \to \infty}\mathcal{E}(u_n)\leq \mathcal{E}(u)$, where
$\mathcal{E}$ is a functional as in \eqref{E functi}.
\end{lemma}

We briefly sketch the proof of Lemma \ref{prop5.3-cor} without
going into details. For simplicity we can assume that $V$ is connected
   since   the same construction can be applied locally if $V$
consists of several components. We also assume that $u
\in PR_L(\mathbb{R}^d)$ since the case $u \in PR_L(U)$ can be obtained
by applying an analogue version of \cite[Theorem
2.1]{BraConGar16}. Furthermore, strictly speaking, we should apply the
same technique of Theorem \ref{prop5.3} in order to pass from $u \in
PR_L(\mathbb{R}^{d})$ to two finite valued functions $p \in
PR_0(\mathbb{R}^d;L)$ and $q \in
PR_0(\mathbb{R}^d;\mathbb{R}^d)$.  This   is omitted  as  details
are already provided in the proof of Theorem \ref{prop5.3}.

Notice that $\partial V \cap U$ is Lipschitz regular by 
hypothesis. Hence, for ($\mathcal{H}^{d-1}$-almost) every $y \in
\partial V \cap U$, let us denote by $\nu_y$ the  inner 
 normal vector to $V$ in the point $y$. Then, for
$\mathcal{H}^{d-1}$-a.e $y \in \partial V \cap U$, the local
construction in Step 1 of \cite[Theorem 2.2]{BraConGar16} can be
repeated, up to replacing $J_u$ with $(\partial V \cap U) 
-c\varepsilon^2 \nu_y\EEE$ for some constant $c>0$ depending on $y$, where $\eps>0$ is a small parameter introduced in
\cite[(2.2)]{BraConGar16}.  Then, we may proceed along the lines of
Steps 2--3 in \cite[Theorem 2.1, Theorem 2.2]{BraConGar16}.
In particular,  in Step 2 a covering argument is applied.   In our case, this leads   to
points  $x_1,...,x_M$ in  $\partial V \cap U$ and
corresponding  radii   $r_1,...,r_M$  such that
\begin{equation*}
\mathcal{H}^{d-1}\Big((\partial V \cap U)\setminus
\bigcup^M_{i=1}B_{r_i}(x_i)\Big)<\varepsilon.
\end{equation*}
In our analogue version of Step 3 of \cite[Theorem 2.2]{BraConGar16}, since the boundary datum $u_0$ is affine on
$V$, without increasing its energy we can modify the  function 
constructed in \cite{BraConGar16} in such a way to be equal to $u_0$ in 
each polyhedral cell (see  \cite[Lemma 2.6]{BraConGar16}) lying  in 
$\{x \in B_{r_i}(x_i) \colon (x-x_i)\cdot \nu_{x_i}\geq - c
\varepsilon^2\}$, so in particular  in a neighbourhood of $V$   covered  by the collection of cells   between $\partial V \cap
B_{r_i}(x_i)$ and the hyperplane $\lbrace x\in B_{r_i}(x_i) \colon
(x-x_i)\cdot \nu_{x_i} =  - c \varepsilon^2\rbrace$.


\section{Example \ref{main-example}}\label{appendix-example}

Given $\rho>0$, we perform the computation on $Q_{2\rho}(0)$ for notational convenience. We will show that 
\begin{align}\label{11111}
\limsup_{\eps\to 0 }  m^{L}_{\mathcal{E}_\eps}\big(u_{0, e_1,e_2  },Q_{2\rho}(0)\big)  \le  2\rho(13 a + 6), 
\end{align}
\begin{align}\label{111112}
\inf\big\{\mathcal{E}_\eps\big(u,Q_{2\rho}(0)\big)\colon \, u\in PR_{L}(Q_{2\rho}(0)) \colon \,  |\nabla u | \le \delta_\eps^{-\alpha/4}, \    u=u_{0, e_1, e_2 }\:\: \text{near}\:\: \partial Q_{2\rho}(0)\big\} \ge 30\rho a.
\end{align}
As $2\rho(13 a + 6)  < 30\rho a$ for $a$ large enough, this will indeed imply the desired statement.

We start with \eqref{11111}. For notational convenience, we assume that for $\varepsilon>0$ there exists $n=n({\rho,\varepsilon})\in \mathbb{N}$ such that $\rho=n\varepsilon$. Consider the strip $S_\eps  = (-\rho + \eps, \rho - \eps) \times (-\eps/8,  \eps/8)  $ which we  partition into  $N_{\rho, \eps} = (2\rho-2\eps)/(\eps/4) = 8\rho/\eps - 8$ many squares $(Q_{\eps/4}(x_i))_{i=1}^{N_{\rho, \eps}}$. We define
$$u(x) = \begin{cases}     M_\eps   (x-  x_i +  \frac{\eps}{8}e_2),  & \text{if } x \in Q_{\eps/4}(x_i) \text{ for } i=1,\ldots, N_{\rho, \eps}, \\ u_{0,e_1,e_2}  &  \text{otherwise in}\:\: Q_{2\rho}(0),      \end{cases} $$
with $M_\eps  :=  \frac{4}{\eps} (e_1  \otimes e_2 -   e_2 \otimes  e_1) $.  We note that $J_u$ splits into a vertical part with normal $e_1$, denoted by  $J_u^{\rm vert}$, and a horizontal part  with normal $e_2$, denoted by  $J_u^{\rm hor}$. The latter splits into a large part related to the boundary of the squares and two segments $\Gamma^\eps_- =(-\rho,\rho+\eps) \times \lbrace 0 \rbrace$ and $\Gamma^\eps_+ =(\rho-\eps,\rho) \times \lbrace 0 \rbrace$. Let $u_i \defas u \cdot e_i$ for $i=1,2$.  The definition of $u$ implies $[u_1] = 0$ and $|[u_2]| \le 1$  on  $J_u^{\rm hor} \setminus (\Gamma^\eps_- \cup \Gamma^\eps_+)$ and  $|[u]| \le 1$  on  $J_u^{\rm vert}$.
Therefore,    from \eqref{densitycounterexample}   we get
$$ \int_{J^{\rm vert}_u}{f(\tfrac{x}{\eps}, [u](x),\nu_{u}(x))\,\mathrm{d}\mathcal{H}^{1}(x)}  \le  (N_{\rho, \eps} +1) \tfrac{\eps}{4} (a+6)  = (2\rho-\tfrac{7}{4}\eps)(a+ 6), $$ 
$$ \int_{J^{\rm hor}_u}{f(\tfrac{x}{\eps}, [u](x),\nu_{u}(x))\,\mathrm{d}\mathcal{H}^{1}(x)}  \le a^3 \mathcal{L}^1(\Gamma^\eps_- \cup \Gamma^\eps_+)+  N_{\rho, \eps} 2 \tfrac{\eps}{4} 6a  =  2\eps  a^3  +  (4\rho - 4\eps)6 a.    $$
Combination of both estimates yields \eqref{11111}.

We now move on to \eqref{111112}.   Let $u$ be a general competitor for  the problem in \eqref{111112} and represent $u$ as in Definition \ref{def: PR} with respect to a Caccioppoli partition  $\lbrace P^+ \rbrace \cup \lbrace P^- \rbrace  \cup  (P_j)_j$   such that $\partial P^\pm$ contains $\partial Q_{2\rho}(0) \cap \lbrace \pm x_2 > 0 \rbrace$, respectively.  Note that all quantities depend on $\eps$ which we do not include in the notation for simplicity. In order to show  inequality  \eqref{111112}, by using Theorem \ref{prop5.3}, we can assume without restriction that the partition consists of a finite number of sets with polyhedral boundary. Moreover, up to considering the connected components (not relabeled), we can assume that all sets $(P_j)_j$ are connected. We denote the corresponding matrices by $(M_j)_j \subset \R^{2 \times 2}_{\rm skew}$. Moreover, by  $h_j$ we denote the length of the orthogonal projection of each $P_j$ onto the $e_2$-axis.

We will use slicing properties of $BV$-functions. To this end, we introduce some further notation. For each $s \in (-\rho,\rho)$, we let $J^s_u := \lbrace t \in (-\rho,\rho) \colon (s,t) \in J_u \rbrace$. We partition the interval $ (-\rho,\rho)$ into the sets $ \mathcal{I}_1 \cup  \mathcal{I}_2 \cup \mathcal{I}_3$, where $\mathcal{I}_i := \lbrace s \in (-\rho,\rho) \colon \# J_u^s = i \rbrace$ for $i=1,2$ and $\mathcal{I}_3:= (-\rho,\rho) \setminus (\mathcal{I}_1 \cup \mathcal{I}_2)$. Recalling the definition  of the density $f$ in \eqref{densitycounterexample}, by slicing properties of $BV$-functions we find
\begin{align*}
\mathcal{E}_\eps\big(u,Q_{2\rho}(0) \big)  \ge  \sum_{ j \in \mathcal{J}}  \frac{\eps}{2}  \Big( \Big\lfloor  \frac{1}{\eps} h_j  \Big\rfloor  - 1 \Big)    a^3 + \sum_{i=1}^3 \int_{\mathcal{I}_i}    \Big( \int_{J_u^s} a\min\lbrace 5  +  a|[u_1](s,t)|  + |[u_2](s,t)|,   a^2 \rbrace   \,    {\rm d}t  \Big)  \,    {\rm d}s,
\end{align*}
where $\lfloor \frac{1}{\varepsilon}h_j \rfloor$ denotes the integer part of $\frac{1}{\varepsilon}h_j $,  and $\mathcal{J}\defas \lbrace j \ge 1\colon h_j \ge 2\eps\rbrace$.   Here, for the vertical part we have used the periodicity of $f$ and the fact  that the projection of each $P_j$ onto the $x_2$-axis contains the projection of at least $\lfloor  \frac{1}{\eps} h_j  \rfloor-1 $ squares of size $\varepsilon$ with centers in $\varepsilon\mathbb{Z}^d$. In turn, this implies that for each $P_j$ the length of the boundary where the density $f(\frac{x}{\eps},\cdot, \cdot)$ coincides with $a^3$ is at least $ \frac{\eps}{2}  (\lfloor  \frac{1}{\eps} h_j  \rfloor   -1) $. On $\mathcal{I}_1$, the jump   satisfies $[u] = e_1$ since on such a slice $u$ necessarily coincides with $u_{0,e_1,e_2}$. This along with  $\# J_u^s \ge 3$ for all $s \in \mathcal{I}_3$ shows  that for $a$ large enough (such that $a(5+a) \ge 15a$) it holds that  \EEE
\begin{align}\label{final1}
\mathcal{E}_\eps\big(u, Q_{2\rho}(0) \big) & \ge  \sum_{ j \in \mathcal{J}}  \frac{\eps}{2} \Big(\Big\lfloor  \frac{1}{\eps} h_j  \Big\rfloor  -1\Big)  a^3   + \mathcal{L}^1(\mathcal{I}_1 \cup \mathcal{I}_3)15a \notag \\ 
&  \ \ \ \ +   \int_{\mathcal{I}_2}                \Big( \int_{J_u^s} a\min\lbrace 5  +  a|[u_1](s,t)|  + |[u_2](s,t)|,   a^2 \rbrace   \,    {\rm d}t  \Big)  \,    {\rm d}s.
\end{align}
We now address the integral over $ \mathcal{I}_2$. We note that for each  $s \in \mathcal{I}_2$ there exists a unique $P_j$  such that exactly the three components $P^+$, $P^-$, and $P_j$ intersect the slice $\lbrace s \rbrace \times (-\rho,\rho)$.   In fact, because of the boundary datum, there are at least two jump points for $u_1$ related to each   $P_j$.   Conversely, for each $P_j$, we denote by $\mathcal{S}_j  \subset \mathcal{I}_2$ the slices intersecting $P_j$.  Note that the sets $(\mathcal{S}_j )_j$ are pairwise disjoint. If  $|  M_j|_\infty \EEE\le \frac{1}{2h_j}$   we have, still because of the boundary condition,   $\sum_{t \in  J_u^s} [u_1(s,t)] \ge \frac{1}{2}$ and thus we   find  
\begin{align}\label{final2}
\int_{\mathcal{S}_j} \Big( \int_{J_u^s} a\min\lbrace 5  +  a|[u_1](s,t)|  + |[u_2](s,t)|,   a^2 \rbrace   \,    {\rm d}t  \Big)  \,    {\rm d}s \ge  \frac{a^2}{2} \mathcal{L}^1(\mathcal{S}_j) .
\end{align}
Assume instead $|M_j \EEE|_\infty > \frac{1}{2h_j}$. Note that by assumption $\eps|\nabla u| \le \eps\delta_\eps^{-\alpha/4} \to 0$ for $\eps \to 0$. Therefore,   for $\eps$ small enough  $|M_j |_\infty > \frac{1}{2h_j}$  implies  $\frac{\varepsilon}{2h_j}< \varepsilon \vert M_j \vert_{\infty} \leq\frac{1}{10}$ . Therefore, $h_j \ge 5\eps$ and    $\frac{\eps}{2} ( \lfloor   h_j / \eps \rfloor -1) \ge \frac{1}{4} h_j $. Then, if $\mathcal{L}^1(\mathcal{S}_j)\le 4 a h_j $, we get
\begin{align}\label{final3}
  \frac{\eps}{2}  \Big( \Big\lfloor  \frac{1}{\eps} h_j  \Big\rfloor -1 \Big)   a^3 +  \int_{\mathcal{S}_j}              \Big( \int_{J_u^s} a\min\lbrace 5  +  a|[u_1](s,t)|  + |[u_2](s,t)|,   a^2 \rbrace   \,    {\rm d}t  \Big)  \,    {\rm d}s  \ge \frac{1}{4} h_ja^3 \ge \frac{1}{16}  a^2\mathcal{L}^1(\mathcal{S}_j).
  \end{align}
Eventually, we treat the case $| M_j|_\infty > \frac{1}{2h_j}$ and $\mathcal{L}^1(\mathcal{S}_j)> 4 a h_j$, i.e.\ $|M_j|_\infty \mathcal{L}^1(\mathcal{S}_j) >   2a$. Note that, since $M_j$ is a $2 \times 2$ skew-symmetric matrix,   $u_2$ is constant on $P_j \cap \lbrace x_1 = s \rbrace$ for all $s \in \mathcal{S}_j$ and that the value depends linearly on $s$ with slope $|M_j|_\infty$. Thus, one can show that there exists  a subset $\mathcal{S}^*_j \subset \mathcal{S}_j$ with $\mathcal{L}^1(\mathcal{S}^*_j) \ge \frac{1}{4} \mathcal{L}^1(\mathcal{S}_j)$  such that 
$u_2 \ge  | M_j  |_\infty \frac{\mathcal{L}^1(\mathcal{S}_j)}{4}  > \frac{a}{2} $ on $P_j  \cap \lbrace x_1 = s \rbrace$ for all $s \in \mathcal{S}^*_j$. This shows that for each  $s \in \mathcal{S}^*_j$ there exists $t \in J^s_u$ such that $[u_2](s,t) \ge a/2$.  This yields 
\begin{align}\label{final4}
  \int_{\mathcal{S}_j}              \Big( \int_{J_u^s} a\min\lbrace 5  +  a|[u_1](s,t)|  + |[u_2](s,t)|,   a^2 \rbrace   \,    {\rm d}t  \Big)  \,    {\rm d}s  \ge \frac{a^2}{ 2 } \mathcal{L}^1(\mathcal{S}^*_j)   \ge \frac{a^2}{8} \mathcal{L}^1(\mathcal{S}_j) .
    \end{align}
Combining all estimates  \eqref{final1}--\eqref{final4} we find
$$\mathcal{E}_\eps\big(u,Q_{2\rho}(0)\big)  \ge  \frac{a^2}{16} \mathcal{L}^1(\mathcal{I}_2) + 15 a \mathcal{L}^1(\mathcal{I}_1   \cup \mathcal{I}_3).  $$
For $a$ large enough, optimizing this expression and recalling $\mathcal{L}^1(\mathcal{I}_1 \cup  \mathcal{I}_2 \cup  \mathcal{I}_3  ) = 2\rho$ leads to  \eqref{111112}.

\addcontentsline{toc}{chapter}{Bibliography}
\bibliographystyle{plain}
\bibliography{bibliography.bib}

\begin{thebibliography}{10}

\bibitem{Krengel1981}
M.A. Akcoglu and U.~Krengel.
\newblock Ergodic theorems for superadditive processes.
\newblock {\em Journal für die reine und angewandte Mathematik}, 323:53--67,
  1981.

\bibitem{alicicglo09}
R.~Alicandro, M.~Cicalese, and A.~Gloria.
\newblock Integral representation results for energies defined on stochastic
  lattices and application to nonlinear elasticity.
\newblock {\em Archive for Rational Mechanics and Analysis}, 20:881--943, 2011.

\bibitem{aliprantis06}
C.D. Aliprantis and K.C. Border.
\newblock Infinite dimensional analysis.
\newblock {\em Springer}, 2006.

\bibitem{ambrosio2000fbv}
L.~Ambrosio, N.~Fusco, and D.~Pallara.
\newblock {Functions of bounded variation and free discontinuity problems}.
\newblock {\em Oxford Mathematical Monographs}, 2000.

\bibitem{BacRuf21}
A.~Bach and M.~Ruf.
\newblock Fluctuation estimates for the multi-cell formula in stochastic
  homogenization of partitions.
\newblock {\em Calculus of Variations and Partial Differential Equations},
  61:84, 2022.

\bibitem{hom2}
M.~Barchiesi and G.~Dal~Maso.
\newblock \text{Homogenization of fiber reinforced brittle materials: The
  extremal cases}.
\newblock {\em SIAM Journal on Mathematical Analysis}, 41:1874–1889, 2009.

\bibitem{BarchiesiFocardi}
M.~Barchiesi and M.~Focardi.
\newblock Homogenization of the \text{Neumann} problem in perforated domains:
  An alternative approach.
\newblock {\em Calculus of Variations and Partial Differential Equations},
  42:257--288, 2012.

\bibitem{hom1}
M.~Barchiesi, G.~Lazzaroni, and C.I. Zeppieri.
\newblock \text{A bridging mechanism in the homogenization of brittle}
  \text{composites} with \text{soft inclusions}.
\newblock {\em SIAM Journal on Mathematical Analysis}, 48:1178--1209, 2015.

\bibitem{atwoscale}
L.~Berlyand, E.~Sandier, and S.~Serfaty.
\newblock A two scale \text{$\Gamma$}-convergence approach for random
  non-convex homogenization.
\newblock {\em Calculus of Variations and Partial Differential Equations},
  56:156, 2016.

\bibitem{BFLM}
G.~Bouchitte, I.~Fonseca, G.~Leoni, and M.~Mascarenhas.
\newblock \text{A global method for relaxation in $W^{1,p}$ and in $SBV^p$}.
\newblock {\em Archive for Rational Mechanics and Analysis}, 165:187--242,
  2002.

\bibitem{Bouchitte}
G.~Bouchitt\'e, I.~Fragalá, and R.~Mahadevan.
\newblock Homogenization of second order energies on periodic thin structures.
\newblock {\em Calculus of Variations and Partial Differential Equations},
  20:175--211, 2004.

\bibitem{Braides:02}
A.~Braides.
\newblock \text{Gamma-Convergence for Beginners}.
\newblock {\em Oxford University Press}, 2002.

\bibitem{BraConGar16}
A.~Braides, S.~Conti, and A.~Garroni.
\newblock Density of polyhedral partitions.
\newblock {\em Calculus of Variations and Partial Differential Equations},
  56:1--10, 2017.

\bibitem{BraDefVit96}
A.~Braides, A.~Defranceschi, and E.~Vitali.
\newblock Homogenization of \text{free-discontinuity problems}.
\newblock {\em Archive for Rational Mechanics and Analysis}, 135:297--356,
  1996.

\bibitem{bragel02}
A.~Braides and M.S. Gelli.
\newblock Limits of discrete systems with long-range interactions.
\newblock {\em Journal of Convex Analysis}, 9:363--399, 2002.

\bibitem{BraPia19a}
A.~Braides and A.~Piatnitski.
\newblock Homogenization of quadratic convolution energies in periodically
  perforated domains.
\newblock {\em Advances in Calculus of Variations}, 15:351--368, 2022.

\bibitem{brasig09}
A.~Braides and L.~Sigalotti.
\newblock Models of defects in atomistic systems.
\newblock {\em Calculus of Variatons and Partial Differential Equations},
  41:71--109, 2011.

\bibitem{truski}
A.~Braides and L.~Truskinovsky.
\newblock \text{Asymptotic expansions by $\Gamma$-convergence}.
\newblock {\em Continuum Mechanics and Thermodynamics}, 20:21--62, 2008.

\bibitem{cagnetti2018gammaconvergence}
F.~Cagnetti, G.~Dal Maso, L.~Scardia, and C.I. Zeppieri.
\newblock \text{$\Gamma$}-convergence of free-discontinuity problems.
\newblock {\em Annales de l'Institut Henri Poincaré C, Analyse non linéaire},
  36:1035--1079, 2018.

\bibitem{cagnetti2017stochastic}
F.~Cagnetti, G.~Dal Maso, L.~Scardia, and C.I. Zeppieri.
\newblock Stochastic homogenisation of free-discontinuity problems.
\newblock {\em Archive for Rational Mechanics and Analysis}, 233:935--974,
  2019.

\bibitem{global}
F.~Cagnetti, L.~Scardia, and C.I. Zeppieri.
\newblock A global method for deterministic and stochastic homogenisation in
  \text{$BV$}.
\newblock {\em Annals of PDE}, 8:8, 2022.

\bibitem{Chambolle:2003}
A.~Chambolle.
\newblock \text{A density result in two-dimensional linearized elasticity, and
  applications}.
\newblock {\em Archive for Rational Mechanics and Analysis}, 167:211--233,
  2003.

\bibitem{Chambolle2006PIECEWISERA}
A.~Chambolle, A.~Giacomini, and M.~Ponsiglione.
\newblock Piecewise rigidity.
\newblock {\em Journal of Functional Analysis}, 244:134--153, 2006.

\bibitem{CicFocZep21}
M.~Cicalese, M.~Focardi, and C.I. Zeppieri.
\newblock Phase-field approximation of functionals defined on piecewise-rigid
  maps.
\newblock {\em Journal of Nonlinear Science}, 31:78, 2021.

\bibitem{ConFocIur15}
S.~Conti, M.~Focardi, and F.~Iurlano.
\newblock Which special functions of bounded deformation have bounded
  variation?
\newblock {\em Proceedings of the Royal Society of Edinburgh}, 148A:33--50,
  2016.

\bibitem{Cristoferi}
R.~Cristoferi, I.~Fonseca, A.~Hagerty, and C.~Popovici.
\newblock A homogenization result in the gradient theory of phase transitions.
\newblock {\em Interfaces and Free Boundaries}, 21:367--408, 2019.

\bibitem{DalMaso:93}
G.~Dal~Maso.
\newblock \text{An Introduction to $\Gamma$-Convergence}.
\newblock {\em Progress in Nonlinear Differential Equations and Their
  Applications, Birkh{\"a}user Boston}, 2012.

\bibitem{Dal11}
G.~Dal~Maso.
\newblock Generalised functions of bounded deformation.
\newblock {\em Journal of the European Mathematical Society (JEMS)},
  15:1943--1997, 2013.

\bibitem{dalfratoa04}
G.~Dal~Maso, G.~Francfort, and R.~Toader.
\newblock Quasistatic crack growth in nonlinear elasticity.
\newblock {\em Archive for Rational Mechanics and Analysis}, 176:165--225,
  2005.

\bibitem{DalMaso-Lazzaroni:2010}
G.~Dal~Maso and G.~Lazzaroni.
\newblock Quasistatic crack growth in finite elasticity with
  non-interpenetration.
\newblock {\em Annales de l'Institut Henri Poincaré C, Analyse non linéaire},
  2010.

\bibitem{maso1984nonlinear}
G.~Dal~Maso and L.~Modica.
\newblock Nonlinear stochastic homogenization.
\newblock {\em Universit{\`a} di Pisa. Dipartimento di Matematica}, 1984.

\bibitem{dalnegper01}
G.~Dal~Maso, M.~Negri, and D.~Percivale.
\newblock \text{Linearized elasticity as $\Gamma\mbox{-limit}$ of finite
  elasticity}.
\newblock {\em Set-Valued Analysis}, 10:165--183, 2002.

\bibitem{hom3}
G.~Dal~Maso and C.I. Zeppieri.
\newblock Homogenization of fiber reinforced brittle materials: The
  intermediate case.
\newblock {\em Advances in Calculus of Variations}, 3:345--370, 2010.

\bibitem{DeGiorgi}
E.~De~Giorgi and L.~Ambrosio.
\newblock A new type of functional in the calculus of variations.
\newblock {\em Atti della Accademia Nazionale dei Lincei. Serie Ottava.
  Rendiconti. Classe di Scienze Fisiche, Matematiche e Naturali}, 82:199--210,
  1988.
\newblock (\textbf{Italian}).

\bibitem{meyer1966probability}
C.~Dellacherie and P.A. Meyer.
\newblock Probability and potentials.
\newblock {\em North-Holland Publishing Company}, 1979.

\bibitem{Donnarumma}
A.~Donnarumma and M.~Friedrich.
\newblock Stochastic homogenisation for free-discontinuity problems in
  linearised elasticity.
\newblock (In preparation).

\bibitem{Francfort-Larsen:2003}
G.~Francfort and C.~Larsen.
\newblock Existence and convergence for quasi-static evolution in brittle
  fracture.
\newblock {\em Communications on Pure and Applied Mathematics}, 56:1465--500,
  2003.

\bibitem{FrancfortMarigo}
G.~Francfort and J.-J. Marigo.
\newblock Revisiting brittle fracture as an energy minimization problem.
\newblock {\em Journal of the Mechanics and Physics of Solids}, 46:1319--1342,
  1998.

\bibitem{friedrich2017derivation}
M.~Friedrich.
\newblock \text{A derivation of linearized Griffith energies from nonlinear
  models}.
\newblock {\em Archive for Rational Mechanics and Analysis}, 225:425--467,
  2017.

\bibitem{friedrich2018piecewise}
M.~Friedrich.
\newblock A piecewise \text{Korn} inequality in \text{SBD} and applications to
  embedding and density results.
\newblock {\em SIAM Journal on Mathematical Analysis}, 50:3842--3918, 2018.

\bibitem{friedrich2019griffith}
M.~Friedrich.
\newblock Griffith energies as small strain limit of nonlinear models for
  nonsimple brittle materials.
\newblock {\em Mathematics in Engineering}, 2:75--100, 2019.

\bibitem{friedrichcompactness}
M.~Friedrich.
\newblock \text{A compactness result in $GSBV^p$ and applications to
  $\Gamma$-convergence} for free discontinuity problems.
\newblock {\em Calculus of Variations and Partial Differential Equations},
  58:86, 2019.

\bibitem{FriPerSol20}
M.~Friedrich, M.~Perugini, and F.~Solombrino.
\newblock Lower semicontinuity for functionals defined on piecewise rigid
  functions and on \text{$GSBD$}.
\newblock {\em Journal of Functional Analysis}, 280(7):108929, 2021.

\bibitem{FriPerSol20a}
M.~Friedrich, M.~Perugini, and F.~Solombrino.
\newblock \text{$\Gamma$-convergence for free-discontinuity problems in linear
  elasticity}: Homogenization and relaxation.
\newblock {\em Indiana University Mathematics Journal}, 2023.

\bibitem{frie14}
M.~Friedrich and B.~Schmidt.
\newblock An analysis of crystal cleavage in the passage from atomistic models
  to continuum theory.
\newblock {\em Archive for Rational Mechanics and Analysis}, 217:263--308,
  2015.

\bibitem{FriSol16}
M.~Friedrich and F.~Solombrino.
\newblock \text{Quasistatic crack growth in 2d-linearized elasticity}.
\newblock {\em Annales de l'Institut Henri Poincaré C, Analyse non linéaire},
  35:28--64, 2018.

\bibitem{Friedrich_2020}
M.~Friedrich and F.~Solombrino.
\newblock Functionals defined on piecewise rigid functions: Integral
  representation and \text{$\Gamma$}-convergence.
\newblock {\em Archive for Rational Mechanics and Analysis}, 236:1325--1387,
  2020.

\bibitem{giapon04a}
A.~Giacomini and M.~Ponsiglione.
\newblock \text{A $\Gamma$-convergence} \text{approach} to \text{stability} of
  \text{unilateral minimality properties} in \text{fracture mechanics} and
  \text{applications}.
\newblock {\em Archive for Rational Mechanics and Analysis}, 180(3):399--447,
  2006.

\bibitem{Neukamm2}
A.~Gloria and S.~Neukamm.
\newblock Commutability of homogenization and linearization at identity in
  finite elasticity and applications.
\newblock {\em Annales de l'Institut Henri Poincaré C, Analyse non linéaire},
  28:941--964, 2011.

\bibitem{griffith1920phenomena}
A.A. Griffith.
\newblock The phenomena of rupture and flow in solids.
\newblock {\em Philosophical transactions / Royal Society of London},
  221:163--198, 1921.

\bibitem{Hammersley1965FirstPassagePS}
J.M. Hammersley and D.J.A. Welsh.
\newblock \text{First-Passage Percolation, Subadditive Processes, Stochastic
  Networks}, and \text{Generalized Renewal Theory}.
\newblock {\em Bernoulli-Bayes-Laplace Anniversary Volume}, 1965.

\bibitem{Jesenko}
M.~Jesenko and B.~Schmidt.
\newblock \text{Closure }and \text{commutability} \text{results} for
  \text{Gamma-limits} and the \text{geometric} \text{linearization} and
  \text{homogenization} of \text{multi-well energy functionals}.
\newblock {\em SIAM Journal on Mathematical Analysis}, 46:2525--2553, 2014.

\bibitem{krengel2011ergodic}
U.~Krengel.
\newblock Ergodic theorems.
\newblock {\em De Gruyter}, 1985.

\bibitem{Neukamm1}
S.~Müller and S.~Neukamm.
\newblock \text{On the commutability of homogenization and linearization in
  finite elasticity}.
\newblock {\em Archive for Rational Mechanics and Analysis}, 201:465--500,
  2011.

\bibitem{Pellet}
X.~Pellet, L.~Scardia, and C.I. Zeppieri.
\newblock Stochastic homogenisation of free-discontinuity functionals in random
  perforated domains.
\newblock {\em Advances in Calculus of Variations}, 2023.
\newblock (accepted).

\bibitem{RufZeppieri}
M.~Ruf and C.I. Zeppieri.
\newblock Stochastic homogenization of degenerate integral functionals with
  linear growth.
\newblock {\em Calculus of Variations and Partial Differential Equations},
  62:138, 2023.

\bibitem{Zemas}
L.~Scardia, K.~Zemas, and C.I. Zeppieri.
\newblock Homogenisation of nonlinear \text{Dirichlet} problems in randomly
  perforated domains under minimal assumptions on the size of perforations.
\newblock 2023.
\newblock (preprint).

\bibitem{Temam1980FunctionsOB}
R.~Temam and G.~Strang.
\newblock Functions of bounded deformation.
\newblock {\em Archive for Rational Mechanics and Analysis}, 75:7--21, 1980.

\bibitem{Toupin1962ElasticMW}
R.A. Toupin.
\newblock Elastic materials with couple-stresses.
\newblock {\em Archive for Rational Mechanics and Analysis}, 11:385--414, 1962.

\bibitem{Toupin1962ElasticMW2}
R.A. Toupin.
\newblock Theories of elasticity with couple-stress.
\newblock {\em Archive for Rational Mechanics and Analysis}, 17:85--112, 1964.

\bibitem{williams_1991}
D.~Williams.
\newblock Probability with martingales.
\newblock {\em Cambridge University Press}, 1991.

\end{thebibliography}

\end{document}